\documentclass[a4paper,11pt]{amsart}

\usepackage{geometry} 
\usepackage{amssymb}  
\usepackage{mathtools}      
\usepackage{mathabx}        
\usepackage[bb=fourier,cal=euler,scr=rsfs]{mathalfa}	
\usepackage{enumitem}       
\usepackage[colorlinks=true]{hyperref} 

\usepackage[ansinew]{inputenc}

\hypersetup{bookmarksdepth = 3} 
\numberwithin{equation}{section}     

\setlist[enumerate,1]{label={\upshape(\roman*)},ref=\roman*}
\setlist[enumerate,2]{label={\upshape(\alph*)},ref=\alph*}

\newtheorem{theorem}{Theorem}[section]
\newtheorem{teo}[theorem]{Theorem}

\newtheorem{coro}[theorem]{Corollary}

\newtheorem{lema}[theorem]{Lemma}

\newtheorem{prop}[theorem]{Proposition}

\newtheorem{claim}[theorem]{Claim}

\newtheorem{question}[theorem]{Question}

\theoremstyle{definition}
\newtheorem{defi}[theorem]{Definition}
\newtheorem{remark}[theorem]{Remark}

\newtheorem{example}[theorem]{Example}

\newcommand{\eps}{\varepsilon}

\newcommand{\wt}[1]{\widetilde{#1}}
\newcommand{\R}{\mathbb R}
\newcommand{\RR}{\R}
\newcommand{\ZZ}{\mathbb Z}
\newcommand{\HH}{\mathbb{H}}
\newcommand{\DD}{\mathbb{D}}
\newcommand{\CC}{\mathbb{C}}

\newcommand{\wwp}{\widetilde \Phi}

\newcommand{\cI}{\mathcal{I}}
\newcommand{\cT}{\mathcal{T}}
\newcommand{\cW}{\mathcal{W}}
\newcommand{\cF}{\mathcal{F}}

\newcommand{\cG}{\mathcal{G}}
\newcommand{\cD}{\mathcal{D}}
\newcommand{\cB}{\mathcal{B}}

\newcommand{\wcF}{\widetilde \cF}

\newcommand{\wcG}{\widetilde \cG}

\newcommand{\cL}{\mathcal{L}}

\newcommand{\cC}{\mathcal{C}}
\newcommand{\cO}{\mathcal{O}}

\newcommand{\cM}{\mathcal{M}}

\newcommand{\cE}{\mathcal{E}}
\newcommand{\cV}{\mathcal{V}}
\newcommand{\cU}{\mathcal{U}}

\newcommand{\mt}{\widetilde M}
\newcommand{\ft}{\widetilde f}

\newcommand{\TT}{\mathbb{T}}

\newcommand{\cs}{\cW^{cs}}
\newcommand{\cu}{\cW^{cu}}

\newcommand{\wcs}{\widetilde \cW^{cs}}
\newcommand{\wcu}{\widetilde \cW^{cu}}

\newcommand{\cTu}{S^1_\infty(\wcF)}
\newcommand{\cTuu}{S^1_\infty(\wcF_1)}

\newcommand{\cTuD}{S^1_\infty(\widetilde{\cD})}

\title[Transverse foliations in 3-manifolds]{Intersection of transverse 
foliations in 3-manifolds \\ \footnotesize Hausdorff leafspace implies leafwise quasigeodesic}

\author{Sergio R.\ Fenley} 
\address{Florida State University, Tallahassee, FL 32306}
\email{fenley@math.fsu.edu}

\author{Rafael Potrie} 
\address{Centro de Matem\'atica, Universidad de la Rep\'ublica, Uruguay}
\email{rpotrie@cmat.edu.uy}
\urladdr{http://www.cmat.edu.uy/~rpotrie/}

\thanks{S.F. was partially supported by Simons Foundation grant 637554; by  National Science Foundation grant 
DMS-2054909, and by the Institute for Advanced Study.
 R. P. was partially supported by CSIC I+D project 'Estructuras Topol\'ogicas de sistemas parcialmente hiperb\'olicos y aplicaciones' and ANII}

\begin{document}

\begin{abstract}
Let $\cF_1$ and $\cF_2$ be transverse two dimensional foliations with Gromov hyperbolic leaves in a closed 3-manifold $M$ whose fundamental group is not solvable, and let $\cG$ be the one dimensional foliation obtained by intersection. We show that $\cG$ is \emph{leafwise quasigeodesic} in $\cF_1$ and $\cF_2$ if and only if the foliation $\cG_L$ induced by $\cG$ in the universal cover $L$ of any leaf of $\cF_1$ or $\cF_2$ has Hausdorff leaf space. We end up with a discussion on the hypothesis of Gromov hyperbolicity of the leaves. 

\bigskip

\end{abstract}

\maketitle

\section{Introduction} 

In this paper we study pairs of transverse Reebless  foliations in closed 3-manifolds and the geometric properties of the one dimensional foliation obtained by intersecting them. This problem goes back at least to \cite{ThurstonCircle1} and was explored for instance in \cite{MatsumotoTsuboi} where it is shown to be a quite subtle one.

We will analyze the situation when we have two transverse $2$-dimensional
foliations, which will be denoted by  $\cF_1, \cF_2$, on a $3$-manifold $M$.
We will denote by $\cG$ 
the intersected  foliation. 
The focus of this article will be on geometric properties
of leaves of the one dimensional foliation $\cG$ inside
leaves of $\cF_1, \cF_2$ when lifted to the universal cover.
The study of similar geometric properties was used very successfully  
by Thurston, generating important results in $3$-manifolds:
for example Cannon and Thurston's construction \cite{Ca-Th}
of group invariant Peano curves (involving singular foliations on surfaces).
This geometric study is also extremely useful
for analyzing the continuous extension property for foliations
in hyperbolic $3$-manifolds \cite{FenleyGeom}.

A well known example of  the situation considered in this article
occurs when 
$\cF_1, \cF_2$ are the weak stable and weak unstable foliations
of an Anosov flow $\Phi$ in the $3$-manifold $M$.
A big motivation for the study in this article comes from partial
hyperbolic diffeomorphisms in $3$-manifolds, which has been an active area in the last 20 years. Under orientability
conditions there is pair of branching two dimensional foliations,
preserved by the partially hyperbolic diffeomorphism, which
are transverse to each other. The intersection is a branching one dimensional
foliation. 
The generic case is that the branching two dimensional foliations
have Gromov hyperbolic leaves.
Again under orientability conditions, the two dimensional branching
foliations are well approximated by actual foliations, which
are transverse to each other, and whose intersection is a one dimensional
foliation. The study of the geometric structure of these 
one dimensional foliations is very useful
as follows. Under the orientability condition, the 
one dimensional foliation generates a flow. If one
supposes that the flow lines are uniform quasigeodesics
in the respective leaves of $\wcF_1$ or $\wcF_2$, 
we obtained some results in \cite{BFP} (see also \cite{ChandaFenley}) that allowed us to promote leafwise quasigeodesic flows to (topological) Anosov flows under certain situations.
This analysis was in turn motivated by the study of a particular class of transverse foliations arising from partially hyperbolic systems in \cite{FP2}.  This induces a lot of structure on the partially 
hyperbolic diffeomorphism, and this structure has
very important consequences, such as accessibility
and ergodicity of the system (if volume preserving) \cite{FPAcc}.

This naturally lead to the following more general problem: suppose
$\cF_1, \cF_2$ are transverse foliations by Gromov hyperbolic
leaves in a $3$-manifold $M$. Let $\cG$ be the intersected  foliation.
When are the leaves of $\cG$ leafwise quasigeodesic?
In \cite[\S 1.1]{FP4} we delineated a very careful strategy to attack this problem.
A very easy necessary condition is that in the universal cover,
in each leaf $L$ of $\wcF_1, \wcF_2$, the one dimensional foliation
$\wcG$ in $L$ (denoted by $\cG_L$)
has Hausdorff leaf space (hence homeomorphic to the reals).
In fact proving such Hausdorff behavior is one of the intermediate
steps in the strategy to prove leafwise quasigeodesic behavior
\cite{FP4}.
 A natural question is how strong is the property of
having leafwise Hausdorff leaf space: for example
is it equivalent to being a foliation by uniform quasigeodesics inside the leaves of each of the foliations? This is what we analyze in this article
and we prove the following:

\begin{teo}\label{teo.main} Let $\cF_1, \cF_2$ be two transverse foliations by Gromov hyperbolic leaves in a closed 3-manifold $M$ whose fundamental group is not solvable, and let $\cG$ be the intersected  foliation. Let $\wcG$ be the lifted foliation to $\mt$ and given $L \in \wcF_i$ denote by $\cG_L$ the restriction of $\wcG$ to $L$.  Then leaves of $\cG_L$ are uniformly quasigeodesic in $L$ for all $L \in \wcF_1, \wcF_2$ if and only if the leaf space $\cO_L$ of $\cG_L$ is Hausdorff for all $L \in \wcF_1,\wcF_2$. 
\end{teo}

At first sight this is a very surprising result: Hausdorff leaf
space is only a topological property, which could be true
in many situations. On the other hand quasigeodesic behavior
is a very strong geometric property with many important 
consequences.

We mention that the 
the study of transverse foliations has been addressed 
before by Thurston \cite[Section 7]{ThurstonCircle1},
and by work of Hardorp \cite{Hardop} on total foliations (we also point to our previous paper \cite{FP4} where the general problem of 
transverse foliations is discussed).

Regarding the assumption that the leaf spaces of the one dimensional foliations in their two dimensional leaves is Hausdorff for all leaves, we remark that this is equivalent to the fact that the leaf space of the one dimensional foliation $\wcG$ in $\mt$ is Hausdorff as we show in Proposition \ref{prop-Hsdff2Dand3D}. This generalizes a result from \cite{Barbot,FenleyAnosov} where it is proved for Anosov one dimensional foliations. We also point out that as a part of our study we get the following consequence which may be interesting on its own (see Theorem \ref{prop-closedleaf}):

\begin{coro}\label{quasigeodesiccoro-periodic}
 Let $\cF_1, \cF_2$ be two transverse foliations by Gromov hyperbolic leaves in a closed 3-manifold $M$ whose fundamental group is not solvable and let $\cG$ be the intersected  foliation. 
Suppose that $\wcG$ has Hausdorff leaf space.
Then $\cG$ has closed leaves. 
\end{coro} 

The hypothesis of having Gromov hyperbolic leaves is natural for several reasons. A result by Candel (see \cite[\S 7]{Calegari-book}) implies that it is 
in some sense the most common situation: Candel's theorem states that if a Reebless foliation in a 3-manifold with non-solvable fundamental group does not to have this property, then the foliation has a transverse invariant measure approximated by some incompressible tori. So even in toroidal manifolds, having
Gromov  hyperbolic leaves is very abundant. 


In addition the applications we have in mind (for instance, for partially hyperbolic diffeomorphisms in
non solvable manifolds, or for Anosov flows) provide such structure as a given. Finally the quasigeodesic property is particularly relevant in the case of Gromov hyperbolic leaves, as one disposes of tools such as the \emph{Morse lemma} (see \cite{GhysdelaHarpe}) which gives particular relevance to quasigeodesics. We point out in particular that in \cite[\S 6]{BFP} we show that in the context of partially hyperbolic dynamics, if two dynamical foliations intersect in a leafwise quasigeodesic (branching) foliation, then the partially hyperbolic diffeomorphism is what we call a \emph{collapsed Anosov flow}. Nevertheless, in \S \ref{s.nonhyperbolicleaves} we explore the case where the foliations do not have Gromov hyperbolic leaves and prove some results, as well as state some questions. 

We stress that the analysis of this paper deals with the 
quasigeodesic properties of leaves of $\wcG$ as seen inside the leaves of
$\wcF_1, \wcF_2$. In general the leaves of $\wcG$ will not
be quasigeodesics in $\mt$. For example, 
when $\cG$ is the intersection of the weak stable and
weak unstable foliations of an Anosov flow which is
${\mathbb{R}}$-covered and $M$ is hyperbolic, then leaves of $\wt{\cG}$ are known not to be quasigeodesics in $\mt \cong \mathbb{H}^3$ as shown in \cite{FenleyAnosov}.

%

\subsection{Organization of the paper}\label{ss.steps} 
The proof of the main theorem will be achieved in several steps based on a
detailed strategy outlined in \cite[\S 1.1]{FP4}:

$-$ {\bf {Landing}} $-$ The first step is to show that rays of the intersected  foliation ($\wcG$)  \emph{land} in their respective circles at infinity in the sense of Definition \ref{def.land}. 

$-$  {\bf {Small visual measure}} $-$ Roughly this says that  arcs,
rays, or full leaves  of $\wcG$ which are  far in $L$ ($L \in \wcF_i$)
from a point $x$ in $L$,  have small visual measure as  seen  from $x$.
The obvious  counterexample are horocycle rays in hyperbolic leaves.

$-$ {\bf {No bubble leaves}} $-$ Show  that  if $c$ is a leaf of 
$\wcG$ in a leaf   $L  \in \wcF_i$, then the two  rays of $c$
do not land in the same point  in $S^1(L)$.

$-$ {\bf {Hausdorff}} $-$  Show that $\cG_L$ has
Hausdorff leaf   space in any leaf of $\wcF_i$.

$-$ {\bf  {Quasigeodesic property}} $-$  Show that  leaves of $\cG_L$
are uniform quasigeodesics in  $L$  for any $L \in \wcF_i$.

In this article the 4th property above is the  overall hypothesis,
and we  show that it implies the quasigeodesic behavior  under
the  conditions of the main  theorem. But the key observation here is that assuming the 4th step, the other steps can be proved in a more direct way (compare with \cite{FP4}).

In \S \ref{s.genfol} we give needed background and some
preliminary results on foliations and transverse foliations. In \S \ref{s.folinfinity} we analyze some properties at
infinity of foliations by Gromov hyperbolic leaves. Most of section \S \ref{s.folinfinity} is well known
in the case that the leaves are negatively curved,
or hyperbolic. Here we give detailed proofs of some properties 
when the leaves are only Gromov hyperbolic.
This is necessary
since working with pairs of foliations requires finding a nice context that can be applied simultaneously. 
We note that \S \ref{s.folinfinity} can be skipped in a first read, or at least until  \S \ref{s.nonvisual}. 

The short \S \ref{s.pushthrough} shows a simple yet powerful consequence of the Hausdorff leaf space property that will be used several times in the paper. The first use is in \S \ref{s.landing} which shows a key property that rays of the intersected  foliation land in the corresponding leaves. This result does not use the full Hausdorff property, just being leafwise Hausdorff in one of the foliations. 

In \S \ref{s.example} we study an example that shows the importance of the hypothesis of having non-solvable fundamental group in the main result. Formally, this section is not needed, but understanding the example can shed light in the arguments that are presented later. 

Section \S \ref{s.SVMQG} studies general foliations by Gromov hyperbolic leaves subfoliated by one dimensional foliations and relates the quasigeodesic property to the small visual measure property that is defined there. 
The main result is that if the one dimensional foliation is leafwise Hausdorff, then the small visual measure property is enough to establish the leafwise quasigeodesic property. 

In \S \ref{s.nonvisual} we produce some structure from the failure of the small visual measure property. In particular, we show that if the small visual measure
property fails in, say the foliation $\cF_1$, then the foliation $\cF_1$ is up to collapsing, topologically conjugate  to  the weak stable foliation of a (topological) Anosov flow which is $\RR$-covered (these notions are introduced and explained in \S \ref{ss.Anosovflows}). Finally, in \S \ref{s.SVM} we show how this structure is enough to prove that if the visual measure property fails, then, the foliations should be similar to the ones presented in \S \ref{s.example},  and we show that the fundamental group has to be solvable. 

In \S \ref{s.nonhyperbolicleaves} the case where leaves are not all Gromov hyperbolic is studied. 
Finally in \S \ref{s.ph} we obtain the application result
to partially hyperbolic diffeomorphisms.

\medskip
{\small \emph{Acknowledgements:} The authors would like to thank Elena Gomes and Santiago Martinchich for helpful discussions that allowed to improve the presentation. We also thank Thomas Barthelme for an important suggestion and the referee whose thorough read was important in improving the paper.}

\section{Foliations}\label{s.genfol}
\label{s.foliation}

We will let $M$ be a closed 3-manifold. Let $\cF_1, \cF_2$ be two transverse foliations (cf. \S \ref{ss.fol}). We denote by $\cG$ the one dimensional foliation obtained by intersecting $\cF_1,\cF_2$, that is, $\cG = \cF_1 \cap \cF_2$. We consider in $\mt$, the universal cover of $M$ the lifted foliations $\wt{\cF_1}, \wt{\cF_2}$ and $\wt{\cG}$. For a leaf $L \in \wt{\cF_i}$ we denote by $\cG_L$ the one dimensional foliation of $L$ obtained as the restriction of $\wt{\cG}$ to $L$.  

We denote by $\cL_i$ the leaf space of $\wt{\cF_i}$ (i.e. the topological space obtained by the quotient of $\mt$ by the equivalence relation of being in the same leaf of $\wt{\cF_i}$. It is known that if $\cF_i$ is Reebless, then $\cL_i$ is a one dimensional, simply connected (but possibly non-Hausdorff) manifold. 

We let $\cO$ denote the leaf space of $\wt{\cG}$ and, if $L \in \wt{\cF_i}$ let $\cO_L$ be the leaf space of $\cG_L$. 

Since we will be mostly working in the universal cover and since our results are stable by taking finite lifts we can and will assume throughout that 
$M$ is orientable and both $\cF_1$ and $\cF_2$ are orientable and transversally orientable. 

Sections \S \ref{ss.markers}, \S \ref{ss.Anosovflows} and \S \ref{ss.uniformfol} are only used at the end of \S \ref{s.nonvisual}  and in \S \ref{s.SVM}. 

\subsection{Foliations}\label{ss.fol}


We will work with foliation of class $C^{0,1+}$. Recall that a foliation $\cF$ (of class $C^{0,1+}$) is a partition of $M$ by injectively immersed surfaces which are tangent to a two dimensional 
subbundle $E$ of $TM$, that is, if $S \in \cF$ is one such surface, then, at each $x \in S$ we have that $T_x S = E_x$ (this is equivalent to the usual definition using charts). The surfaces of $\cF$ are called \emph{leaves} of $\cF$. For $x \in M$ we denote by $\cF(x)$ to the leaf of $\cF$ containing $x$. We denote by $T\cF$ to the two dimensional subbundle $E$ of $TM$ which is tangent to the leaves of $\cF$. The standing assumption of orientability is equivalent to ask that each of the surfaces of $\cF$ or that the bundle $T\cF$ as well as $M$ are orientable. See  \cite{CandelConlon,Calegari-book}. The regularity assumptions are convenient, but not crucial (see Remark \ref{rem-reg}).

If $\cF_1$ and $\cF_2$ are two transverse foliations (i.e. the subbundles $T\cF_1$ and $T\cF_2$ are everywhere transverse) we get charts in $M$ where the leaves of one of the foliations are mapped into horizontal planes of the form $\RR^2 \times \{t\}$ and the other foliation to vertical planes of the form $\{s\}\times \RR^2$. By compactness, there is some value of $\eps_0>0$ so that every point in $M$ verifies that its $\eps_0$-neighborhood belongs to such a chart. We call such value of $\eps_0$ the size of \emph{local product structure}, and \emph{local product structure boxes} to these charts. This will be used throughout.

The foliations we will be considering are those which do not contain \emph{Reeb components}, that is, there is no solid torus in $M$ which is saturated by leaves of $\cF$,  so that only the boundary is a compact leaf of $\cF$ and the interior leaves are all planes. Such foliations are called \emph{Reebless}. We compile in the next statement results by Reeb, Novikov and Palmeira that give some implications of being Reebless.  

\begin{teo}\label{teo.Novikovetc} 
Let $\cF$ be a Reebless foliation on a closed 3-manifold without spherical leaves $M$, then: 
\begin{enumerate}
\item the fundamental group of every leaf of $\cF$ is injected in the fundamental group of $M$, in particular, if $\wcF$ denotes the lifted foliation in the universal cover $\mt$ of $M$ then every leaf is homeomorphic to the two dimensional plane. 
\item in the universal cover, given a curve $\gamma$ transverse to $\wcF$ we have that $\gamma$ can intersect each leaf of $\wcF$ at most once. 
\item the leaf space $\cL = \mt/_{\wcF}$ (i.e. the topological quotient of $\mt$ obtained by the equivalence relation given by being in the same leaf) is a one dimensional (possibly non-Hausdorff) simply connected manifold. 
\item the foliation $\wcF$ is homeomorphic to the product of a foliation of the plane with $\RR$ (i.e. to the foliation given by product of leaves of the foliation of the plane with $\RR$), in particular $\mt \cong \RR^3$. 
\end{enumerate}
\end{teo}

All of these  are  proved in volume 2 of  \cite{CandelConlon}: for
(i) see  \cite[Thm 9.1.3]{CandelConlon}.
Since there are  no sphere leaves, (i)  follows.  For  (ii) see \cite[Thm. 9.1.4]{CandelConlon}. Item (iii) follows immediately from (ii). For (iv) see \cite[Thm. 9.1.10]{CandelConlon}.

\begin{remark}
Note that since we will work with transverse foliations $\cF_1$ and $\cF_2$ we know that the only compact leaves can be tori (by our orientability assumption). Thus, there cannot be spherical leaves. On the other hand, we point out that it is possible to produce transverse foliations with Reeb components (see \cite{Hardop}). Foliations with Reeb components
are too flexible and lack tools to be studied in generality.
In addition such foliations are sometimes possible to exclude by other considerations (for instance, our assumption that the intersected  foliations have leafwise Hausdorff leaf space\footnote{If $\cF_1$ and $\cF_2$ are transverse foliations so that the intersected  foliation has  leafwise Hausdorff leaf space in every leaf, then, if $\cF_1$ has a Reeb component, it follows that when lifted to the universal cover, this Reeb component is either a solid torus, or an infinite solid cylinder. Let $L$ be a boundary leaf of the solid torus or cylinder, one can look at a leaf $E$ of $\wcF_2$ intersecting $L$. Then either the leaf space of $\cG_E$ or the leaf space of $\cG_L$ is
non Hausdorff, or there is some tangency between $\wcF_1$ and $\wcF_2$
in the interior of $L$. We have a detailed proof in a more
specific setting in a later section.}). Hence it makes sense to assume that the foliations are Reebless. 
\end{remark}

The following result is \cite[Lemma 7.21]{Calegari-book} (it also follows from \cite{Imanishi}). This will be used in the proof of Lemma \ref{lem.stabilizerinlimit} and in the proof of Theorem \ref{prop-closedleaf} to obtain leaves with non trivial stabilizer. 

\begin{teo}\label{teo.nonholonomy}
Let $\cF$ be a Reebless foliation on a closed 3-manifold $M$ with non abelian fundamental group and let $\wcF$ be its lift to the universal cover $\mt$. Let $\Lambda$ be a non-empty closed $\wcF$ saturated set which is $\pi_1(M)$-invariant, then, there exist $L \in \Lambda$ and $\gamma \in \pi_1(M) \setminus \{\mathrm{id}\}$ such that $\gamma L = L$.  
\end{teo}

\subsection{Hausdorff leaf space}\label{ss.HLS}

Here we show that the hypothesis of Theorem \ref{teo.main} have some equivalent formulations. This extends work of \cite{Barbot,FenleyAnosov} in the case of Anosov foliations to a more general context.

\begin{prop}\label{prop-Hsdff2Dand3D} 
The leaf space $\cO$ of $\wt{\cG}$ is Hausdorff if and only if 
the leaf space  $\cO_L$ of $\cG_L$ is Hausdorff for each $L$  of $\wcF_1$ and $\wcF_2$, 
\end{prop}

We stress that we need leafwise Hausdorff leaf space for both $\wcF_1$
and $\wcF_2$.  We first show the following elementary result that we will use repeatedly (see figure~\ref{fig-twocomponents}):

\begin{lema}\label{l.twointersnonHsdff}
Let $\cF_1, \cF_2$ be two transverse Reebless foliations of $M$ and assume that for some $L \in \wt{\cF_1}$ there is a leaf $E \in \wt{\cF_2}$ intersecting $L$ in more than one connected component. Then, the leaf space $\cO_L$ is not Hausdorff. 
\end{lema}

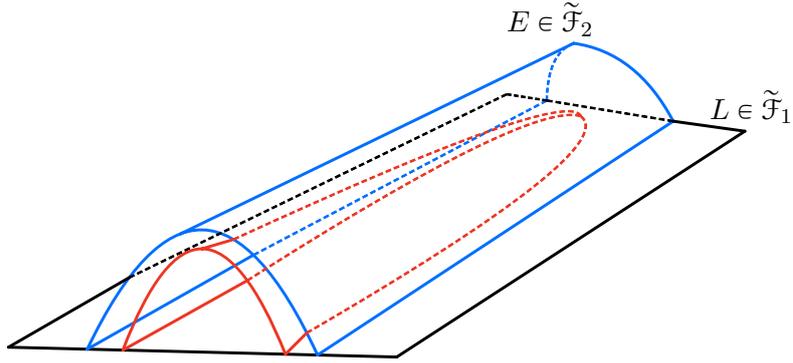
\begin{figure}[ht]
\begin{center}
\includegraphics[scale=0.82]{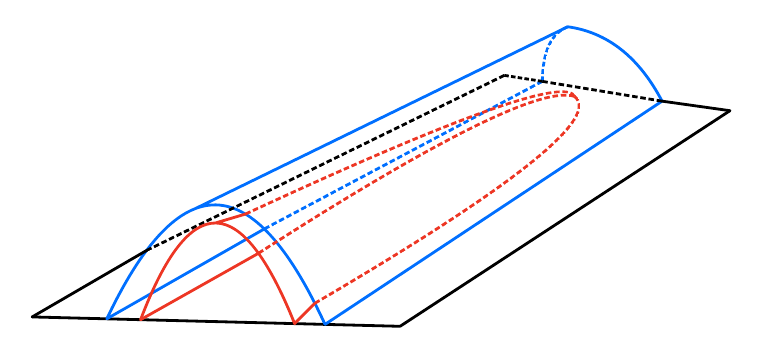}
\begin{picture}(0,0)
\put(-34,100){$L \in \wcF_1$}
\put(-110,133){$E \in \wcF_2$}
\end{picture}
\end{center}
\vspace{-0.5cm}
\caption{{\small When a leaf $L$ of $\wcF_1$ intersects a leaf $E$ of $\wcF_2$ in two connected components, the leafspace of $\cG_L$ is not Hausdorff.}}\label{fig-twocomponents}
\end{figure}

\begin{proof}
If there were a transversal to $\cG_L$ intersecting more than one connected components of $L \cap E$ we would get a transversal to $\wt{\cF_2}$ intersecting $E$ twice, contradicting that $\cF_2$ is Reebless due to Novikov's theorem (cf. Theorem \ref{teo.Novikovetc} (ii)). Thus, there must be some non-separated leaves in between. 
\end{proof}

\begin{proof}[Proof of Proposition \ref{prop-Hsdff2Dand3D}]
Note that if there is a leaf $L$ in one of $\wcF_1$ or $\wcF_2$ such that $\cO_L$ is not Hausdorff, then, this means that there is a sequence of leaves $\ell_n \in \cG_L$ which accumulate in at least two distinct leaves $c_1, c_2 \in \cG_L$. Since $\cG_L \subset \wt{\cG}$ this implies that $\cO$ cannot be Hausdorff either. This shows the direct implication.   

To show the converse we consider a sequence of leaves $\ell_n \in \wt{\cG}$ with points $p_n, q_n \in \ell_n$ such that $p_n \to p$ and $q_n \to q$. Assuming that for any $L$ which is either a leaf  of  $\wt{\cF_1}$ or a  leaf of  $\wt{\cF_2}$ we have that $\cG_L$ is Hausdorff, we want to show that $p$ and $q$ must belong to the same leaf of $\wt{\cG}$. Without loss of generality and up to subsequence, we can assume that in small foliated neighborhoods of $p$ and $q$ respectively, the points $p_n$ and $q_n$ are weakly monotonic in the leaf space of each of the foliations. 

Let $L_n \in \wt{\cF_1}$ and $E_n \in \wt{\cF_2}$ be so that $\ell_n \subset L_n \cap E_n$. 
In fact this implies that $\ell_n = E_n \cap E_n$ by the Hausdorff
hypothesis in $E_n$ or $L_n$ (here we only need the  Hausdorff
hypothesis for  one of the foliations $\wcF_1$ or $\wcF_2$).
Let also $L_p, L_q \in \wt{\cF_1}$ so that $p \in L_p$ and $q \in L_q$ and $E_p, E_q \in \wt{\cF_2}$ are defined so that $p \in E_p, q \in E_q$.  

We define $c_{n,k} = E_n \cap L_k$. Note that it is non-empty since $p_n$ and $p_k$ both belong to a foliated box close to $p$. Moreover, $c_{n,k}$ is a unique curve (that is, it is connected)
in both $L_k$ and $E_n$ because of Lemma \ref{l.twointersnonHsdff}.  

We can consider sequences $x_n \to p$ with $x_n \in L_p \cap E_n$ and similarly $y_n \to q$ with $y_n \in L_q \cap E_n$. Fixing $n>0$ so that $x_n$ is very close to $p$ we get that we can find points $z_k \to x_n$ and $w_k \to y_n$ in $c_{n,k}$. Since $\cO_{E_n}$ is Hausdorff, we deduce that $x_n$ and $y_n$ belong to the same leaf $e_n$ of $\cO_{E_n}$. In particular $L_p = L_q$. 
Here  we  used  the Hausdorff  property in leaves  of $\wcF_2$.

Now  we will use the  Hausdorff hypothesis in  leaves of $\wcF_1$:
notice that the leaves $e_n \subset L_p =L_q$ accumulate in both $p$ and $q$, thus, using that $\cO_{L_p}$ is Hausdorff we deduce that $p$ and $q$ belong to the same leaf of $\cG_{L_p}$ and thus of $\wt{\cG}$, concluding the proof.  
\end{proof}     

We note here that the fact that the intersected  foliation has leafwise Hausdorff leaf space implies that the topology of the leaves must be somewhat  restricted:

\begin{prop}
Let $\cF$ be a foliation in a closed 3-manifold and $\cT$ a one dimensional foliation which subfoliates $\cF$. If in the universal cover $\mt$ of $M$ we have that a leaf $L \in \wcF$ verifies that the foliation $\cT_L = \wt{\cT}|_L$ has Hausdorff leaf space. Then if the stabilizer $\mathrm{Stab}_L = \{\gamma \in \pi_1(M) \ : \ \gamma L = L\}$ is not abelian, it follows that there is $\ell \in \cT_L$ which is fixed by some $\gamma \in \mathrm{Stab}_L$.\footnote{In fact with more work one can prove in general that if $\cT_L$ has Hausdorff leaf space, then $Stab_L$ must be abelian even if it fixes leaves of $\cT_L$. But we do not need it here, so will not prove it.}
\end{prop}

\begin{proof}
This is a direct consequence of H\"{o}lder's theorem on free actions on the line (see e.g. \cite[Appendix E]{BFFP}). 
\end{proof}

Finally, we will show the following consequence of the leaf spaces being Hausdorff and compactness that allows us to detect non-Hausdorfness by looking at finite arcs of the foliation. Compare with \cite[Lemma 4.48]{Calegari-book}.

\begin{prop}\label{prop.Hsdffimpliesunifpropemb}
Let $\cF_1, \cF_2$ be two transverse foliations on $M$ and let $\cG=\cF_1 \cap \cF_2$. Then, the leaf space of $\wcG$ is Hausdorff if and only if there exists a proper function $\rho: \RR_+ \to \RR_+$ such that for every 
leaf $\ell$ of $\wcG$, we have that
if $x,y \in \ell$  then the length of the arc of $\ell$ joining $x$ and $y$ is bounded above by $\rho(d(x,y))$. 
There is a similar  statement for $\cF_2$.
\end{prop}

\begin{proof}
We assume first that there is a sequence of points $x_n, y_n$ in leaves $\ell_n \in \wcG$ such that the length of the segment $[x_n,y_n]$ in $\ell_n$ is  larger than $n$,
 but such that $d(x_n,y_n) \leq K$.

As $M$ is compact, then up to composing with deck transformations and taking subsequences, we assume that $x_n \to x_\infty$. Since $d(x_n,y_n) < K$,  then   $y_n$ has a subsequence converging to a point $y_\infty$.

Suppose that  that $x_\infty, y_\infty$ belong to the same leaf of $\wcG$, which  we  denote by $\ell_\infty$, and denote by $[x_\infty, y_\infty]$ the segment in $\ell_\infty$ that joins them. 

Let $U$ be a small foliated neighborhood (of $\wcG$) of $[x_\infty, y_\infty]$. It follows that for $n$ sufficiently large, the segment $[x_n,y_n]$ has to be completely contained in $U$ (this is because a leaf of $\wcG$ cannot intersect the same foliation box in more than one connected component, else we would contradict Novikov's Theorem \ref{teo.Novikovetc} for either $\cF_1$ or $\cF_2$).  This implies that the segment $[x_n, y_n]$ has bounded length, but it was chosen so that its length was larger than $n$, a contradiction.
We have proved then that the sequence $\ell_n \in \wcG$ converges to at least two distinct leaves of $\wcG$, therefore implying that the leaf space of $\wcG$ is not Hausdorff. 

To show the converse, consider a  sequence of points $x_n, y_n$ in leaves $\ell_n \in \wcG$ so that $x_n \to x_\infty$ and $y_n \to y_\infty$. Since $x_n, y_n$ converges, then $d(x_n,y_n)$ remains bounded, thus, the length of the segment $[x_n,y_n]$ that joins them must also remain bounded. Hence up to subsequence, $[x_n,y_n]$ converges to a segment of leaf $[x_\infty,y_\infty]$. This implies that the leaf space of $\wcG$ must be Hausdorff. 
\end{proof}

\subsection{Foliations by Gromov hyperbolic leaves} 

Denote by $\DD^2 = \{ z \in \CC \ : \ |z| \leq 1 \}$. We identify $\DD^2$ with the compactification of $\HH^2 = \{z \in \CC \ : \ \mathrm{Im}(z)>0 \}$ where the circle at infinity $\partial \HH^2 = \RR \cup \{\infty\}$ is identified with $S^1 = \partial \DD^2$. In $\HH^2$ or in the interior of $\DD^2$ one can put the canonical hyperbolic metrics and with this metric $\DD^2$ corresponds to the usual Gromov compactification of the interior of the Poincare disk. 

A classical result by Candel (see \cite[\S 7.1]{Calegari-book}) implies that if a foliation $\cF$ of a closed 3-manifold does not admit a transverse invariant measure of Euler characteristic $\geq 0$, then, one can choose a continuous Riemannian metric on $M$ so that when restricted to leaves it has constant negative curvature. As an example of the applicability of this result, we state the following consequence: 

\begin{prop}\label{prop-CandelThm0}
Let $\cF$ be a minimal foliation of a closed 3-manifold so that $\pi_1(M)$ is not virtually nilpotent. Then, every leaf of $\wt{\cF}$ is uniformly Gromov hyperbolic with the metric induced from the universal cover. 
\end{prop}
\begin{proof}
Candel's result implies that the result is true if there is no transverse invariant measure. The other case is a consequence of \cite[Theorem 5.1]{FP}.
\end{proof}

We say that a foliation $\cF$ is by (uniformly) \emph{Gromov hyperbolic leaves}\footnote{We note here that this is not the standard definition, which
for instance can be defined as satisfying a linear
isoperimetric inequality. However thanks to Candel's theorem it is equivalent. In fact, before Candel's theorem was proved already showing that the Gromov compactification of each leaf was a disk required some non-trivial arguments, see e.g. \cite{FenleyQI}.}  if there is a constant $Q$ so that every leaf $L \in \wt{\cF}$ we have that $L$ is $Q$-quasi-isometric to the hyperbolic plane with the metric of constant negative curvature when endowed with the Riemannian metric induced by the inclusion $L \hookrightarrow \mt$. Recall that a $Q$-\emph{quasi-isometric embedding} between metric spaces $(X,d_X)$, $(Y,d_Y)$ is a map $q: X\to Y$ so that:

\begin{equation}
Q^{-1} d_X(x_1,x_2) - Q \leq d_Y(q(x_1),q(x_2)) \leq Q d_X(x_1,x_2) + Q.
\end{equation}

A $Q$-\emph{quasi-isometry} is a $Q$-quasi-isometric embedding $q: X\to Y$ whose image is $Q$-dense, that is, for every $y \in Y$ there is $x \in X$ such that $d_Y(q(x),y)<Q$.

When a metric space $X$ homeomorphic to the plane is quasi-isometric to the hyperbolic disk there is a canonical compactification $\overline X = X \cup S^1(X) \cong \DD^2$ to a compact disk with the property that every quasi-isometric embedding of $\RR$ into $X$ extends to the two point compactification 
$\overline{\RR}$ of $\RR$ into $\overline X$ so that it is continuous
at the   endpoints of $\overline{\RR}$. Here 
the topology in $\overline X$ is given by declaring that a 
quasi-isometry homeomorphism from $X$ to $\HH^2$ extends to a homeomorphism of $\overline{X}$ to $\DD^2 = \HH^2 \cup \partial \HH^2$). 

If $\cF$ is a foliation by Gromov hyperbolic leaves, and $L \in \wt{\cF}$ we will denote by 

\begin{equation}\label{eq:compactifL}
\overline{L} = L \cup S^1(L)
\end{equation} \noindent its compactification.

\subsection{Closed geodesics} 

Here we show the following useful property: 

\begin{lema}\label{lem-closedgeodesics}
Let $\cF$ be a foliation by Gromov hyperbolic leaves on a closed 3-manifold $M$. Then, for every $L \in \wcF$ and $\gamma \in \pi_1(M)$ so that $\gamma L=L$ there is a geodesic $g  \in L$ which is $\gamma$-invariant (i.e. $\gamma g = g$).  
\end{lema}

We note that if the metric is not negatively curved in leaves, this geodesic may not be unique, but by Gromov hyperbolicity we know that any two such geodesics are (uniformly) bounded distance apart. We omit the proof of this standard result which depends on the classification of isometries of Gromov hyperbolic spaces (see e.g. \cite[Chapter 8]{GhysdelaHarpe}). 


\subsection{Limits of leaves} 

Consider $\cF$ a foliation by Gromov hyperbolic leaves of a closed 3-manifold $M$ and let $\mt, \wt{\cF}$ be the lifts to the universal cover. Given a leaf $L \in \wt{\cF}$ we have a compactification $\overline{L}= L \cup S^1(L) \cong \DD^2$ as explained in the previous section (see equation
\eqref{eq:compactifL}). 

Given a properly embedded curve $\ell \subset L$ we denote by: 

\begin{equation}\label{eq:limitcurve}
\partial \ell = \overline{\ell} \cap S^1(L) = \overline{\ell} \setminus \ell,
\end{equation}

\noindent where the closure is taken in the compactification $L \cup S^1(L)$. If the curve $\ell$ is oriented, we consider, for $x \in \ell$ the rays $\ell^+_x$ and $\ell^-_x$ to be the (closure of the) connected components of $\ell \setminus \{x\}$ according to the orientation. (Note that we took the closure so that we consider $x \in \ell^{\pm}_x$.) This way one can define (see figure \ref{fig-limitleaf}): 

\begin{equation}\label{eq:pmlimit} 
\partial^+ \ell = \overline{\ell^+_x} \setminus \ell^+_x \text{ and } \partial^- \ell = \overline{\ell^-_x} \setminus \ell^-_x
\end{equation}

\begin{figure}[ht]
\begin{center}
\includegraphics[scale=0.72]{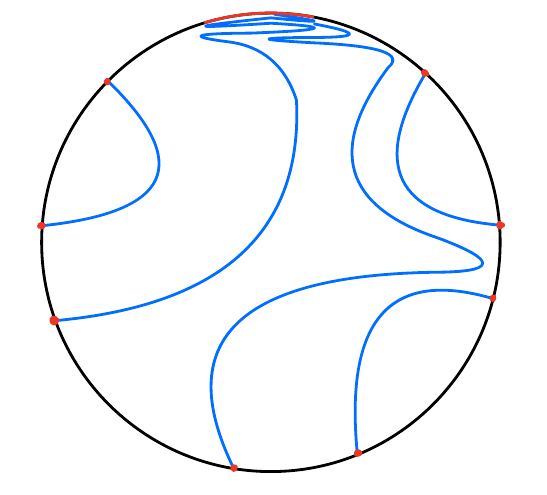}
\begin{picture}(0,0)
\end{picture}
\end{center}
\vspace{-0.5cm}
\caption{{\small Some leaves of $\cG_L$ inside $L$ with different landing behaviour. Note that $\cG_L$ does not have Hausdorff leaf space because there are three leaves of $\cG_L$ which land and define disjoint intervals in $S^1(L)$.}}\label{fig-limitleaf}
\end{figure}

Note that the notation does not include $x$ since the result of the operation is easily seen to be independent of the choice of $x$. Note that $\partial^{\pm} \ell$ is always a closed connected set.  It is connected because
the rays are properly embedded and any two  points disconnect  the
circle.   Hence  the limit set can be a closed proper interval, the full circle, or a singleton. The last case is important so it deserves a definition: 

\begin{defi}[Landing]\label{def.land}
Given a properly embedded oriented curve $\ell \subset L$ a leaf of $\wt{\cF}$, we say that the positive (resp. negative) ray of $\ell$ \emph{lands} in $\xi \in S^1(L)$ if $\partial^+ \ell = \{\xi\}$ (resp. $\partial^- \ell = \{\xi\}$). 
We also say that a ray $\ell$ lands if $\overline{\ell} \setminus \ell$
is a single point.
\end{defi}

It is a direct consequence of plane topology that the following holds: 

\begin{prop}\label{prop-overlaplimit}
If $\ell_1$ and $\ell_2$ are two oriented properly embedded disjoint curves in a leaf $L \in \wt{\cF}$ then, if $I=\partial^+ \ell_1$ is not equal to $S^1(L)$, then we have that $\partial^{\pm} \ell_2$ cannot be contained in the interior of $I$. 
\end{prop}
\begin{proof}
This is \cite[Proposition 2.12]{FP4}. 
\end{proof}

\subsection{Markers}\label{ss.markers}

In this subsection we borrow some results from Thurston, Calegari and Dunfield \cite{CD} (see also \cite{Calegari-book}). This will only be used in \S \ref{s.nonvisual}.

Given a foliation $\cF$ by Gromov hyperbolic leaves on $M$, a \emph{marker} $m$ is a map $m : [0,1] \times \RR_{\geq 0} \to \mt$ with the following properties: 

\begin{itemize}
\item $m(\{t\} \times \RR_{\geq 0})$ is a uniformly quasigeodesic ray in the leaf $L \in \wcF$ so that $m(t,0) \in L$. 
\item there is $\eps>0$ smaller than the foliation size boxes,
so  
so that for every $s \in \RR_{>0}$ we have that the length of $m([0,1]) \times \{s\})$ is smaller than $\eps$.  
\end{itemize}

The \emph{leaf pocket theorem} (see \cite[\S 5]{CD}) states that for a given leaf $L \in \wcF$ there is a dense set in $S^1(L)$ of \emph{marker directions}, i.e. directions $\xi \in S^1(L)$ for which there is a marker $m: [0,1] \times \RR_{\geq 0} \to \mt$ so that $m(0,0) \in L$ and so that $m(\{0\} \times \RR_{>0})$ lands in $\xi$. In fact, the result states that there is a dense set of marker directions in \emph{both sides} meaning that for any given transverse orientation, there is a dense set of marker directions associated to markers for which the curves $m([0,1] \times \{0\})$ are positive in the chosen orientation. It is worth pointing out that in \cite{CD} the results are proved using a Candel metric on which all leaves are isometric to hyperbolic planes in the universal cover; under those assumptions, the authors produce markers by geodesics instead of uniform quasigeodesics. By our assumption on Gromov hyperbolicity, Candel's theorem produces a metric in $M$ so that leaves are hyperbolic, and the identity is a homeomorphism between the metrics which induces uniform quasi-isometries between the leaves. This implies the result in our setting.

\subsection{Uniformly equivalent foliations and the universal circle of uniform foliations}\label{ss.uniformfol}

Here we review some results that will be used in \S \ref{ss.nonsolvable}. We refer the reader to \cite{ThurstonCircle1}, \cite{FPMin} and \cite{FenleyRcoveredtransverse} for more information. 

Given a foliation $\cF$ without Reeb components in a closed 3-manifold $M$ we say that $\cF$ is \emph{uniform} if for every pair of leaves $L, L' \in \wcF$ the Hausdorff distance between $L$ and $L'$ in $\mt$ is finite. Such foliations are always $\RR$-covered (see \cite{FPMin}). 

For such a foliation, the universal circle $S^1_{u}(\cF)$ has an easy description: Given $L, L' \in \wcF$ there is a coarsely well defined quasi-isometry $f_{L,L'}: L \to L'$ which induces a well defined homeomorphism $h_{L,L'}: S^1(L) \to S^1(L')$ between the circles at infinity. With this identification we can define the \emph{universal circle} $S^1_u(\cF)$ by identifying circles with this map. See \cite{ThurstonCircle1, FPMin, FenleyRcoveredtransverse} for details. 

We say that two foliations $\cF_1$ and $\cF_2$ are \emph{uniformly equivalent} if for every leaf $L \in \wcF_1$ there is a leaf $E \in \wcF_2$ which is bounded Hausdorff distance away from $L$ in $\mt$ and for every leaf $E \in \wcF_2$ there is a leaf $L \in \wcF_1$ which is bounded Hausdorff distance away form $E$ in $\mt$. Similarly to the case of a single uniform foliation, when we have two $\RR$-covered, uniformly equivalent foliations $\cF_1$ and $\cF_2$, then given $L \in \wcF_1$ and a leaf $E \in \wcF_2$ at bounded Hausdorff distance away, there is a coarsely well defined (it is defined by mapping each point in $L$ to a closest point in $E$) quasi-isometry $f_{L,E}: L \to E$. This 
induces a well defined homeomorphism $h : S^1(L) \to S^1(E)$. 
This uses that the foliations are $\RR$-covered, and
in the case that the foliations are uniform,
this induces an identification between the universal circles. This is because,
as we explained, when the foliations are $\RR$-covered and uniform, there is
a very natural identification of the universal circle
with the circle at infinity of each of its leaves.

\subsection{Anosov foliations}\label{ss.Anosovflows}

The contents of this subsection  will only be used in \S \ref{ss.RcoveredAF} and \S \ref{s.SVM}. 

Let $M$ be a closed 3-manifold, a \emph{(topological) Anosov flow} $\Phi$ on $M$ is a flow which has $C^1$ orbits, none is a point, and
 preserves  two topological foliations $\cW^{ws}$ and $\cW^{wu}$,
which are topologically transverse and  intersect in the orbit foliation of $\Phi$ with the property that the flow $\Phi$ is expansive. 
{\em Expansive} means  that there is an $\eps > 0$ so that
if two orbits (this should be checked in the universal cover $\mt$)
are Hausdorff distance $\eps$ from each other, then they are the same
orbit.
Up to relabeling, it follows that the foliation $\cW^{ws}$ which is called the \emph{weak stable foliation} of $\Phi$ consists of the orbits which are forward asymptotic to any given orbit in the leaf. Similarly, the \emph{weak unstable foliation} is made by backward asymptotic orbits. 

 We note that when $\Phi$ is transitive (in particular, when $\cW^{ws}$ or $\cW^{wu}$ are minimal) we have that $\Phi$ is orbitally  equivalent to a true Anosov flow, see \cite{Shannon}.  Although it is not standard, we will assume in this paper that the foliations $\cW^{ws}$ and $\cW^{wu}$ are $C^{1,0}$ and thus the flow is by $C^1$ curves. This allows to avoid technical discussions, for instance, on how to define the length of an arc of the flow and in this paper we will not need to deal with more general cases. We refer the reader to \cite[\S 5]{BFP} for more details, general definitions and discussion. 
  
The following fact will be used without reference: 

\begin{itemize}
\item The foliations $\cW^{ws}$ and $\cW^{wu}$ are by Gromov hyperbolic leaves, and orbits inside the leaves in the universal cover, form \emph{quasigeodesic fans} (that is, they are all asymptotic to the same point in the circle at infinity). See \cite[\S 5]{BFP} for proofs. 
\end{itemize}

We say that a (topological) Anosov flow is $\RR$-covered if the foliation $\cW^{ws}$ is $\RR$-covered (i.e. leaf space of $\wt{\cW^{ws}}$ is Hausdorff). As proved in  \cite{Barbot, FenleyAnosov}, this is equivalent to asking that $\cW^{wu}$ is $\RR$-covered. Topological Anosov flows which are $\RR$-covered are always transitive and thus orbitally  equivalent to true Anosov flows by \cite{Shannon} as explained above.  

The $\RR$-covered Anosov flows have two possibilities, one of which is
orbitally  equivalent to suspensions of linear automorphisms of $\TT^2$ (in which case, the fundamental group of $M$ is solvable). For $\RR$-covered
Anosov flows, the following conditions are all equivalent (see  \cite{Barbot, FenleyAnosov, BM}):

\begin{itemize}
\item The foliations $\wt{\cW^{ws}}$ and $\wt{\cW^{wu}}$ have a global product structure (i.e. for every $L \in \wt{\cW^{ws}}$ and $E \in \wt{\cW^{wu}}$ we have that $L \cap E \neq \emptyset$).
\item The foliation $\wt{\cW^{ws}}$ or the foliation $\wt{\cW^{wu}}$ is not uniform.
\item The fundamental group of $M$ is solvable.
\item The flow is orbitally  equivalent to the suspension of a linear hyperbolic automorphism of $\TT^2$. 
\end{itemize} 

If an Anosov flow is $\RR$-covered but does not verify some of the previous equivalent conditions, then it is called \emph{skewed}-$\RR$-\emph{covered}. As a consequence, being skewed-$\RR$-covered is equivalent to the weak stable and unstable foliations being \emph{uniform} and \emph{uniformly equivalent} to each other (recall \S\ref{ss.uniformfol} for definitions). Proofs are again contained in  \cite{Barbot, FenleyAnosov}. 

We will need the following property about the action of the fundamental group on the leaf space of a skewed-$\RR$-covered Anosov flow (which follows from \cite{FenleyAnosov}, see also \cite{BM}). We will continue to assume as a standing assumption that all foliations, in particular $\cW^{ws}, \cW^{wu}$ are transversally orientable. 

\begin{prop}\label{p.skewedAnosov}
Let $\Phi$ be a skewed-$\RR$-covered Anosov flow and let $\cW$ be its weak-stable or weak-unstable foliation. Then, for every $\gamma \in \pi_1(M) \setminus \{\mathrm{id}\}$ we have that either $\gamma$ acts as a translation on the leaf space of $\wt{\cW}$ or it has a countable number of fixed leaves $\{L_n\}_{n\in\ZZ}$ ordered by the transverse orientation and going to $\pm \infty$ in the leaf space as $n \to \pm \infty$. Moreover, in each $L_n$ there is a (unique) orbit $o_n$ of $\Phi$ which is invariant under $\gamma$ and on which $\gamma$ acts as a translation whose orientation with respect to the flow direction is different depending on whether $n$ is even or odd. 
\end{prop}

Finally, the following property of skewed-$\RR$-covered Anosov flows will be crucial to get a contradiction to complete the proof of Theorem \ref{teo.main}. (See \cite{FenleyAnosov}.)

\begin{prop}\label{prop.nonmarkermoves}
For skewed-$\RR$-covered Anosov flows, the map that sends each leaf of $\wt{\cW^{ws}}$ (resp. $\wt{\cW^{wu}}$) to the point at infinity in the universal circle corresponding to the common point where all orbits lands, is a strictly monotonic map.
\end{prop}

Note that this point is sometimes called the non-marker point because all the other points are markers as in \S\ref{ss.markers}.

\section{Topology at infinity of foliations}\label{s.folinfinity} 
\label{s.topologytub}

We will consider in this section a foliation $\cF$ of $M$ by uniformly Gromov hyperbolic leaves, which as we mentioned before means that each leaf of $\wcF$ with the metric induced by its ambient path metric is quasi-isometric to $\HH^2$. 
Using the geometry of Gromov hyperbolicity each leaf $F$ of
$\wcF$ is canonically compactified with a circle at infinity.
We want to put a topology in the union of $\mt$ with all these
circles at infinity.
The main point is to be able to analyze the topology at infinity as one moves from one leaf to a nearby one. 
This has been done previously, and in a very natural way, when the leaves of
$\wcF$ have a hyperbolic metric or a negatively curved metric
\cite{FenleyRcov,Calegari-book}.
In particular the treatments in \cite{FenleyRcov,Calegari-book} require changing the original metric
to another metric. 
We give a different presentation. Instead of changing the metric (for
example to a negatively curved metric) to work with our foliation, we try to present the results from a coarse point of view, which will work with any metric.
This is because we want to be able to work with a fixed underlying metric so as to avoid needing to change the metric back and forth when working with two foliations,
which is the situation we are considering in this article.
The results here will only be used here in sections 
\ref{s.nonvisual} and \ref{s.SVM}, but obviously will be useful
whenever considering pairs of  transverse foliations.

%

\begin{defi} \label{def.tubulation}
The tubulation of $\cF$ or $\wcF$ is the following set:

$$\cTu \ \  = \ \ \bigsqcup_{L \in \wcF} S^1(L)$$ 
\end{defi}

Sometimes we will call $\cTu$ the tubulation at infinity.
We will endow $\cTu$ with a topology which makes it a circle bundle over the leaf space $\cL$ of $\wcF$. 

Given a Riemannian metric on $M$ we consider the lifted metric on $\mt$ which induces in each leaf of $\wcF$ a path metric which by our assumption is uniformly quasi-isometric to $\HH^2$. The metric varies continuously between leaves with respect to uniform convergence in compact sets. 

A {\em length minimizing}, or {\em globally minimizing},
or {\em minimizing}  segment, ray or bi-infinite curve $\beta$ in
a leaf $L$ of $\wcF$ is one so that: for any $x,y \in \beta$,
the length of the segment $[x,y]$ in $\beta$ from $x$ to $y$
is  a shortest length of any path $x$ to  $y$.

\begin{remark} The following well known fact will be used throughout 
this section: if $\ell_n$ is a sequence of length minimizing segments,
rays, or bi-infinite geodesics in leaves $L_n$ of $\wcF$, and
$\ell$ is a limit of $\ell_n$, then $\ell_n$ is also
length minimizing in its leaf of $\wcF$.
To show this just notice that if a path is not minimizing, then closeby paths cannot be minimizing either (see e.g. \cite[Proposition 4.3]{FenleyQI} for a detailed proof). 
Also unless otherwise stated, by a geodesic  we mean a length
minimizing geodesic.
\end{remark}

Given $p \in \mt$ we consider $S^1_p = T^1_p L$ the set of unit tangent vectors at $p$ to the leaf $L \in \wcF$ such that $p \in L$. This is a topological circle. We consider in $S^1_p$ the subset $A_p \subset S^1_p$ of the vectors such that the geodesic ray $\gamma_p(\theta): [0,\infty) \to L$ starting at $p$ with velocity $\theta \in S^1_p$ verifies that it is  globally minimizing. 
The remark above implies that the set $A_p$ is closed. The remark also
implies that if $p_n \to p$ then $\mathrm{limsup} A_{p_n} = \bigcap_{n>1} \overline{\bigcup_{k>n} A_{p_k}} \subset A_p$ (considering the unit tangent bundle $T^1\wcF$ to the leaves of the foliation as a subbundle of $T^1\mt$). 

It is shown in \cite[Lemmas 4.4 and 4.5]{FenleyQI} that the map sending each point $\theta \in A_p$ to $S^1(L)$ is a monotone quotient, that is, there is a continuous surjective map $\psi_p: A_p \to S^1(L)$ with the property that if $\psi_p(\theta)= \psi_p(\omega)$ then there is exactly one connected component $B_{\theta,\omega}$ of the complement in $L$ of the geodesic rays from $p$ with initial velocities $\theta$ and $\omega$ which is contained in a uniform neighborhood of each of the geodesic rays (this region is an ideal geodesic bigon). One can extend $\psi_p : S^1_p \to S^1(L)$ (we abuse notation and use the same notation for the extended map) by mapping all the vectors pointing into $B_{\theta,\omega}$ to $\psi_p(\theta)=\psi_p(\omega)$. 


Fix a transversal $\tau: (-\eps, \eps) \to \mt$ to $\wcF$ and consider a continuous curve $a: (-\eps, \eps) \to T^1\cF$ such that $a(t) \in S^1_{\tau(t)}$ for all $t$. If $a_-,a_+: (-\eps,\eps) \to T^1\cF$ are the (not necessarily continuous) maps so that $[a_-(t),a_+(t)] = \psi_{\tau(t)}^{-1}(\psi_{\tau(t)} (a(t)))$ (with a chosen continuous orientation of $S^1_{\tau(t)}$) we get that if $t_n \to t$ then the (minimizing) geodesic rays from $\tau(t_n)$ with velocity $a_{\pm}(t_n)$ converge to a minimizing geodesic ray from $\tau(t)$ and with endpoint in $\psi_{\tau(t)}(a(t))$. 

We want to give a topology on $\mt \cup \cTu$ by extending the topology of $\mt$. For this, given $\xi \in \cTu$ we want to define a basis of neighborhoods $V_n(\xi)$ for $\xi \in S^1(L)$ for some $L \in \wcF$ that will be given by the following data:
\begin{itemize}
\item a sequence $p_n \in L$ so that $p_n \to \xi$ in he Gromov
compactification of $L$,
\item a sequence of transversals $\tau_n: (-\eps_n, \eps_n) \to M$ to $\wcF$ with $\tau_n(0) = p_n$ and such that the sequence of leaves intersected by $\tau_n$ forms a basis of neighborhoods of $L$ in $\cL$ the leaf space of $\wcF$ (we choose the parametrization so that for every $n,m$, if $t \in (-\eps_n,\eps_n) \cap (-\eps_m, \eps_m)$ then the leaf $L_{\tau_n(t)} = L_{\tau_m(t)}$ and we call it $L_t$).
We assume that the length of $\tau_n$ goes to $0$ as $n \to \infty$.
\item a strictly decreasing sequence of closed intervals $I_n$ of $S^1(L)$ forming a basis of neighborhoods of $\xi$ in $S^1(L)$, 
\item continuous curves $a^n, b^n : (-\eps_n, \eps_n) \to T^1\wcF$ such that $a^n(t) \neq b^n(t) \in S^1_{\tau_n(t)}$ and so that the points $a^n(0), b^n(0) \in S^1_{p_n}$ and correspond to the endpoints of the interval $I_n$ in $S^1(L)$ via $\psi_{p_n}$ (i.e. for a given orientation we have that  $\psi_{p_n}([a^n(0),b^n(0)]) = I_n$).
\item if we denote by $I_n(t) = \psi_{\tau_n(t)}([a^n(t)), bn(t)])$ then we have that for all $n \geq m$ we have that if $t \in (-\eps_n,\eps_n) \cap (-\eps_m, \eps_m)$ then $I_n(t) \subset \mathrm{Interior}(I_m(t))$.  
\end{itemize}

\begin{figure}[ht]
\begin{center}
\includegraphics[scale=0.62]{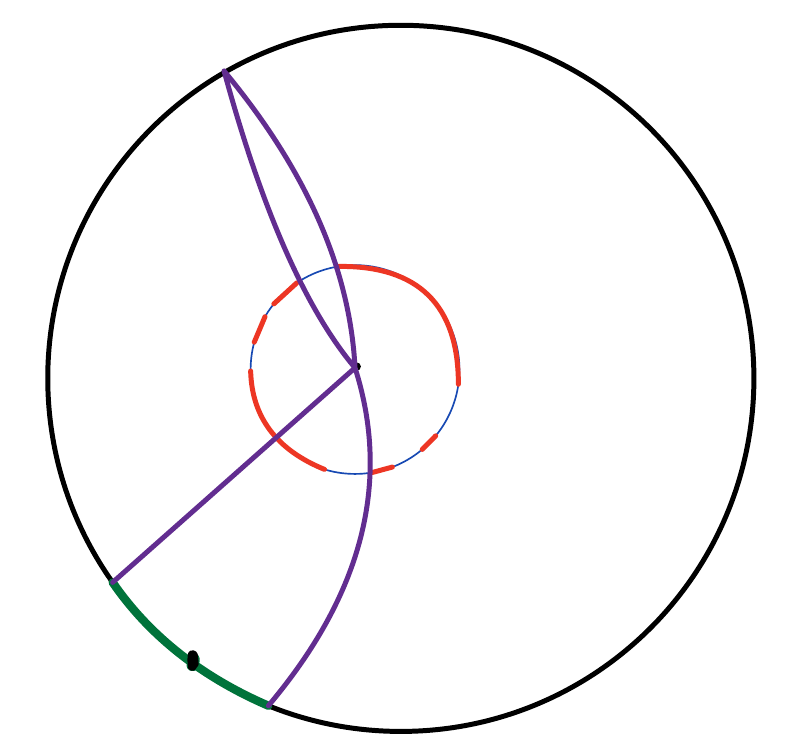}
\begin{picture}(0,0)
\put(-37,40){$L$}
\put(-120,136){$A_p$}
\put(-134,110){$p$}
\put(-195,20){$\xi$}
\put(-214,32){$I_n$}
\put(-176,45){{$W_n(0)$}}
\end{picture}
\end{center}
\vspace{-0.5cm}
\caption{{\small Depiction of some of the objects appearing in the definition of the neighborhood $V_n(\xi)$. The inner circle represents the unit tangent vectors to $p$ and painted in red are the vectors in $A_p$.}}\label{fig-topo}
\end{figure}

With this data, we construct 

$$V_n(\xi) = \bigcup_{t \in (-\eps_n,\eps_n)} W_n(t) \cup I_n(t)$$

\noindent where $W_n(t)$ are the open wedges in $L_t$ (the leaf of $\wcF$ containing $\tau_n(t)$) bounded by the geodesic rays $\gamma_+^t$ and $\gamma_-^t$ starting at $\tau_n(t)$ with velocities $a^n_+(t)$ and $b^n_-(t)$ where $[a^n_-(t),a^n_+(t)] = \psi_{\tau(t)}^{-1}(\psi_{\tau(t)} (a^n(t)))$ and $[b^n_-(t),b^n_+(t)] = \psi_{\tau(t)}^{-1}(\psi_{\tau(t)} (b^n(t)))$ with the chosen orientation, and $I_n(t) \subset S^1(L_t)$ are as defined above. 

\begin{lema}\label{remark-nested}
It is possible to construct sequences of points $p_n$, transversals $\tau_n$, intervals $I_n$ and continuous functions $a^n, b^n$ with the desired properties. Moreover, they can be chosen so that for $n>m$ one has that $V_n(\xi) \subset V_m(\xi)$. 
\end{lema}
\begin{proof}
 The existence of points $p_n$ and transversals $\tau_n$ does not need justification, similarly with the intervals $I_n$ that we can assume are nested in the sense that if $n >m$ then $I_n$ is contained in the interior of $I_m$ (which is non-empty as it contains $\xi$).  
 Before we justify the rest of the properties, let us assume (as we will from now on, as it amounts to taking a subsequence) that for $n > m$ we have that $(-\eps_n, \eps_n) \subset (-\eps_m,\eps_m)$. Under this assumption, by choosing the points $p_n$ correctly, and maybe by further taking subsequences and maybe shorter transversals, we can assume that if $n>m$ then the image of the curve $\tau_n$ is contained in $V_m(\xi)$ constructed above.
 
Let $I_n = I_n(0)$, notice this is an interval with non empty
interior for all $n$.

 Now, we need to justify why we can choose the curves $\{a^n(t)\}_n, \{b^n(t)\}_n$ so that $I_n(t)$ is contained in the interior of $I_m(t)$ when $n >m$ (and $t \in (-\eps_n,\eps_n)$). For simplicity we will consider the case where $n = m+1$ since then an induction completes the argument. Indeed, we will see that this can be achieved by taking the transversals shorter. 

\vskip .06in
First we show that by taking transversals shorters we can assume that the intervals $I_n(t)$ have non-empty interior for all $t \in (-\eps_n, \eps_n)$.
Suppose this is not the case, then fix an $n$ for which this is not true.
Then we can take $t_k \to 0$ so that $\psi_{\tau_n(t_k)}(a^n(t_k)) = \psi_{\tau_n(t_k)}(b^n(t_k))$. Now we define two sequences of minimizing geodesic
rays as follows: let $\gamma_+^{t_k}$ (resp. $\gamma_-^{t_k}$)
be the ray in $L_{t_k}$ starting
in $\tau_n(t_k)$ with direction $b^n_+(t_k)$ (resp. $a^n_-(t_k)$). 
By definition 

$$\psi_{\tau_n(t_k)}(b^n_+(t_k)) \ = \ 
\psi_{\tau_n(t_k)}(b^n(t_k))$$

\noindent
 and similarly for $a^n_-(t_k)$.
It follows that $\psi_{\tau_n(t_k)}(b^n_+(t_k)) =
\psi_{\tau_n(t_k)}(a^n_-(t_k))$.
In particular $\gamma_-^{t_k}, \gamma_+^{t_k}$ are a globally
bounded distance from each other in $L_{t_k}$ for all $k$.
Up to subsequence, suppose that
the minimizing geodesic rays $\gamma_+^{t_k}$ and $\gamma_-^{t_k}$ converge 
to minimizing geodesic rays $\gamma_+^\infty$ and $\gamma_-^\infty$ starting
in $p_n$, with directions $b_{\infty}, a_{\infty}$.
It follows that $\gamma_+^{\infty}, \gamma_=^{\infty}$ are
a bounded distance from each other in $L_0$.
In particular this implies that 
$\psi_{\tau_n(0)}(a_{\infty}) = \psi_{\tau_n(0)}(b_{\infty})$.
We analyze the limit rays to obtain a contradiction.
By construction, for all $k$ we have that

$$[a^n(t_k),b^n(t_k)] \ \subset \ [a_-^n(t_k), b_+^n(t_k)].$$

\noindent
Hence $[a^n(0),b^n(0)] \subset [a_{\infty},b_{\infty}]$.
This is because $a^n,b^n$ are continuous and $a_{\infty}, b_{\infty}$
were defined as limits of $a_-^n(t_k),b_+^n(t_k)$ respectively.
The definition of $\psi$ then implies that 
$\psi_{\tau_n(0)}(a^n(0)) = \psi_{\tau_n(0)}(b^n(0)$.
But these points are the endpoints of $I_n$ which are not
the same. This contradiction shows that if $t$ is small
then $I_n(t)$ has non empty interior. 


\vskip .06in
Now we proceed in  a similar way to show that we can choose the intervals shorter so that for $n>m$, 
 we have that $I_n(t)$ is contained in the interior of $I_m(t)$ for all $t \in (-\eps_n,\eps_n)$.
As mentioned before
we consider the case $n=m+1$ and apply induction.

Suppose the property of intervals is not true for some $m$. Then
 there is a sequence $t_k \to 0$ and
so that $I_n(t_k)$ intersects the boundary of $I_m(t_k)$.
Then there are 
$v(t_k) \in [a^n_+(t_k), b^n_-(t_k)]$ with
 $\psi_{\tau_n(t_k)}(v(t_k)) \in \partial I_m(t_k)$.
Without loss of generality, we can assume that 
$$\psi_{\tau_n(t_k)}(v(t_k)) \ = \ \psi_{\tau_m(t_k)}(a^m(t_k)).$$
Since $\psi_{\tau_n(t_k)}(A_{\tau_n(t_k)})= S^1(L_{t_k})$ we can assume also without loss of generality that 
$$v(t_k) \ \in \ A_{\tau_n(t_k)} \cap [a^n_+(t_k), b^n_-(t_k)].$$
Up to taking subsequences, we can assume that the (minimizing) geodesic rays $\eta_k$ from $\tau_n(t_k)$ with initial vector $v(t_k)$ converge to some minimizing geodesic ray $\eta_\infty$ from $p_n=\tau_n(0)$ with starting velocity $v_\infty$. Since for every $k$, 
$$[a^n_+(t_k),b^n_-(t_k)] \ \subset \ [a^n(t_k),b^n(t_k)], \ \ 
{\rm then} \ \  v_\infty \in [a^n(0), b^n(0)],$$ 
and thus, the ideal point $c$ of $\eta_\infty$ (that is, $c = \psi_{\tau_n(0)}(v_\infty)$) 
belongs to $I_n(0)=I_n$. 

On the other hand, we have that (maybe after further subsequence) the (minimizing) geodesic rays $\gamma_+^{t_k}$ (resp.
$\gamma_-(t_k)$) in $L_{t_k}$ with initial point $\tau_m(t_k)$ and velocity $a^m_+(t_k)$ 
(resp. $a^m_-(t_k)$ converge to some minimizing geodesic ray 
$\gamma_+^\infty$ (resp. $\gamma_-^\infty$)
from $p_m = \tau_m(0)$ with velocity
$u_+$ (resp. $u_-$) and ideal point $d_+ = \psi_{\tau_m(0)}(u_+)$ 
(resp $d_- = \psi_{\tau_m(0)}(u_-)$).
We claim that $d_+, d_-$ are both in  the boundary of $I_m$, in fact 
both are equal to $\psi(a^m(0))$.
For all $k$ 
$$\psi_{\tau_m(0)}(a^m_+(t_k)) \ = 
\psi_{\tau_m(0)}(a^m_-(t_k))$$
so $\gamma_+^{t_k}, \gamma_-^{t_k}$ are a uniformly bounded
distance from each other in $L_{t_k}$. Consequently $\gamma_-^\infty,
\gamma_+^\infty$ are a bounded distance from each other in $L_0$
and have the same ideal point.
For each $k$, $a^m(t_k) \in [a^m_-(t_k),a^m_+(t_k)]$, hence
$a^m(0) \in [u_-,u_+]$ and so $d_- = d_+ = \psi_{\tau_m(0)}(a^m(0))$
which is in the boundary of $I_m$ as stated.

Now we remark that the rays $\eta_k$ and $\gamma_+^{t_k}$ are at uniformly bounded Hausdorff distance in $L_{t_k}$ $-$ because they have the starting points $\tau_m(t_k), \tau_n(t_k)$ which are a globally
bounded distance from each other in $L_{t_k}$ and they
have the same point at infinity $\psi_{\tau_m(t_k)}(a^m_+(t_k))$. 
Thus by Gromov hyperbolicity, they remain at uniformly bounded distance in $L_{t_k}$.
Therefore their limits must also be at finite Hausdorff distance and thus land in the same point in $S^1(L)$. 
The limit point of $\eta_\infty$ is $c$ which is in $I_n$.
The limit point of $\gamma_+^\infty$ is $d_+$ which is in the boundary of $I_m$. They are equal to each other and to $\psi_{p_m}(a^m_+(0))$.
Therefore $I_n$ has a point in the boundary of $I_m$.
This contradicts the choice that $I_n$ 
is contained in the interior of $I_m$.  
 \end{proof}

\begin{lema} \label{lem-hausd}
Let $\xi$ in $S^1(L)$.
One can  choose the sets $V_n(\xi)$ so that $\xi$ is  separated
from any point in $L$. 
If  $\zeta \in  S^1(L) \setminus \{  \xi \}$ one
can chooose $V'_m(\zeta)$ so that $\xi, \zeta$  are
separated  from each other.
We can choose the  $V_n(\xi)$ so that 
$d_{L_t}(W_n(t),\tau_1(t))$  converges to infinity
when  $n \to \infty$.
\end{lema}

\begin{proof}
Recall that $L  = L_0$.
We  first prove the last property. The experession 
$d_{L_t}(W_n(t),\tau_1(t))$  assumes that $|t| < \eps_n$.
First we consider   $t = 0$. We   know that $d_{L_0}(p_1,p_n)
\to \infty$. We show that  
$d_{L_0}(p_1,W_n(0)) \to  \infty$, as follows:  
let  $\gamma^n_-,\gamma^n_+$ be
the boundary rays of $W_n(0)$.  
If  the property does not follow then  distance to at least  one
of  the sides  is bounded with  $n$.
Suppose  (say) that $d_{L_0}(p_1,\gamma^n_+)$ is bounded.
Then there is a subsequence $n_k$  so that $\gamma^{n_k}_+$
converges to  a bi-infinite geodesic $\gamma$. One of the ideal
points of $\gamma$ is not  $\xi$, and this is a contradiction
as in the proof of the previous lemma.

Let $r_n = d_{L_0}(p_1,W_n(0))$. This converges to $\infty$ with $n$.
Fix $n_0$ so that  if $n  > n_0$ then $r_n$  is bigger than
$10  a_0$, where $a_0$ is an upper bound of the Hausdorff distance
between pairs  of minimizing rays with same endpoints in a leaf of
$\wcF$.
By decreasing $\eps_n$ if necessary,  we  claim that  $d_{L_t}(W_n(t),\tau_1(t)) \geq 
r_n/2$ for all $t$.
If not  there is a subsequence  $t(k) \to 0$ so that
$d_{L_{t(k)}}(W_{n_k}(t(k)),\tau_{n_k}(t_k) < r_n/2$.
Let  $\gamma^k_+,\gamma^k_-$ be the
boundary rays  of $W_{n}(t(k))$ and assume  wlog that
$d_{L_{t(k)}}(\tau_1(t(k)), \gamma^k_+) < r_n/2$ for all $k$.
Then $\gamma^k_+$ converges to a geodesic ray $\gamma$.
As in the  proof  of the previous lemma, this  geodesic
ray is contained  in the $a_0$  neighborhood of $W_n(0)$ in $L_0$.
But it  has points at least $r_n/2$ from this neighborhood,
contradiction.
This proves the first property.

This shows that one can choose $V_n(\xi)$ so that $\xi$ is
separated from any point in $L$.

For the  second   property choose $V'_n(\zeta)$ a neighborhood
basis of $\zeta$ so that $V'_n(\zeta)$ intersects the  same
set of leaves  of $\wcF$  that $V_n(\xi)$   does  (change the $t$
parametrization,  and restrict intervals if necessary). 
Also  chooose it  so that at $t =  0$  the sets for 
$\zeta, \xi$ are disjoint in $L  \cup S^1(L)$. 
In fact we can choose them so that $d_{L_0}(W_1(0),W'_1(0)) 
> 10 a_0$.
This  implies  that the $I_n, I'_n$
are  disjoint  for  any $n$.
If  there is $t_k \to 0$ so that $W_1(t_k),  W'_1(t_k)$
intersect, then (say) one boundary component of  $W'_1(t_k)$
(call it $\gamma_k$) intersects $W_1(t_k)$.  Take
a subsequence and limit  $\gamma$ as $k \to \infty$. Then 
as  in  the  previous lemma,
$\gamma$ is $a_0$ distant from $W_1(0)$, but it intersects the
$a_0$ neighborhood of $W'_1(0)$. This  contradicts  the choices
of $W_1(0), W'_1(0)$.
This   finishes the proof.
\end{proof}

We will show that the sets $V_n(\xi)$ allow us to give a topology on $\mt \cup \cTu$ and then that this topology is independent of our choices. 

\begin{lema}\label{lem-topologytubulation1}
The topology generated by this (decreasing) basis of neighborhoods of points in $\cTu$ is compatible with the topology of $\mt$.
\end{lema} 

\begin{proof}
The fact that the basis is decreasing for a given $\xi \in \cTu$ as $n$ increases is a consequence of the choice of neighborhoods $V_n(\xi)$ (cf. Lemma \ref{remark-nested}). 

We must then show that $V_n(\xi) \cap \mt$ is open.  For this, it is enough to show that if $t_k \to t \in (-\eps_n,\eps_n)$ is such that $\gamma_-^{t_k}$ and $\gamma_{+}^{t_k}$ converge to geodesics $\hat \gamma_-$ and $\hat \gamma_+$ in $L_t$ then it follows that $\hat \gamma_{-}$ and $\hat \gamma_+$ are not contained in the interior of the wedge bounded by $\gamma_+^t$ and $\gamma_-^t$. 
Here $\gamma^t_+$ (resp. $\gamma^t_-$) 
is the geodesic ray in $L_t$ with starting 
velocity $a^n_+(t)$ (resp. $b^n_-(t)$).
To see this, note that if $v_k^+$ and $v_k^-$ are vectors in $S^1_{\tau_n(t_k)}$ directing the geodesic rays $\gamma_-^{t_k}$ and $\gamma_{+}^{t_k}$ we get that $\psi_{\tau_n(t_k)}(v_k^-)= \psi_{\tau_n(t_k)}(a^n(t_k))$ and  $\psi_{\tau_n(t_k)}(v_k^+)= \psi_{\tau_n(t_k)}(b^n(t_k))$, in particular $[v_k^-, v_k^+] \subset [a^n(t_k),b^n(t_k)]$. Thus in the limit we get that if $v_\infty^-$, $v_{\infty}^+$ are the limits of $v_k^-, v_+$, we get $\psi_{\tau_n(t)}(v_\infty^-)= \psi_{\tau_n(t)}(a^n(t))$ and $\psi_{\tau_n(t)}(v_\infty^+)= \psi_{\tau_n(t)}(b^n(t))$ 
(see the proof of the Lemma \ref{remark-nested}).
Note also that $v_\infty^\pm \in A_{\tau_n(t)}$ because they are limits of points in $A_{\tau_n(t_k)}$. Hence it follows that the geodesics $\hat \gamma_\pm$ cannot cross the geodesics $\gamma_\pm^t$ and cannot intersect the interior of the wedge because the geodesics $\gamma_\pm^t$ are directed by the innermost vectors with the property that mapped by $\psi_{\tau_n(t)}$ go to the same points as $a^n(t)$ and $b^n(t)$.   
\end{proof}

\begin{lema}\label{lem-topologytubulation2}
The topology is independent of the choices of the points $p_n$, the transversals $\tau_n$, the sequence of intervals $I_n$ in $S^1(L)$, and the curves $a^n,b^n$. 
\end{lema}

\begin{proof}
Take different sequences $p_n'$ transversals $\tau'_n$ and curves $(a^n)',(b^n)'$ of tangent vectors and denote by $V_n'(\xi)$ the obtained sets. We must show that given $n$, there is $m$ such that $V'_m(\xi) \subset V_n(\xi)$. 

Note that if $m$ is large enough, the transversal $\tau'_m$ intersects a subset of the leaves intersected by $\tau_n$, we can, up to reparametrizing, assume that $\tau'_m(t)$ and $\tau_n(t)$ belong to the same leaf of $\wcF$. Also, since $p_m' \to \xi$ and the lengths of $\tau'_m$ converge to zero,
 we can assume by further taking subsequences that the image of $\tau_m'$ is contained in $V_n(\xi)$. 


To prove the result it is therefore enough to show that for $m$ sufficiently large, one has that $I'_m(t) \subset I_n(t)$ for all the $t$ so that $I'_m(t)$ is defined (note that as $m \to \infty$ the values of $t$ where $I_m'(t)$ is defined converge to $0$. Note that if $m$ is large enough we have that $I_m'=I_m'(0)$ is contained in the interior of $I_n=I_n(0)$.

If this is not the case, since $I_k'(t) \subset I_m'(t)$ for $k>m$ and $t \in (-\eps'_k, \eps'_k) \subset (-\eps_m', \eps'_m) \subset (-\eps_n,\eps_n)$ one can construct a sequence $t_j \to 0$ on which $I_m'(t_j)$ is not contained in $I_n(t_j)$. Now we can argue as in Lemma  \ref{remark-nested} to construct a sequence of minimizing geodesic rays from $\tau_m'(t_j)$ to a point in the boundary of $I_n(t_j)$ so that the ray is contained in $V_m'(\xi)$ for all $j$. Taking limits, one obtains a minimizing geodesic ray from $\tau_m'(0)=p_m'$ to some point in the boundary of $I_n$ which is contained in the closure of $V'_m(\xi)$. This contradicts that $I_m'$ is contained in the interior of $I_n$. 
\end{proof}

\begin{lema}\label{lem-topologytubulation3}
The topology induced is independent of the metric in $M$. In fact if $f: M \to M$ is a homeomorphism preserving $\cF$ it follows that any lift $\ft$ to $\mt$ extends continuously to a homeomorphism of $\cTu$ (in particular, this holds for deck transformations). 
\end{lema}
\begin{proof} 
Let $f: M \to N$ be a homeomorphism sending a foliation $\cF$ of $M$ to a foliation $\cD$ in $N$ then, we will show that any lift $\ft: \mt \to \widetilde{N}$ extends to a homeomorphism from  $\mt \cup \cTu$ to $\widetilde{N} \cup  \cTuD$ with the induced topologies. 

To see this, since $f$ is the lift of a homeomorphism on a compact manifold, it maps geodesic rays to (uniform) quasigeodesic rays. 
In particular, leafwise the following happens: $\ft: L  \to \ft(L)$ extends
canonically to  a homeomorphism still denoted by $\ft:  L \cup  S^1(L)
\to \ft(L) \cup  S^1(\ft(L))$.
This is the extension and it is a bijection.
We show that it is continuous (hence by the same argument
the inverse is also continuous).

Let $\xi \in S^1(L)$, $L  \in \wcF$ a point in the tubulation.
Let  $L_0 =  L$,  and
choose neighborhood basis  $V_n(\xi)$ and $V'_m(\ft(\xi))$.
Fix one $m$ and let $Z' = V'_m(\ft(\xi))$ and  $Z = \ft^{-1}(Z')$.
We will  show that $Z$ contains some $V_n(\xi)$  and this
will prove  continuity of  $\ft$  at $\xi$.
First of all since $\ft$ is a homeomorphims from $L_0 \cup S^1(L_0)$
to $\ft(L_0) \cup S^1(\ft(L_0))$ it follows that $Z \cap L_0$ contains
some $V_{n_0}(\xi) \cap L_0$ so contains $V_n(\xi)  \cap  L_0$  for
$n \geq n_0$.  We consider $n \geq n_0$.

Let $\tau'$  be the transversal to $\widetilde \cD$ made up of the
corners of $Z'$. Let $\tau = \ft^{-1}(\tau')$.
Assume that $Z$ does not contain $V_n(\xi)$ for  any $n$.
But we know that $Z$ contains $\tau_n(0)$ for all $n$ ($\geq n_0)$.
Hence it contains $\tau_n(t)$ for all  $|t| \leq  t_0$ for some
$t_0 > 0$. If $Z$ contains $V_n(\xi) \cap L_t$ for all $|t|  \leq  t_0$,
then $Z$  will contain $V_n(\xi)$ as soon as $\eps_n < t_0$.
In this case we are done.

Therefore $Z$ does not  contain  $V_n(\xi) \cap L_t$ for  
some $|t| < \eps_n$, but  $\tau_n(t)$ is contained in $Z$.
We denote this $t$ by $t_n$. It follows  that one 
boundary ray of $V_n(\xi) \cap   L_{t_n}$ intersects
a  boundary ray of  $Z \cap L_{t_n}$.  Let  an
intersection point be denoted by $p_n$.

Either $p_n$ is in a minimizing geodesic ray in $L_{t_n}$ starting
from $\tau(t_n)$, or $p_n$ is containing between two
such rays which are at most  $a_0$ Hausdorff distant from
each other in $L_{t_n}$. One of these rays  has to intersect the boundary
of  $V_n(\xi) \cap  L_{t_n}$. Denote such a ray by $\gamma_n$.
This ray has an ideal point $\theta_n$, and since it intersects
the boundary of  $V_n(\xi) \cap L_{t_n}$  it follows  that
$\theta_n$ is in $I_n(t_n)$.

Now let  $n  \to \infty$, so $t_n \to  0$. Up to a subsequence
assume that $\gamma_n$ converges to a subset of $L_0$.
This  set is  made up of minimizing rays and  geodesics.
We first claim that the set is connected: if $x,y$  are in the 
limit of the $\gamma_n$ , 
then there are $x_n, y_n$ in $\gamma_n$ with $x_n \to x,
y_n \to  y$. Since $L_{t_n} \to L_0$, it follows that  
distance  in $L_{t_n}$ between $x_n, y_n$ is bounded, hence
the segments in $\gamma_n$ from $x_n$ to $y_n$ have bounded
length and converge to a geodesic segment connecting $x, y$
in $L_0$. This shows the claim. Note that if we were not
dealing with  minimizing geodesics, but rather general curves,
the limit certainly could be a disconnected union of curves.
In our case the limit is a single 
minimizing  geodesic ray $\gamma$
in $L_0$,
which starts in $\tau(0)$. 
Since $I_n(t)$ converges to $\xi$, it follows as
in Lemma \ref{remark-nested} that the ideal point of $\gamma$ is $\xi$.

Notice that  in $L_{t_n}$ we have the  following:  \ 1) a geodesic
segment, denoted by  $s_n$, contained in $\gamma_n$ from $\tau(t_n)$ to 
a  point at most $a_0$ distant from $p_n$;
 \ 2)  the compact segment contained
in $\partial Z \cap L_{t_n}$ from $\tau(t_n)$ to $p_n$, this
is denoted by $v_n$. 
The second segment is a uniform quasigeodesic in $L_{t_n}$. 
Since $s_n, v_n$ have the  same starting point, and endpoints
at most $a_0$ distant  in $L_{t_n}$, there is global $a_1$ so that
$s_n, v_n$ are at most $a_1$ Hausdorff distant from each
other in $L_{t_n}$.

By the previous lemma it follows that the distance in $L_{t_n}$ from
$\tau(t_n)$ to $p_n$ converges to $\infty$.
This implies that the $v_n$ have subsegments which are converging to a
ray in the boundary of $Z \cap L_0$, let this
ray be  $v$. It follows that $\gamma$ is a bounded distance
from $v$. But $v$ has  ideal point one  of   the endpoints
of  $Z  \cap S^1(L_0)$ and the ideal point of $\gamma$ is $\xi$,
which is not an endpoint of  $Z \cap S^1(L_0)$.

This is a contradiction, and shows  that the assumption
that  $Z$ does  not contain $W_n(\xi)$ for all $n$ is
impossible. 
This finishes the proof that $\ft$ is continuous,
and hence it is a homeomorphism.

This has in particular the consequence that this definition of the topology is independent of the metric one chooses in $M$. 
\end{proof}

We now have some  important consequences:
The previous lemma in particular implies that the topology in $S^1_{\infty}(\wcF)$ coincides with the one introduced in \cite[\S 7.2]{Calegari-book} (which is the same topology, but constructed for a Candel metric, on which all leaves are negatively curved or hyperbolic, and thus all geodesics are minimizing geodesics). 

The  following  remark  will  be essential in future sections:

\begin{remark}  \label{rem-topology}
Given a closed transversal $\tau$ to $\wcF$, let $S^1_\infty(\tau) := \bigcup_{L \cap \tau} S^1(L)$. Then $S^1_{\infty}(\tau)$ has the topology of a closed cylinder. This can either be shown directly, or in the following way:
Apply the previous lemma using a Candel metric (in which all leaves are negatively curved, so $A_p = S^1_p$ for every $p \in \mt$). Then note that in
this metric, the map which sends each $v \in T^1_{\tau(t)}\wcF$ to the endpoint of the geodesic ray starting at $\tau(t)$ with velocity $v$ gives a homeomorphism between a closed topological cylinder and $S^1_\infty(\tau)$. In addition, note that one can find a metric inducing this topology (on the cylinder alone, since the global topology of $\cTu$ may be non-Hausdorff) by considering the Gromov product in each leaf and extending it to the boundary (see \cite[III.H.3]{BridsonHaefliger}).

This will be used in \S\ref{s.SVMQG} to discuss the small visual measure property.
\end{remark}


\section{Pushing through and separation}\label{s.pushthrough} 
Here we will develop one of the main tools we will use to produce geometric properties for the intersected  foliation. 
We prove the following:

\begin{prop}[Pushing Property]\label{prop-pushing}
Let $\cF_1, \cF_2$ be two transverse foliations by Gromov hyperbolic leaves of a closed $3$-manifold $M$ so that for every $E \in \wt{\cF_2}$ we have that the leaf space $\cO_E$ of $\cG_E$ is Hausdorff. Consider $L \in \wt{\cF_1}$ and $E \in \wt{\cF_2}$ and let $\ell_0 = L \cap E$ be the unique leaf of $\wt{\cG}$ in their intersection (cf. Lemma \ref{l.twointersnonHsdff}).
Suppose that $c$ is a compact arc contained
in $\ell_0$ with endpoints $x, y$. Suppose that
$\tau_i: [0,t_0] \to \mt$ are transversals to $\wcF_1$, starting
in $x, y$ both contained in $E$, and
so that for all $t$, $\tau_1(t), \tau_2(t)$ are in the same leaf
$L_t$ 
of $\wcF_1$. Then for each $t$, $\tau_1(t), \tau_2(t)$ are the
endpoints of an arc $c_t$ of $\wcG$ in $L_t \cap E$ and these arcs vary
continuously with $t$.
The same holds with $\wcF_1, \wcF_2$ switched.
\end{prop}

\begin{figure}[ht]
\begin{center}
\includegraphics[scale=0.82]{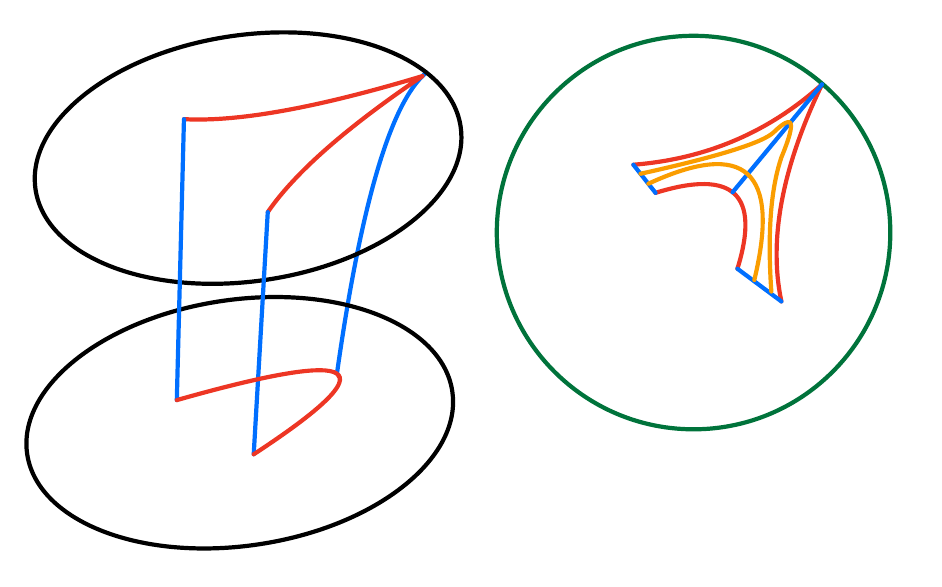}
\begin{picture}(0,0)
\put(-329,30){$L_{0} \in \wcF_1$}
\put(-230,120){$E \in \wcF_2$}
\put(-283,224){$L_{t_1} \in \wcF_1$}
\put(-156,58){$E$}
\put(-100,142){$c_0$}
\put(-84,172){{$c_s$}}
\put(-254,61){{$c_0$}}
\put(-135,160){$\tau_1$}
\put(-82,110){$\tau_2$}
\put(-315,125){$\tau_1$}
\put(-283,100){$\tau_2$}
\end{picture}
\end{center}
\vspace{-0.5cm}
\caption{{\small In the left a piece of leaf $E$ of $\wcF_2$ intersecting two leaves $L_0$ and $L_{t_1}$ of $\wcF_1$. $E$ intersects $L_0$ in a compact arc $c_0$, and $E$ intersects $L_{t_1}$ in two rays that go to infinity. In the right the leaf of $E$ is depicted and showed how this structure of intersection forces the arcs to split and induce non Hausdorff leaf space.}}\label{fig-split}
\end{figure}

In other words given a segment $c$ of $\wcG$ in a leaf of $\wcF_1$,
if we can push the endpoints of $c$ transversely to $\wcF_1$ and
mantaining the endpoints in the same leaf of $\wcF_2$, then
we can push the entire segment $c$ to a new segment $c'$ of a
 leaf of $\wcG$.

We stress that we require the Hausdorff leaf space on leaves of $\wcF_2$ and we obtain the pushing property for arcs of leaves of $\wcG$ in leaves of $\wcF_1$.

\begin{proof}
Let us define

$$V = \{ t \in [0,t_0] \ \text{ so  that }  
\tau_1(t),\tau_2(t) \text{ bound  an  arc } c_t \text{ of }  \wcG
\ \text{ in }  L_t \}, $$
In other words $\tau_1(t), \tau_2(t)$ are in the same
leaf of $\wcG$ in $L_t$ if $t$ is in $V$.
By the local product structure of foliations the set $V$ is
open and by assumption it contains $0$. Clearly the arcs
$c_t$ are unique if they exist (they are contained in $L_t \cap E$)
and vary continuously in $V$.

By way of contradiction suppose that there is $t_1 \leq t_0$ so
that $[0,t_1) \subset V$, but $t_1$ is not in $V$.
Now consider the situation in the leaf $E$ of $\wcF_2$ 
which contains $x, y$. See figure \ref{fig-split}. 
Notice that $E$ contains all $c_t$
with $0 \leq t < t_1$. By assumption $\tau_1(t_1), \tau_2(t_1)$
are in $L_{t_1}$ but are not in the same leaf of $\cG_E$.  In particular, we get that $L_{t_1} \cap E$ has  more than one connected component, and then Lemma \ref{l.twointersnonHsdff} completes the proof of the proposition. 
\end{proof}

\section{Landing of rays}\label{s.landing}
The goal of this section is to show the first step of the  strategy presented in \S \ref{ss.steps}: 

\begin{teo}[Hausdorff implies landing]\label{t.Landing}
Let $\cF_1$ and $\cF_2$ be two transverse foliations  by Gromov hyperbolic leaves of a closed manifold $M$ so that for every $E \in \wt{\cF_2}$ we have that the leaf space $\cO_E$ of $\cG_E$ is Hausdorff, then, for every $L \in \wcF_1$ and every $\ell \in \cG_L$ we have that both rays of $\ell$ land in $S^1(L)$. 
\end{teo}

Again we emphasize on the asymetry of the statement and the fact that we require Hausdorff leaf space along one of the two foliations. As in Proposition \ref{prop-pushing} we need Hausdorfness along $\wcF_2$ to deduce a property of rays in leaves of $\wcF_1$. We remark here that if we assume that both foliations are leafwise Hausdorff, then some parts of the proof can be simplified (see Remark \ref{rem-simplificationHsdff} below). 

The proof of this theorem will occupy all of this section. Before going into
the several intermediate results,
 we roughly explain the overall idea of the proof.
The leaves of $\wcF_1$ are Gromov hyperbolic, if a ray $\ell$ in a leaf
$F$ of $\wcF_1$ does not land, then the region where it limits
to grows exponentially in size, so projecting to $M$ it limits to a
sufficiently big set, and this will have consequences
which will lead to a contradiction. Suppose first that $\pi(F)$ contains
a closed curve and lift it to a curve $\beta$ in $F$ periodic
under a non trivial deck transformation $\gamma$. Suppose that
one of the endpoints of $\beta$ is in the interior of the limit
set of $\ell$, so $\beta$ keeps intersecting $\ell$. Suppose there
are two intersections $x,y$ so that the segment between them 
$[x,y]$ in a leaf of $\cG_F$ does not intersect $\beta$. Suppose
that the length of $[x,y]$ is very big so that $x, \gamma(x), y,
\gamma(y)$ are lined in $\beta$. Then $[x,y]$ and $\gamma([x,y])$
intersect transversely, which contradicts that $\cG_F$ is a foliation
in $F$. In general $F$ does not have such a $\gamma$, but we can
get a geodesics in $F$ intersecting $\ell$ in more and more points,
so that deck translates get closer and closer to such a 
$\beta$ $-$ this is what we were alluding to in the remark above that
$\ell$ keeps going about a region that grows exponentially. 
Using the push through result, Proposition \ref{prop-pushing},
we can push these intersections and arcs in between
to a leaf $F'$ having such a deck transformation $\gamma$,
and eventually obtaining a contradiction. There is a further 
case, which we will get into eventually.

The proof of the theorem
will proceed by contradiction, separating the proof in two cases. To set up the proof, we consider $\ell_0 =L \cap E$ a leaf of $\wcG$ where $L \in \wcF_1$ and $E \in \wcF_2$ and we assume that $\partial^+\ell_0$ in $L$ is not a single point (the proof is analogous if $\partial^- \ell_0$ is not a point). 

It follows that $\partial^+ \ell_0$ is either all of $S^1(L)$ or a proper non-trivial interval. In any case we can consider $I \subset S^1(L)$ a proper non trivial closed interval such that $I \subset \partial^+\ell_0$. We can also consider $J \subset I$ a non trivial closed interval contained in the interior of $I$. 

We let $H \subset L$ to be the half space in $L$ bounded by a geodesic in $L$ joining the endpoints of $J$ and whose closure in $L \cup S^1(L)$ contains $J$. Given a half space $H$ in a leaf $L$, we consider $H_n \subset H$ to be the set of points at distance $\geq n$ of $\partial H$, see figure \ref{fig-Hn}. We define: 

\begin{equation} \mathrm{Lim}(H) = \bigcap_{n>0} \overline{\bigcup_{\gamma \in \pi_1(M)} \gamma H_n}.
\end{equation}

\begin{figure}[ht]
\begin{center}
\includegraphics[scale=0.82]{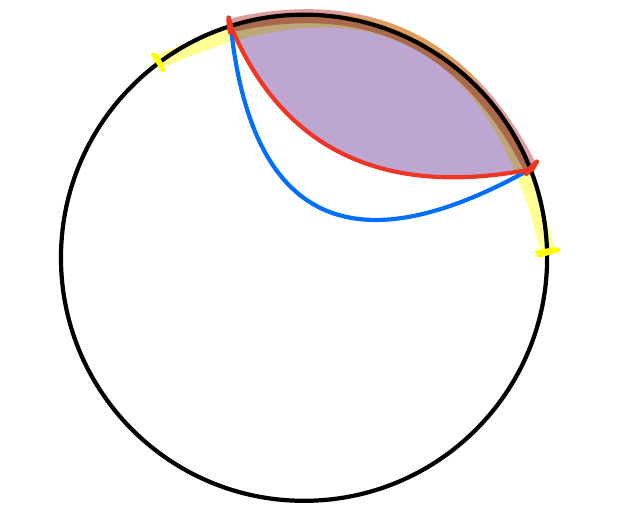}
\begin{picture}(0,0)
\put(-100,110){$\ell$}
\put(-64,181){{$J$}}
\put(-125,165){$H_n$}
\put(-30,118){$I$}
\end{picture}
\end{center}
\vspace{-0.5cm}
\caption{{\small A depiction of the sets $H_n$ inside $L$.}}\label{fig-Hn}
\end{figure}

This is a closed $\pi_1(M)$-invariant set saturated by leaves of $\wcF_1$. To see that $\mathrm{Lim}(H)$ is saturated by leaves one only needs to use the fact that leaves of $\wcF$ vary continuously in the topology of convergence in compact sets, as follows:
If $p \in \overline{\bigcup_{\gamma \in \pi_1(M)} \gamma H_n}$, then
there are disks of radius $n$  in the union which  converge to the
the disk of  radius $n$ around $p$,  so the disk of radius $n$ 
about $p$ is contained in the  closure. Since  this is  true for
all $n$, the entire leaf  of $p$  is contained in $\mathrm{Lim}(H)$.

If $\pi: \mt \to M$ is the universal cover projection, $\mathrm{Lim}(H)$ corresponds to the preimage by $\pi$ of the accumulation set of $\pi(H_n)$ as $n\to \infty$. 

\begin{lema}\label{lem.stabilizerinlimit}
Let $\mathrm{Lim}(H)$ be the accumulation set of $H$ as defined above and let $\Lambda$ be any minimal $\pi_1(M)$-invariant sublamination of $\mathrm{Lim}(H)$. Then, $\Lambda$ has some leaves which have non-trivial stabilizer (i.e. there is some $L_1 \in \Lambda$ such that there is $\gamma \in \pi_1(M) \setminus  \{ \mathrm{id}\}$ such that $\gamma L_1  = L_1$). 
\end{lema}

\begin{proof}
Since $\mathrm{Lim}(H)$  is clearly closed and $\pi_1(M)$-invariant by definition, it is a sublamination of $\wcF$.

The rest now follows directly from Theorem \ref{teo.nonholonomy}. 
\end{proof}

We now fix a minimal sublamination $\Lambda$ of $\mathrm{Lim}(H)$ and consider a leaf $L_1 \in \Lambda$ which is invariand under some $\gamma \in \pi_1(M) \setminus \{ \mathrm{id}\}$. Let $\alpha \in L_1$ be a minimizing geodesic in $L_1$ invariant under $\gamma$ (see Lemma \ref{lem-closedgeodesics}). 

We will use the following result that also follows from the proof of the existence of sawblades and marker directions in the proof of the leaf pocket theorem in \cite[\S 5]{CD}. Since it is a simple argument we give a proof for the convenience of the reader not familiar with \cite{CD}.  We recall here our standing assumption that all foliations are orientable and transversally orientable.

\begin{lema}\label{lem.liftwithbothendpointsinH}
For every $\eps>0$ there exists a sequence $\eta_n \in \pi_1(M)$ and quasigeodesics $\beta_n \subset L$ with one endpoint in $J \subset S^1(L)$ and such that $\alpha$ and $\beta_n$ contain  rays $r_n \subset \alpha$ and $s_n \subset \beta_n$ such that $r_n$ and $\eta_n s_n$ are at Hausdorff distance less than $\eps$ in $\mt$. Moreover, either $s_n$ limits in a point of $J$ for infinitely many $n$ or the starting points $x_n$ of the rays $s_n$ converge to a point in $J$. 
\end{lema}

See figure \ref{fig-limitleaf2}. We note that the conditions in the last assertion need not be exclusive. 

\begin{proof}
Since the deck traslates of $H_n$ accumulate in $L_1$ by construction, we can find a sequence of deck transformations $\eta_n$ and points $p_n \in H$ so that the distance in $L$ from $p_n$ to the boundary of $H$ goes to infinity and such that $\eta_n p_n \to p_\infty \in \alpha \subset L_1$.  Since the distance of $p_n$ to the boundary of $H$ grows, we know that $p_n$ converges  in  $L \cup S^1(L)$, up to subsequence, to a point in the circle at infinity.
This point is in  $J$, and we denote this by $p_n \to J$ \  (in $L \cup S^1(L)$). 

Pick a short transversal $\tau$ to $\wcF_1$ through $p_\infty$,  and consider $\tau_k = \gamma^k \tau$. Let $\tau^+$ be a half interval of $\tau$ which intersects infinitely many translates $\eta_n L$. Up to taking subsequences, we can assume it intersects all the $\eta_n L$. Let $\tau^+_k$ denote $\gamma^k \tau^+$.

Since $\alpha$ is $\gamma$-invariant then up to taking $\gamma^{-1}$ we can assume that the holonomy from $p_\infty$ to $\gamma p_\infty$ maps of $\tau^+$ inside $\tau_1^+$. It follows that the ray $\hat r$ of $\alpha$ starting at $p_\infty$ and containing $\gamma^n p_\infty$ for all $n>0$ can be lifted by holonomy to a curve $\hat s_n$ in $\eta_n L$ which must be a uniform quasigeodesic ray in $\eta_n L$,  and $\hat s_n$ is always close to $L_1$. One can complete $\hat s_n$ to a full quasigeodesic $\hat \beta_n$ in $\eta_n L$. 
Since the point $p_n$ in $L$ is very far in $L$ from the boundary of $H$, we get that for large $n$ one of the endpoints of $\beta_n = \eta_n^{-1} \hat \beta_n$ must be in $J$ since $s_n= \eta_n^{-1} \hat s_n$ has starting point $p_n$ and $p_n \to J$.

\begin{figure}[ht]
\begin{center}
\includegraphics[scale=0.72]{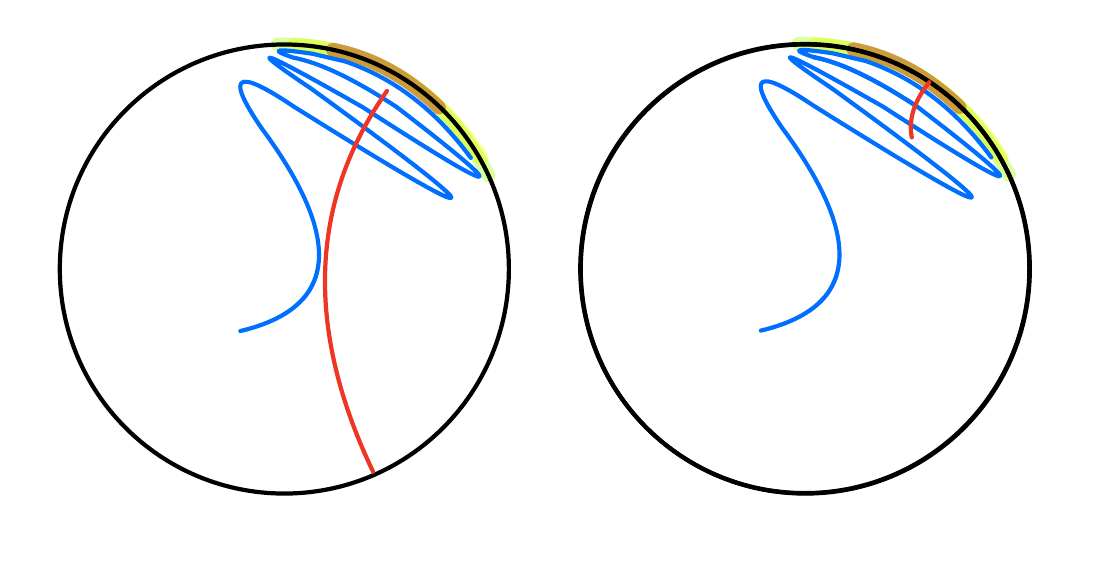}
\begin{picture}(0,0)
\put(-90,100){$\ell_0$}
\put(-259,63){$s_n$}
\put(-293,94){$\ell_0$}
\put(-64,152){$s_n$}
\put(-67,174){$J$}
\put(-239,174){{$J$}}
\end{picture}
\end{center}
\vspace{-0.5cm}
\caption{{\small The possibilities for the ray $s_n$ inside $L_1$.}}\label{fig-limitleaf2}
\end{figure}

Fix some $\eps>0$. By taking $n$ sufficiently large we get that $\hat s_n$ and $\hat r$ are at Hausdorff distance less than $\eps$ in $\mt$. Also, by construction we know that either $s_n$ limits in a point in $J$, or the initial points of $s_n$ (which are $p_n$) converge to a point  in $J$. This completes the proof. 
\end{proof}

In particular as $n$ grows, the number of intersections of 
$\beta_n$ with $\ell_0$ grows without bound. This is because
$\ell_0$ limits in the whole  closed  interval $I$ and $p_n$ converges
to a point  in the interior of $I$.

Using the previous lemma we will produce many segments in $L_1$ that will later allow us to produce a contradiction. For this, we will consider $\eps$ much smaller than the size of local product structure boxes and consider $\eta_n, p_n, s_n$ given by the previous lemma, in particular so that the ray $s_n$ verifies that $\eta_n s_n$ is very close to $\alpha$.  We recall that $\ell_0 = L \cap E$ is connected and is a curve such that $\partial^+ \ell_0$ contains $I$ which has $J$ in its interior by assumption. We denote by 

$$\ell_n \ = \ L_1 \cap \eta_n E$$ 

\noindent
(note that the intersection is non-empty because there are points of $\eta_n(\ell_0)$ in $\eta_n(s_n)$ which is very close to $\alpha$ for large $n$). 
In addition the intersection is connected for all $n$ because
of Hausdorff leaf space of $\wcG$ in $\eta_n E$.

Denote $C,D$ to be the closure of the connected components of $L_1 \setminus \alpha$.

Recall that for points $x,y \in \ell \in \wcG$ we denote by $[x,y]$ the closed segment joining $x$ and $y$.  See figure \ref{fig-l3}.

\begin{lema}\label{lema.findingcurves}
There are constants $0<a_0<a_1<a_2$ so that for every $N>0$, there is $k_0$ (which depends only on $N$) so that if $k\geq k_0$ then, there are points:
\begin{enumerate}
\item $x_k,y_k \in \alpha \cap \ell_k$ such that $[x_k,y_k] \subset C$ and $[x_k,y_k]$ contains a point at distance (in  $L_1$) larger than $N$ from $\alpha$. Moreover, the intersection points $\{x_k, y_k\} = [x_k,y_k] \cap \alpha$ are contained in a sub-interval of $\alpha$ of length $\leq a_2$ and are at distance $\geq a_1$ from each other.  
\item $w_k,z_k \in \alpha \cap \ell_k$ such that $[w_k,z_k] \subset D$  and $[w_k,z_k]$  contains a point at distance (in $L_1$) larger than $N$ from $\alpha$. Moreover,the intersection points $\{w_k,z_k\} =[w_k,z_k] \cap \alpha$ are contained in a sub-interval of $\alpha$ of length $\leq a_2$ and are at distance $\geq a_1$  from each other.  
\end{enumerate}
For every $m_0$ sufficiently large there is $k$ and $m>m_0$ such that there are points $p$ and $q$ in $\ell_k \cap \alpha$ such that $d(\gamma^m p, q) < a_0$. 
\end{lema}

\begin{figure}[ht]
\begin{center}
\includegraphics[scale=0.92]{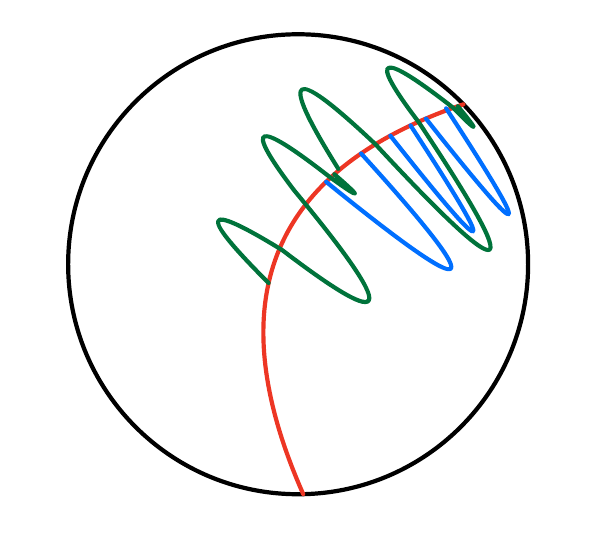}
\begin{picture}(0,0)
\put(-160,80){$\alpha$}
\put(-145,113){$p$}
\put(-92,85){$C$}
\put(-192,85){$D$}
\put(-126,171){$q$}
\put(-138,211){{\small $\gamma^m [p,q]$}}
\put(-89,116){{\small $[x_k,y_k]$}}
\end{picture}
\end{center}
\vspace{-0.5cm}
\caption{{\small Depiction of some objects in Lemma \ref{lema.findingcurves}.}}\label{fig-l3}
\end{figure}

\begin{proof}
We first determine the bounds $a_1$ and $a_2$. We  can cover $\alpha$ by a locally finite family of local product structure boxes 
of $\wcG$ in $L_1$ (which we can choose to be $\gamma$-periodic). A leaf $\ell$ of $\wcG$ cannot intersect the same box twice, since otherwise one can produce a transversal to $\wcF_2$ intersecting a leaf more than once, contradicting Theorem \ref{teo.Novikovetc}. It follows that if a segment $[x,y]$ of a leaf intersects $\alpha$ exactly twice and contains points far from $\alpha$ 
in $L_1$ then the endpoints must be at some distance bounded from below by some constant $a_1$ related to the covering we chose.

Now, let $a_2$ be much larger than the translation distance of $\gamma$ along $\alpha$.
Assume we have an arc $[x,y]$ of a leaf of $\wcG$ so that it has both endpoints in $\alpha$ and is contained in either $C$ or $D$ intersecting $\alpha$ only at the endpoints. Then we have that $[x,y]$ together with the arc $c$ of $\alpha$ joining $x$ and $y$ forms a Jordan curve. This implies that $\gamma x$ cannot be in the interior of the arc $c$, because if that were the case one would produce an intersection between $[x,y]$ and $\gamma [x,y]$  (recall that our standing assumption is that everything is orientable and transversally orientable, thus $\gamma$ preserves $C$ and $D$). Similarly, for $\gamma^{\pm} y$ and $\gamma^{-1}x$. We deduce that the length of $c$ must be smaller than $a_2$ showing the upper bound. 

We now fix some large $N$ and produce the arcs $[x_k,y_k]$ inside $\ell_k$ for sufficiently large $k$. The construction of $[w_k,z_k]$ is completely analogous. We first choose $k_0$ large so that for every $k>k_0$ we have that the ray $s_k \subset L$ created in Lemma \ref{lem.liftwithbothendpointsinH} verifies that it intersects $\ell_0$ in points $[\hat x_k, \hat y_k]$ and so that the interior of the arc is contained in the component which is close to $C$ after applying $\eta_k$. Since $\ell_0$ accumulates in all of $I$ 
(with $J$ contained in the interior of $I$), and the initial point of $s_k$ is very close to a point in $J$ we can assume, by taking $k$ large enough, that the arc $[\hat x_k, \hat y_k]$ has points at arbitrary large distance in $L$ from $\beta_k$.

We can now apply Proposition \ref{prop-pushing} to push the segment $\eta_k [\hat x_k, \hat y_k]$ to a segment $[x_k, y_k]$ (after maybe cutting the boundaries to obtain that the arc only intersects $\alpha$ in the boundaries). 
We emphasize that this is a crucial application of the Hausdorff property
of $\wcG$  in $\wcF_2$ leaves.
To show that this segment has the desired properties, we need to show that it has points at distance larger than $N$ from $\alpha$. Note that $\eta_k \beta_k$ and $\alpha$ are close only in a ray $\alpha_k$ of $\alpha$ (the one close to $\eta_k s_k$) and so it could be that the arc $[x_k, y_k]$ remains at distance less than $N$ from $\alpha$. Note however that we can assume that the length of $[x_k,y_k]$ is as large as desired (it goes to infinity with $k$). Suppose that
for some fixed $N$ the segments $[x_k, y_k]$ never exits the $N$ neighborhood
of $\alpha$. Project to $M$ via $\pi: \mt \to M$ the universal covering projection. The $[x_k,y_k]$ cannot project
to a closed curve because the endpoints are in $\alpha$ but
the interior does not intersect $\alpha$.
But the arcs $[x_k, y_k]$ project to longer and longer segments
in the $N$ neighborhood of a closed curve. 
Take the midpoints of these segments and a convergent subsequence 
of midpoints: this shows that 
there is an entire leaf
contained in this neighborhood. This leaf must spiral 
towards a closed leaf. In fact this forces the existence of
a Reeb annulus: the boundary circles are leaves and the interior
leaves are lines spiraling towards the boundary on both directions
to produce a foliation with non Hausdorff leaf space in the
annulus. This would force the projection of $[x_k, y_k]$ to
be entirely contained in this annulus for $k$ sufficiently big.
 This is incompatible with the segment $[x_k, y_k]$
returning to intersect $\alpha$.

Finally, fix $a_0\ll a_1$ and we note that if $k$ is large enough then $\ell_k$ must have more than $a_2/a_0$ intersections with $\alpha$. 
This is possible as was remarked after the proof of 
Lemma \ref{lem.liftwithbothendpointsinH}. Taking the intersections to a fundamental domain by applying $\gamma^m$ we get pairs of points at distance less than $a_0$ giving the last property. 
\end{proof}

\begin{remark}\label{rem-simplificationHsdff}
Note that if we assume that $\wcG$ is Hausdorff (instead of just asking that the leaf spaces restricted to only leaves in $\wcF_2$ to be Hausdorff as we do in Theorem \ref{t.Landing}), Proposition \ref{prop-Hsdff2Dand3D} implies that
leaf spaces of $\cG_L$ for $L$ in $\wcF_1$ are also  Hausdorff.
This contraditcs the first item of the previous lemma.
This is all that is needed for the proof of the landing property in the setting of Theorem \ref{teo.main}, so the remainder of this section can be skipped by someone only interested in the proof of Theorem \ref{teo.main}. 
\end{remark}

Now we can complete the proof of Theorem \ref{t.Landing}:

\begin{proof}[Proof of Theorem \ref{t.Landing}]
We must show that the previous Lemma gives a contradiction. First, fix $k_0$ and points $p,q \in \ell_{k_0} \cap \alpha$ so that for some $m$, 
$d(\gamma^m p, q) < a_0$. Pick a small arc $u$ of $\alpha$ from $\gamma^m p$ to $q$, it follows that 
$u \cup [p,q]$ (where $[p,q] \subset \ell_{k_0}$) verifies that it projects in $L_1/_{<\gamma^m>}$ to a simple closed curve, equivalently, $\bigcup_{j} \gamma^{jm} (u \cup [p,q])$ provides a simple curve $\beta$ at finite Hausdorff distance from $\alpha$, say distance less than $N_0$. Denote $\hat C$ and $\hat D$ the connected components of $L_1 \setminus \beta$ such that $\hat C$ is contained in the $N_0$-neighborhood of $C$ and $\hat D$ in the $N_0$-neighborhood of $D$ in $L_1$. 

Since $\gamma^m p$ and $q$ are in a foliated box,
we can assume that the foliation $\cG_{L_1}$ is transverse to $\alpha$ in $u$ and so, every leaf of $\cG_{L_1}$ intersects $u$ with the same orientation,
which we assume is from $\hat C$ to $\hat D$. This happens for
any $\gamma^{jm}(u)$, and it follows that leaves of $\wcG$ in $L$
intersecting $\beta$ either cross from $\hat C$ to $\hat D$ or
follow along  $\gamma^{jm}([p,q])$ for  a while  exiting
to  $\hat D$ in the  future  and  $\hat D$ in  the past.
Hence the intersection of  a leaf  of $\wcG$ with $\beta$
is  connected,
and every $\ell_{k}$ is either contained in $\hat C$, contained in $\hat D$, or has one ray contained in $\hat C$ and one contained in $\hat D$.

Since given $N>N_0$ fixed we can, using Lemma \ref{lema.findingcurves}, find arcs of the form $[x_k,y_k]$ on the ray contained in $\hat D$ and arcs of the form $[w_k,z_k]$ contained in $\hat C$. But by choice, this implies that either $[x_k,y_k]$ is contained in a $N_0$-neighborhood of $D$ or $[w_k,z_k]$ is contained in a $N_0$-neighborhood of $C$ contradicting the fact that they contain points at distance larger than $N$ from $\alpha$. This contradiction proves the result. 
\end{proof}
 
\begin{remark}\label{rem-reg}(On regularity) For the proofs of this section
we used the hypothesis on $\cF_1$, namely $C^{0,1+}$ foliation with
Gromov hyperbolic leaves. However we did not use any hypothesis on
$\cF_2$, so the result of landing of rays
 still works for $\cF_2$ with only topological
leaves, and $\cF_1$ with the other hypothesis.  
The same holds for the push through result of 
Section \ref{s.pushthrough}, and also for the small visual
measure property in section \ref{s.nonvisual}, 
where in both cases it only applies to $\cF_1$, that
is small visual  measures of certain  arcs of $\wcG$
in leaves of $\wcF_1$, and similarly for the push through
property.
\end{remark}

\section{A relevant example}\label{s.example}

Before we continue with the proof of 
Theorem \ref{teo.main} we will present an enlightening example that shows how the non-solvability of the fundamental group of $M$ must enter into the proof. Formally, this section is not needed for the proof of Theorem \ref{teo.main} and can be skiped.

Let $\Phi_1$ be a suspension of a linear hyperbolic 
automorphism of the torus. We assume that the weak foliations $\cW^{ws}_1, \cW^{wu}_1$ of $\Phi_1$ are transversely orientable.

Now perturb $\Phi_1$ to $\Phi_2$ as follows:
fix $\theta > 0$ small, and turn the tangent vector of $\Phi_1$
by an angle of $\theta$ still keeping it tangent to $\cW^{wu}_2$. 
Let $\Phi_2$ be the resulting flow. If $\theta$ is small enough,
then $\Phi_2$ is Anosov, and in fact orbitally equivalent to
$\Phi_1$ by the structural stability of Anosov flows. Denote by $\cW^{ws}_2, \cW^{wu}_2$ the weak foliations of $\Phi_2$. Note that by construction we have that $\cW^{wu}_1 = \cW^{wu}_2$ because the weak unstable bundle of $\Phi_1$ is $\Phi_2$-invariant and thus must be one of the invariant bundles of $\Phi_2$. By continuity it must be the weak unstable one. However, the weak stable foliation changes, and in fact, since the vector field tangent to $\Phi_2$ is everywhere transverse to $\cW^{ws}_1$ by construction, we get that $\cW^{ws}_1$ and $\cW^{ws}_2$ are transverse (minimal) foliations.  




We have the following properties: 

\begin{prop}\label{prop-example} 
The foliations $\cW^{ws}_1$ and $\cW^{ws}_2$ are transverse minimal foliations by Gromov hyperbolic leaves. Moreover, the intersected  foliation $\cG$ coincides with the strong stable foliation of $\Phi_1$ and $\Phi_2$ which is Hausdorff but is not by quasigeodesics inside the corresponding leaves. 
\end{prop}

\begin{proof}
This construction has a product structure: one can 
start with $x, y$ coordinates in the plane and a linear
matrix $A$ so that the flow is associated with $A$
and the vertical direction is $t$. If this is the case for $\Phi_1$, then, the flow $\Phi_2$ is then the suspension of the affine automorphism $\hat A$ given by $\hat A p = A p + v$ where $v$ is a vector in $E^u_A$ (the unstable direction of the matrix $A$). 

The product is with
respect to both $x$ and $y$.
In particular the intersection of a weak stable leaf of
$\wwp_1$ and a weak stable leaf of $\wwp_2$ is a horocyle
in the weak stable leaf of $\wwp_1$ and similarly a horocycle
in the weak stable leaf of $\wwp_2$.
More precisely if one considers a horocycle in a 
weak stable leaf of $\wwp_1$, then 
flowing time $t$ under $\wwp_2$ one sees that points get
exponentially close $-$ because of the product structure.
Hence this is also a horocycle in a weak stable leaf of $\wwp_2$.
\end{proof}

Note that this proof is very specific to suspension flows, and it is the reason we added the assumption of $M$ not having solvable fundamental group in the statement of Theorem \ref{teo.main}. If one does the same construction starting with a general Anosov flow in a $3$-manifold and perturbing the flow along the weak unstable foliation one will obtain that the weak stable foliations of the original and perturbed flow are transverse, but it is no longer true that their intersection will happen in horocycles as follows from our main result.\footnote{Heuristically, if one looks at the geodesic flow in the unit tangent bundle of a hyperbolic surface and one perturbes slightly the vector field tangent to geodesics to be outside the weak-stable foliation, the new flow will still be Anosov and its weak stable foliation will intersect the former one in curves which have constant geodesic curvature in the weak-stables, and less than the curvature or horocycles, thus, will be quasigeodesics.}

\section{Small visual measure and the quasigeodesic property}\label{s.SVMQG}
We recall here the notion of small visual measure, its properties, and its relation with the quasigeodesic property. 

First, we give the precise definition and fix our setup. Fix a foliation $\cF$ by Gromov hyperbolic leaves of a closed 3-manifold $M$ and consider a leaf $L \in \wcF$ in the universal cover. Given $x \in L$ and $K \subset L$ we define the shadow $\mathrm{Sh}_x(K,L)$ to be the set of points $\xi$ in $S^1(L)$ for which there is a minimizing geodesic ray from $x$ to some point in $K$ that lands in $\xi$.  

If the metric in $M$ makes all leaves to have negative curvature, 
given $x \in L$ and an interval $I \subset S^1(L)$ we define the \emph{visual measure} of $I$ from $x$ to be the length of the interval of $T^1_x L$ consisting of unit vectors $v$ for which the geodesic ray from $x$ with initial condition $(x,v)$ lands in some point in $I$.  Note that this is not canonical, but the notion we shall define is independent of this choice (see \cite[\S 2.5,4.3]{FP2}). 

The general case of Gromov hyperbolic leaves requires using the definition of visual measure using the Gromov product. Recall that a visual metric on $S^1(L)$ with parameter $a>1$ seen from $x$ is a distance $d_x$ in $S^1(L)$ such that there is a constant $k>0$ so that for $\xi,\xi' \in S^1(L)$ it verifies: 

\begin{equation}\label{eq:visualmeasure} k^{-1} a^{-(\xi | \xi')_x} < d_x(\xi,\xi') < k a^{-(\xi|\xi')}. \end{equation}

\noindent where $(\xi|\xi')_x$ denotes the Gromov product seen from $x$ and defined as $(\xi|\xi')_x = \liminf_t \frac{1}{2}( d_L(x, x(t)) + d_L(x, x'(t)) - d_L(x(t),x'(t)))$ where $x(t)$ and $x'(t)$ are geodesic rays landing respectively at $\xi$ and $\xi'$. Such metrics exist in any Gromov hyperbolic space (see \cite[III.H.3]{BridsonHaefliger}) for values of $a$ that depend only on the hyperbolicity constant of the spaces. For us, we will just pick one with a fixed given parameter and only use some coarse properties of the metric, so that the property \eqref{eq:visualmeasure} is more than enough.

\begin{defi}
We say that a one dimensional subfoliation $\cT$ of $\cF$ has the \emph{small visual measure} property if for every $\eps>0$ there is some uniform constant $R>0$ such that if $x\in L \in \wcF$ and $c \subset \ell \in \wt{\cT}$ is a segment of leaf contained in $L$ such that $d_L(x,c) >R$, then the shadow $\mathrm{Sh}_x(c,L)$ has visual measure smaller than $\eps$. 
\end{defi}

Note that the small visual measure property is strictly stronger than the landing property. 
On the one hand, if $\cT$ has the small visual measure property in $\cF$ it follows that any ray $r$ of $\cT$ in a leaf must land in its corresponding leaf
because of the following argument:  parametrize $r$ by arclength and consider $r_n$ the subray of $r$ so that the segment between the starting points has length $n$. Then,  since $r$ is proper, the ray $r_n$ is at distance going to infinity from the starting point of $r$ as $n \to \infty$. By the small visual measure property, this implies that the closure of $r_n$ in $S^1(L)$ is contained in $r_n$ plus an interval of $S^1(L)$ of small visual measure, with measure
going to $0$ when $n$ goes to infinity. Therefore the limit set of $r$
can only be a singleton, and $r$ lands.  On the other hand, one can make an example on which all leaves land, but for which the small visual measure property fails, this is given for instance by the horocyclic foliation of an Anosov flow, for which all rays land in their corresponding leaves, but it does not have the small visual measure property. Note that this example verifies that the leaf space of $\widetilde \cT$ is Hausdorff. 

All along, when $\cT$ is a subfoliation of $\cF$ we will implicitly assume that $\cT$ is by $C^1$-leaves and tangent to a continuous vector field (this includes the orientability assumption) so that notions such as length make sense. 

A relevant implication of small visual measure is the fact that segments of curves of the foliation 
contain geodesic segments joining their endpoints in a bounded neighborhood.
More specifically  the following is proved in \cite[Lemma 5.9]{FP2} and \cite[Proposition 5.2]{FP4}. 

\begin{prop}\label{prop-charSVM} 
If $\cT$ has the small visual measure property in $\cF$, it follows that there is a constant $a_0>0$ such that if $c \subset \ell \in \wt{\cT}$ is a compact segment inside a leaf $L \in \wcF$ then we have that any geodesic segment joining the endpoints of $c$ is contained in $B_{a_0}(c) = \{z \in L \ : \ d_L(z,c) < a_0 \}$.  
\end{prop}

Proposition \ref{prop-charSVM} says that small visual measure property is \emph{half} of what one needs to prove to get the uniform quasigeodesic property for the foliation. More precisely, we have the following characterization of being uniformly quasigeodesic that requires the property ensured by the previous proposition plus a symmetric one. Recall that, for $L \in \wcF$ we denote by $\cT_L$ to the restriction of the foliation $\wt{\cT}$ to the leaf $L$.  

\begin{prop}\label{prop-charQI}
A one dimensional subfoliation $\cT$ of $\cF$ is leafwise (uniformly) quasigeodesic\footnote{Here we are assuming that length in leaves of $\cT$ is measured by arclength by choosing a vector field tangent to the leaves with unit size with respect to the metric. Recall that $\cT$ is by $C^1$-leaves tangent to a vector field.} if and only if there is a constant $a_1>0$ such that for every $L \in \wcF$ and every compact segment $c \subset \ell \in \cT_L$  we have that if $g_c$ is a geodesic ray in $L$ joining the endpoints of $c$, then the Hausdorff distance between $c$ and $g_c$ in $L$ is less than $a_1$, more precisely: 
\begin{itemize}
\item $g_c \subset B_{a_1}(c)$, and, 
\item $c \subset B_{a_1}(g_c)$. 
\end{itemize}
\end{prop}

As mentioned, one of the conditions (namely that $g_c \subset B_{a_1}(c)$) is guaranteed by the small visual measure property. 

\begin{proof}
The direct implication is the classical Morse Lemma for Gromov hyperbolic spaces (see e.g. \cite[Theorem III.H.1.7]{BridsonHaefliger}). For the converse, see \cite[Proposition 7.9]{FP2} (see also \cite[\S 6]{ChandaFenley}).  
\end{proof}


We can now state the main result of this section.
It says that in our setting, the Hausdorff property plus the
small visual measure property implies the uniform quasigeodesic
property:

\begin{prop}\label{prop.svmimpliesqg} 
Let $\cT$ be a subfoliation of $\cF$ so that leaves of $\cT$ have the small visual measure property. Assume moreover that for every $L \in \wcF$ we have that the leaf space of $\cT_L$ is Hausdorff. Then, the foliation $\cT$ is by uniform quasigeodesics in $\cF$. 
\end{prop}

An important consequence of the Hausdorff hypothesis
is that one can take limits of leaves and get some results on these limits:

\begin{lema}\label{lema-limitsvm} 
Let $\cT$ be a one dimensional subfoliation of $\cF$ with the small visual measure property and consider a sequence $x_n \to x_\infty$ in $\mt$. Let $\ell_n, L_n$ be respectively the leaves of $\wt{\cT}$ and $\wcF$ containing $x_n$ and let $\ell_\infty, L_\infty$ be the leaves of $\wt{\cT}$ and $\wcF$ containing $x_\infty$. If the set of leaves of $\wt{\cT}$ on which the leaves $\ell_n$ limit inside $L_\infty$ contains more than one leaf, then $\cT_{L_\infty}$ does not have Hausdorff leaf space. 
\end{lema}

\begin{proof} 
Let $\ell_x$ be the leaf of $\wt{\cT}$ through $x_\infty$ and assume that $\ell_n$ has points $y_n$ so that $y_n \to y_\infty \in L_\infty$ so that the leaf $\ell_y \in \wt{\cT}$ containing $y_\infty$ is different from $\ell_x$. 
Let $L_n$ be the leaf of $\wcF$ containing $\ell_n$.

Assuming that $\cT_{L_\infty}$ has Hausdorff leaf space, then, there is a transversal $\tau: (-\eps,1+\eps) \to L_\infty$ to $\cT_{L_\infty}$ so that $\tau(0)= x_\infty$ and $\tau(1)= y_\infty$. We can extend the transversal $\tau$ to a disk $\cD: (-\eps,1+\eps) \times (-\delta,\delta) \to \mt$ which is everywhere transverse to $\wt{\cT}$ and so that for a given $s \in (-\delta,\delta)$ we have that $\cD(t,s)$ belongs to the same leaf of $\wcF$ for every $t \in (-\eps,1+\eps)$, in particular the curve $t \mapsto \cD(t,s)$ is transverse to the foliation $\wt{\cT}$ restricted to the corresponding leaf of $\wt{\cF}$. Since $x_n,y_n$ converge to $x_\infty$ and $y_\infty$, we deduce that the leaf $\ell_n$ intersects the image of $\cD$ in points $\hat x_n, \hat y_n$ close to $x_n$ and $y_n$ respectively. This implies that $\ell_n$ intersects the disk $\cD$ twice. Hence in $\cD$ there is a segment in $L_n$,
a transversal $\beta$ to 
$\cT_{L_n}$, which joins $\hat x_n$ and $\hat y_n$.
It follows that $\cT_{L_n}$ is a one dimensional
foliation of the plane $L_n$, and $\beta$ is a transversal
from the  leaf of $\cT_{L_n}$ through $\hat x_n$
to itself. This is impossible.
\end{proof}


Now we are ready to prove Proposition \ref{prop.svmimpliesqg}:

\begin{proof}[Proof of Proposition \ref{prop.svmimpliesqg}]
We know by Proposition \ref{prop-charSVM} that there is $a_0>0$ such that if $L \in \wcF$ and $c \subset \ell \in \cT_L$ is a compact segment and $g_c$ a geodesic joining its endpoints, then $g_c \subset B_{a_0}(c)$. We must show that there is another constant $a_1>0$ such that $c \subset B_{a_1}(g_c)$ and the proposition will follow from Proposition \ref{prop-charQI}. 

We assume by contradiction this is not the case, so we can construct a sequence $c_n=[x_n,y_n]$ of compact segments of curves $\ell_n \in \wt{\cT}$ such that if $\ell_n$ is contained in $L_n \in \wcF$ then we have that there is a geodesic segment $g_n \subset L_n$ joining $x_n,y_n$ such that $c_n \notin B_n(g_n)$, that is, there is a point $z_n \in c_n$ which is at distance larger than $n$ from $g_n$. 

Note that $g_n \subset B_{a_0}(c_n)$ by the small visual measure property (cf. Proposition \ref{prop-charSVM}) and so, we can pick points $x'_n, y_n' \in [x_n,y_n]$ so that the following properties are verified: 

\begin{itemize}
\item $z_n \in [x'_n,y'_n]$, 
\item $x'_n, y'_n \in B_{a_0}(g_c)$, 
\item $d_L(x'_n,y'_n) \leq 3 a_0$.  
\end{itemize}

To construct such points, given $n$ one can subdivide $g_n$ into finitely many points at distance less than $a_0$ from each other, and there are points $x_n=q_0, q_1, \ldots, q_k = y_n$ in $[x_n,y_n]$ which are $a_0$ close
(that is $d_L(q_i, g_c) < a_0$) to such points (in order). One gets that $d_L(q_i, q_{i+1})\leq 3a_0$,  and the $q_i$ can be chosen so that the union of the intervals $[q_i, q_{i+1}]$ covers $[x_n,y_n]$. Thus, one can find consecutive ones $q_i, q_{i+1}$ so that $z_n \in [q_i,q_{i+1}]$.
Now let $x'_n = q_i, y'_n = q_{i+1}$, then all the conditions are verified. 

Up to composing with deck transformations and taking subsequences we can assume that $x_n' \to x_\infty$ and $y_n' \to y_\infty$. Since
$d_{L_n}(x'_n, y'_n) \leq 3 a_0$ it follows that
$x_\infty, y_\infty$ are in the same leaf $L_\infty \in \wcF$
(here $L_n$ is the leaf containing $x'_n$). Now, since the length of the arc $[x_n',y_n']$ must go to infinity, this implies that the leaves of $\cT_{L_\infty}$ containing $x_\infty$ and $y_\infty$ cannot coincide. Using Lemma \ref{lema-limitsvm} we conclude. 

\end{proof}


\section{Non small visual measure implies all bubble leaves}\label{s.nonvisual}

We now return to the setting of $\cF_1$ and $\cF_2$ two transverse foliations with Gromov hyperbolic leaves intersecting in a one dimensional foliation $\cG$. We will assume that $\wcG$ has Hausdorff leaf space (in particular, for every $L \in \wcF_i$ we have that $\cG_L$ has Hausdorff leaf space, c.f. Proposition \ref{prop-Hsdff2Dand3D}). 
 
The main result of this section is the following: 

\begin{prop}\label{p.novisualallbubble}
Assume that $\wcG$ has Hausdorff leaf space,
and that $\wt{\cG}$ does not have the small visual measure property in $\wcF_1$. Then for every $L \in \wcF_1$ we have that there is a point $\xi_L \in S^1(L)$ so that every leaf of $\cG_L$ has both rays landing in $\xi_L$. Moreover, up to collapsing some foliated products of $\cF_1$ and keeping the foliation transverse to $\cF_2$ we get that $\cF_1$ must be the weak stable foliation of  a topological $\RR$-covered Anosov flow. 
\end{prop}

The proof will proceed in three stages. In \S\ref{ss.minimalcaseB} we find a sublamination $\cL$ of $\cF_1$ where every leaf is a {\em {bubble leaf}}: i.e. a leaf $L \in \wcF_1$ so that for every $\ell \in \cG_L$ the two landing points of the rays in
$\ell$ are the same point in $S^1(L)$ (this notion is used in \cite{FP4}).
This gives the first statement of Proposition 
\ref{p.novisualallbubble} for this sublamination (and covers the case where $\cF_1$ is minimal). We extend the property to the whole foliation in \S \ref{ss.extensionfullfol} after explaining the notion of collapse appearing in the statement (this is done in \S\ref{ss.collapse}). Then in \S\ref{ss.RcoveredAF} we complete the proof of Proposition \ref{p.novisualallbubble}. 

By Section \ref{s.topologytub} the
topology of $\mt \cup S^1_\infty{\wcF_1}$ is well defined,
irrespective of the metric in $M$. This will be implicitly used
throughout this section.

\subsection{A minimal lamination with the desired property}\label{ss.minimalcaseB}

The goal of this subsection is to show an intermediate result, Lemma \ref{lem-lam}, that completes the proof of the first part of Proposition \ref{p.novisualallbubble} assuming that $\cF_1$ is minimal.  

We first need the following (compare with \cite[Lemma 2.5]{FP4}):

\begin{lema}\label{lem-nonvisualgoestoinf}
Under the assumptions of Proposition \ref{p.novisualallbubble} there is a sequence of points $y_n \in \mt$ so that $y_n \to y_\infty$ and such that there are arcs $c_n$ of  leaves of $\wcG$ contained in $L_n = \wcF_1(y_n)$ such that $c_n \cap B_{L_n}(y_n, n) = \emptyset$ and such that the visual measure from $y_n$ of $S^1(L_n) \setminus \mathrm{Sh}_{y_n}(c_n, L_n)$ is smaller than $1/n$. 
\end{lema}

\begin{proof}
By definition, if the foliation $\cG$ does not have the small visual measure property on $\cF_1$ we know that there is some $\eps_0>0$, sequences of points $x_n \in L_n \in \wcF_1$ and segments $c_n \in \cG_{L_n}$ such that $d_{L_n}(x_n, c_n) > 2n$ and such that $I_n =\mathrm{Sh}_{x_n}(c_n, L_n) \subset S^1(L_n)$ has visual measure larger than $\eps_0$. Note that without loss of generality, we can assume that the endpoints of $c_n$ which we call $w_n,z_n$ project from the $x_n$ to the endpoints of $I_n$. 

Since geodesic segments that are far away from a point
have uniformly small visual measure from that point, we know that the distance from a minimizing geodesic arc $s_n$ from  $w_n$ to $z_n$ in $L_n$
 to $x_n$ is uniformly bounded above.

\begin{figure}[ht]
\begin{center}
\includegraphics[scale=0.72]{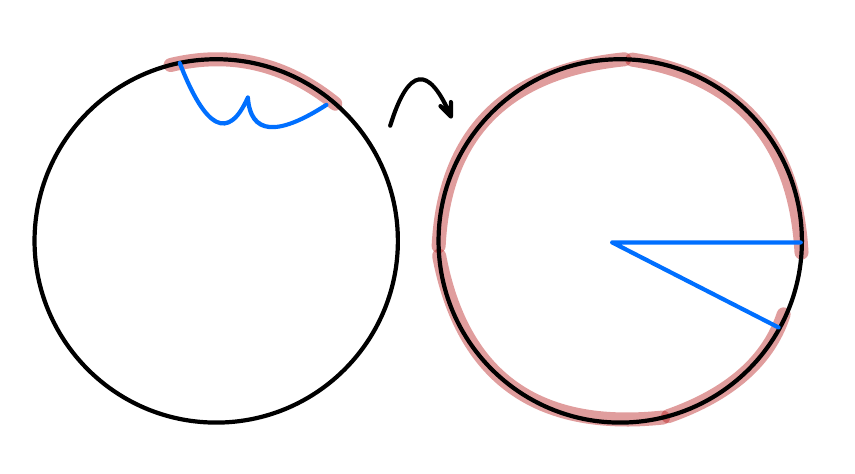}
\begin{picture}(0,0)
\put(-98, 84){$\gamma_n y_n$}
\put(-216,132){$y_n$}
\put(-163,143){$\gamma_n$}
\end{picture}
\end{center}
\vspace{-0.5cm}
\caption{{\small Sending deep points to a fixed fundamental domain provides big visual measure.}}\label{fig-zoomin}
\end{figure}

Fix a geodesic ray $r_n:[0,\infty) \to L_n$ from $x_n$ landing in the midpoint $\xi_n$ of the segment $I_n$. It follows that the distance from $y_n= r_n(n)$ to $s_n$ goes to infinity with $n$. In particular, we know that the visual measure from $y_n$ to the complement of $I_n$ goes to $0$ with $n$. Also, $B_{L_n}(y_n, n) \cap c_n = \emptyset$. One can take deck transformations to assume that all $y_n$ belong to a given compact fundamental domain of $M$ in $\mt$ (see figure \ref{fig-zoomin}). Thus, taking subsequences if necessary, we obtain the result. 
\end{proof}

We will now show an application of Proposition \ref{prop-pushing} that we will need to use more than once. In some sense, what this lemma says is that if $c_n$ is a sequence of arcs of $\cG_{L_n}$ so that $c_n$ converges to some interval $I \subset S^1(L_\infty)$ in the topology of $\mt \cup \cTuu$ (c.f. \S \ref{s.topologytub}) where $L_\infty$ is a leaf in the limit of $L_n$, then we have that applying Proposition \ref{prop-pushing} to the arcs $c_n$ we obtain arcs $\hat c_n$ of leaves of $\cG_{L_\infty}$ such that $\hat c_n \to I$ in $L_\infty \cup S^1(L_\infty)$. 

\begin{lema}\label{lema-pushing} 
Let $\eps>0$ and $y_n \to y_\infty$ in $\mt$ and denote $L_n = \wcF_1(y_n)$ and $L_\infty = \wcF_1(y_\infty)$ the leaves through those points.  Let $c_n$ be arcs of curves in $\cG_{L_n}$ so that the complement of the shadow of $c_n$ from $y_n$ has visual measure smaller than $\eps$ and such that $d_{L_n}(y_n, c_n) \to \infty$. Then, there are arcs $\hat c_n$ of curves in $\cG_{L_\infty}$ which converge in $L_\infty \cup S^1(L_\infty)$ to an interval $I_\infty \subset S^1(L_\infty)$ so that the visual measure of $S^1(L_\infty) \setminus I_\infty$ from $y_\infty$ is smaller than $3\eps$. 
\end{lema}

\begin{proof} We will use the markers introduced in \S \ref{ss.markers}.
From $y_\infty$ we can produce a finite number of markers $m_1, \ldots, m_k : [0,1] \times \RR_{\geq 0} \to \mt$ so that the rays  $\hat m_i= m_i(\{0\} \times \RR_{\geq 0})$ land at points $\xi_i \in S^1(L_\infty)$ satisfying that the visual measure from $y_\infty$  of the interval made by two consecutive ones is smaller than $\eps$. The markers are considered in the direction where the leaves $L_n$ accumulate $L_\infty$ (which up to subsequence we can assume are all in the same side).  We will denote by $\hat m_i^t$ to the rays $m_i(\{t\} \times \RR_{\geq 0})$. Since there are finitely many markers, we can parametrize the interval so that the leaves $L_t$ containing $m_i(t,0)$ with $1\leq i\leq k$ coincide. One can always  change
the initial transversal of the marker (that is, the  set
$m_i([0,1] \times \{ 0 \}$) because what matters is the asymptotic
behavior of nearby leaves with respect to that ray in  $L_\infty$. 
Hence, up to reducing the  sizes, we can assume without loss of generality that $m_i(0,0) = y_\infty$ and that $m_i(t,0)=m_j(t,0)$ for all $i\neq j$.  We denote by $t_n \in (0,1]$ the parameter such that $L_{t_n}=L_n$ (clearly, $t_n \to 0$). 

Given $i \neq j$ we denote by $W^{ij}_t$ the wedge between $\hat m_i^t$ and $\hat m_j^t$ whose closure in $L_t \cup S^1(L_t)$ contains the interval $J^{ij}_t$ between the limit points of $\hat m_i^t$ and $\hat m_j^t$ which is oriented in the same way as the interval between $\xi_i$ and $\xi_j$ which has smallest visual measure seen from $y_\infty$. 
(Note that we will only consider the case where $\xi_i$ and $\xi_j$ are very close in $S^1(L_\infty)$ so there will be no ambiguity.)

Up to considering subsequences, and relabeling we can assume that the points $\xi_1$ and $\xi_2$ define an interval $J$ whose visual measure from $y_\infty$ is smaller than $2\eps$ and verifies that

\begin{itemize}
\item $c_n$ has a subsegment $s_n$ with endpoints in $\hat m_i^{t_n}$ and is completely contained in $L_n \setminus (W^{12}_{t_n} \cup B(y_n,n))$.  
\end{itemize} 

If we denote $E_n \in \wcF_2$ so that $c_n \subset L_n \cap E_n$, we can apply Proposition \ref{prop-pushing} to the endpoints of $s_n$ pushing them along the corresponding markers to obtain arcs $\hat c_n$ of $\cG_{L_\infty}$
with $\hat c_n \subset L_\infty \cap E_n$. The $\hat c_n$ connect points $p_n, q_n$ in $\hat m_1$ and $\hat m_2$ respectively and avoid the wedge $W^{12}_0$ between $\hat m_1$ and $\hat m_2$. Note that $p_n \to \xi_1$ and $q_n \to \xi_2$ so we can assume that the curves $\hat c_n$ are nested in the sense that to connect $y_\infty$ to $\hat c_{n+1}$ in the complement of $W^{12}_0$ we must intersect $\hat c_n$. We want to show that given a ball $B$ around $y_\infty$, there is $n_0$ so that $\hat c_{n_0}$ (and therefore every $\hat c_n$ with $n\geq n_0$)  is outside $B$. 

We assume by contradiction that $\hat c_n$ intersects $B$ for all $n$. Thus, the sequence of leaves through $\hat c_n$ limits in some leaf $\ell \in \cG_{L_\infty}$ 
which intersects $B$ and has both landing points in the complement of $J^{12}_0$. 
 
 Now, let $E_\infty \in \wcF_2$ so that $\ell = E_\infty \cap L_\infty$. It follows that $E_\infty$ does not intersect the wedge $W^{12}_0$ and so it cannot intersect $W^{12}_t$ for small $t$. Applying Proposition \ref{prop-pushing} we obtain, for large $n$ leaves $\ell_n \in L_n$ landing outside the interval $J^{12}_{t_n}$, thus forbidding the curves $c_n$ to remain far from $y_n$. This is a contradiction and finishes the proof of the lemma. 
\end{proof}
Now we can show the main result of this subsection. 

\begin{lema}\label{lem-lam}
Under the assumptions of Proposition \ref{p.novisualallbubble} there is a closed $\pi_1(M)$-invariant subset $\cL$ of leaves of $\wcF_1$ with the property that for every $L \in \cL$ there is a point $\xi_L \in S^1(L)$ so that every leaf of $\cG_L$ has both rays landing in $\xi_L$ (i.e. every leaf is a bubble leaf).  
\end{lema}

\begin{proof}
Using Lemma \ref{lem-nonvisualgoestoinf} we get a sequence $y_n \to y_\infty$ in $\mt$ of points with the property that if $L_n = \wcF_1(y_n)$ is the leaf of $\wcF_1$ through $y_n$ there are arcs $c_n$ of leaves of $\cG_{L_n}$ contained in $L_n$ such that $d_{L_n}(y_n, c_n) >n$ and such that the visual measure of the complement of the shadow of $c_n$ from $y_n$ is less than $1/n$. 

Consider first $L_\infty = \wcF_1(y_\infty)$. 

\begin{claim}\label{claim84}
If $\ell \in \cG_{L_\infty}$ then both rays of $\ell$ must land in a unique point. 
\end{claim}
\begin{proof}
Fix an arbitrary $\eps>0$ and we show that the landing points of $\ell$ must be contained in an interval of $S^1(L_{\infty})$ of visual measure less than $\eps$ from $y_\infty$. 

Apply Lemma \ref{lema-pushing} to obtain arcs $\hat c_n$ in $L_\infty$ which converge in $L_\infty \cup S^1(L_\infty)$ to a segment $I_\infty$ in $S^1(L_\infty)$ so that the visual measure of $S^1(L_\infty) \setminus I_\infty$ from $y_\infty$ is smaller than $\eps$. Assume that a ray $\ell$ lands in a point $\xi$ the interior of $I_\infty$, then, we can choose a neighborhood $U$ of $\xi$ in $L_\infty \cup S^1(L_\infty)$ which verifies that $U \cap S^1(L_\infty)$ is contained in the interior of $I_\infty$. Now, it follows that $\ell$ can be cut so that it is fully contained in $U$ and joins a point $z \in U \cap L_\infty$ to $\xi$. However, the fact that $\hat c_n$ converges to $I_\infty$ implies that it contains arcs separating $\xi$ from $z$ in $U$ a contradiction. 
\end{proof}

Since $\eps$ was arbitrary, we deduce that there is at most one point, $\xi_{L_\infty}$ where such leaves can land.

Now, for every $\gamma \in \pi_1(M)$ then $\gamma L_\infty$ has the same property (and one gets that $\gamma \xi_{L_\infty}= \xi_{\gamma L_\infty}$). Moreover, if $\mathcal{L}$ is the set of leaves $L$ which verify that every ray of $\cG_L$ lands in a unique point $\xi_L$ then this set is closed in the leaf space. Indeed, given $L_n \in \mathcal{L}$ with $L_n \to L_\infty$, choose $y_n \in L_n$ so that $y_n \to y_\infty$. It follows that for every $n$ the leaf $L_n$ contains an arc $c_n$ of a leaf of $\cG_{L_n}$ which verifies that $d_{L_n}(y_n, c_n) >n$ and such that the visual measure of the complement of the shadow of $c_n$ from $y_n$ is smaller than $1/n$. So, we can apply the same argument to show that $L_\infty \in \mathcal{L}$. 
\end{proof}

\begin{coro}\label{coro-stabilizer}
For every leaf $L \in \cL$ we have that the stabilizer is trivial or cyclic. 
\end{coro}

\begin{proof}
Given $\gamma \in \pi_1(M)$, if $\gamma L = L$, then $\gamma$ must fix $\xi_L$ in  $S^1(L)$.
Let $\gamma, \zeta$ non trivial elements in the stabilizer of $L$,
and let $\mu_\gamma, \mu_\zeta$ be axes for $\gamma, \zeta$ 
in $L$ respectively. 
The axes exist  by Lemma \ref{lem-closedgeodesics}.  Then
$\mu_\gamma, \mu_\zeta$ both limit in $\xi_L$, and hence
have rays $r_1, r_2$ which are a bounded
distance from  each other in $L$. Project to $L/ <\gamma>$, which
is an annulus.  \ $r_1$
projects to a closed curve $\nu$, and $r_2$ projects
to a curve
a bounded distance from $\nu$ in $L/<\gamma>$. The projection of $r_2$
is also an embedded curve, and hence the projection of $r_2$
limits to a closed curve in $L/<\gamma>$. Since the projection
of $\mu_\zeta$ to $M$ is a closed curve, it now follows that
$\mu_\zeta$ projects to a closed curve in $L/<\gamma>$.
Therefore $\gamma, \zeta$ share  both fixed points in $S^1(L)$,
they admit a common axis, and hence they are in a cyclic group. 
\end{proof}
%

We end by showing an important property of the points $\xi_L$.
We note that this uses the results of Section   \S \ref{s.topologytub}
and specifically Remark \ref{rem-topology} and the preceeding
paragraph  discussing the topology of $\mt \cup S^1_\infty(\wcF_1)$.

\begin{lema}\label{lem-cont1}
The point $\xi_L$ varies continuously in $\cTuu$ with respect to the leaf $L \in \cL$. 
\end{lema}

\begin{proof} 
Consider $L_n \in \cL$ a sequence of leaves so that $L_n \to L$ and consider $\xi_n := \xi_{L_n} \in S^1(L_n) \subset \cTuu$. We want to show that $\xi_n \to \xi_L$ in $\cTuu$. Note that it is enough to consider the case where $\xi_n$
converges to a point $\xi_\infty$, and show that $\xi_\infty = \xi_L$, because else we consider converging subsequences. 

Suppose by contradiction that $\xi_\infty \neq \xi_L$. 
Using Lemma \ref{lema-pushing} as in the proof of Lemma \ref{lem-lam} we can 
do the following: choose $\eps > 0$ so that the neighborhoods of $\xi_\infty$
and $\xi_L$ of radius $2\eps$ in $S^1(L)$ have disjoint closures. Let $J$ be the neighborhood of radius $\eps/2$ of $\xi_\infty$ in $S^1(L)$,
and let  $I$  be the closure of the  complement of
the  neighbhorhood of radius $\eps/2$  of $\xi_\infty$. Notice
that $J$ is  contained in the interior of  $I$.
 Then
we can find a neighborhood $\cU$ of $L$ so that for every leaf $L_n \in \cU \cap \cL$ there are sequences of arcs of leaves of $\cG_{L_n}$ converging to an interval $I_{L_n}$ with $\xi_{L_n}$ in the interior of $S^1(L_n) \setminus I_{L_n}$ and 
$S^1(L_n) \setminus I_{L_n}$ has length less than $2 \eps$.
We can arrange that $I_{L_n}$ converges to $I$ in $S^1(L)$
when $L_n$ in $\cU \cap \cL$ converges to $L$. In addition $\xi_L$ is
contained in the interior of $I$.
This uses that $\xi_\infty \not = \xi_L$.
Exactly as in Claim \ref{claim84} 
one obtains that $\xi_L$ cannot be the landing point of a
ray of any leaf of $\cG_L$. This contradiction proves the lemma.
\end{proof}

\subsection{Collapsing}\label{ss.collapse} 

We start by defining precisely what we mean by \emph{collapsing foliated products} of $\cF_1$. 
First, a foliated product of $\cF_1$ is an $\cF_1$-saturated set $B$ which
is topologically a product (that is boundary leaf times $[0,1]$). Note that leaves of a foliation may be just immersed and not be embedded (in particular, boundary leaves of a foliated product may not be embedded). Therefore, to define precisely what we mean we need to actually consider $\tilde B$ to be a connected component of the lift of $B$ to $\mt$ which by definition is precisely invariant (deck translates of $\tilde B$ are disjoint from $\tilde B$ or coincide with $\tilde B$).
In addition we ask that $\tilde B$ is homeomorphic to $\RR^2 \times [0,1]$ with a homeomorphism sending $\RR^2 \times \{t\}$ to leaves of $\wcF_1$. Since there is a transverse foliation $\cF_2$ we will moreover ask that every leaf $E \in \wcF_2$ intersecting $\tilde B$ intersects it in a set homeomorphic to $\RR \times [0,1]$ with the homeomorphism sending sets of the form $\RR \times \{t\}$ to intersection between $E$ and leaves of $\wcF_1$ (that is, leaves of $\wcG$). 

The collapsing operation collapses the product leaf times
$[0,1]$ to a single leaf. The foliation $\cF_2$ in the
product is collapsed to a foliation in the collapsed leaf, and thus, the foliation $\cF_2$ also descends to the collapsed quotient and is still transverse to the new foliation. 

\begin{remark}
We note that the collapsing is a \emph{monotone} map (i.e. it collapses cellular sets, that is, sets which are decreasing intersections of balls, in this case, intervals). In particular the topology of $M$ does not change after this procedure \cite{FenleyRcov}. It is also true that after collapsing, the new foliations induced in the quotient preserves the property of being or not being $\RR$-covered (note that the fact of being minimal or not can change, indeed, one of the reasons to collapse is to try to make the foliation minimal). 
\end{remark}

To be able to collapse, we will need the following result from \cite[Proposition 2.6]{FenleyRcov}: 

\begin{prop}\label{prop-precollapse} 
Let $\cF$ be an $\RR$-covered foliation without compact leaves. Then, it has a unique minimal saturated set whose complement is a union of $I$-bundles over non-compact surfaces and the foliation can be collapsed to a minimal foliation by collapsing each complementary region to a single leaf. 
\end{prop}

Note that our foliation $\cF_1$ does not have compact leaves (because it admits a one dimensional subfoliation and its leaves are Gromov hyperbolic in the universal cover).  We will show next that $\cF$ is $\RR$-covered. This will be done after understanding more carefully the complementary regions to $\cL$ the sublamination produced in the previous section. 

\subsection{Extending to the whole foliation}\label{ss.extensionfullfol}
We begin by studying the foliation $\cF_1$ and the 
complementary regions of $\cL$. The goal is to show that these complementary regions are $I$-bundles,
and that $\xi_L$ is in some sense constant in each one. This will then allow to show that $\cF_1$ is $\RR$-covered. 

Let $\pi(\cL)$  be the projection of $\cL$ to $M$ by the universal covering projection $\pi: \mt \to M$. Corollary
\ref{coro-stabilizer} shows that every leaf of $\pi(\cL)$ is either a plane or an annulus. Let $\cV$ be the closure of a connected component of $\mt \setminus \cL$ (it is a foliated region), note that leaves in $\pi(\partial \cV)$ are also planes and annuli since they are contained in $\pi(\cL)$.

To $\pi(\cV)$ we can apply an  the octopus decomposition with respect to $\cF_2$ (see \cite[Proposition I.5.2.14]{CandelConlon}) to get $\pi(\cV) = K \cup A$ where $K$ is a compact set, and $A$ are the arms. The relation with $\cF_2$ is given by the fact that one can fix some $\eps_0>0$ and consider $A$ so that every point in the boundary of $A$ is contained in a foliated box of size less than $\eps_0$ which also intersects the other boundary component, in particular, leaves of $\cF_2$ are a product in $A$ (in other words one can think of the $I$-bundle decomposition in $A$ as made by arcs contained in $\cF_2$-leaves). This means the following:  one can choose coordinates  so that in each foliated box one has that $\cF_1$ leaves restricted to $A$ are horizontals and $\cF_2$ leaves restricted to $A$ are verticals. Vertical means union
of $I$-fibers, and horizontal means transverse to
the $I$-fibers.  

Let $L$ be a boundary leaf of $\cV$, and
let $F = \pi(L)$. We have two cases:

\begin{lema}\label{lem-ibundle1}
If $F$ is a plane, then $\cV \cong L \times [0,1]$ and
the foliation $\wcF_2$ restricted to $\cV$ is a product. In particular, 
$\pi(\cV)$ can be collapsed with respect to $\cF_2$.
\end{lema}

\begin{proof}
In this case, 
by an appropriate choice of $K$, we can assume that 
the intersection $F \cap K$ lifts to a compact disk $D \subset L$.
Let $\widetilde  K$ be the lift of $K$ to $\mt$ contained in $\cV$,
so that $F \cap K$ lifts to $D$..
Then $\widetilde K$ is homeomorphic to a closed disk times $I$ and we write
$\partial \widetilde K = D \cup C \cup D'$,
where $D  \subset L$, $D' \subset L'$ and $C$ is a compact  
annulus so that its interior is  contained in the interior of $\cV$.

We now claim that for every $\ell \in \cG_L$, so that $\ell = L \cap E$ with $E \in \wcF_2$, we have that $E \cap \cV$ is homeomorphic to $\ell \times I$.  For this, note first that every leaf $\ell \in \cG_L$ must intersect $D$ in a compact set, outside of which is fully contained in $A$ where the product behavior is part of the definition. Now, each compact arc of intersection of $\ell$  with $D$, has its endpoints in $A$ that  can be `pushed' along the $\wcF_2$ foliation to $L'$, and then Proposition \ref{prop-pushing} completes the claim. This ends the proof of the lemma.  
\end{proof}

The case that $F$ is an annulus is somewhat more complicated since it is harder to show that the rays escape the compact part. 

\begin{lema}\label{lem-ibundle2}
If $F$ is an annulus, then 
the foliation $\wcF_2$ restricted to $\cV$ is a product. 
\end{lema}

\begin{proof}
We keep the notation from the previous lemma. 
Let $\cC = K \cap F$.
In this case, again by choosing $K$ appropriately
we can assume that $\cC$ is a compact annulus with boundary components $c_1,c_2$.
The set $\cC$  lifts to a band $\cB$ inside $L$ which is bounded by quasigeodesics $g_1,g_2$ lifting $c_1,c_2$. Denote by $\gamma$ a generator of the stabilizer of $L$. 
Denote by $G_1$ and $G_2$ the connected components of $L \setminus \cB$ whose boundaries are $g_1$ and $g_2$ respectively.

We first remark that given $\ell \in \cG_L$ it follows that both rays land in $\xi_L$ because $L \in \cL$. Moreover, if we fix an orientation in $L$ and $\cG_L$ we can denote $P_\ell$ to be the connected component of $L \setminus \ell$ in the positive direction of $\ell$ (chosen so that the closure of $P_\ell$ in $L \cup S^1(L)$ is $P_\ell \cup \ell \cup \xi_L$). Up to changing $\gamma$ for its inverse we can assume that $\gamma P_\ell  \subset P_\ell$ and that $\bigcup_{n \leq 0} \gamma^n P_\ell = L$. 

We want to show that given a leaf $\ell = L \cap E$ with $E \in \wcF_2$ then, one ray is eventually contained in $G_1$ and the other eventually contained in $G_2$. First, we show that a ray $r$ of $\ell$ cannot  intersect $\cB$ indefinitely. For this, assume that there is a sequence $x_n \in r$ going to $\infty$ so that $x_n \in \cB$. Up to composing with $\gamma^{i_n}$ for appropriate $i_n$, and taking subsequences, we get that $\gamma^{i_n} x_n \to y_\infty$. It follows that $\gamma^{i_n} \ell \to \ell_\infty$ which is the leaf of $\cG_L$ through $y_\infty$ (this limit is unique because the leaf space of $\cG_L$ is Hausdorff). Now, it follows that $\gamma \ell_{\infty} = \ell_{\infty}$ because all the leaves $\gamma^{i_n} \ell$ are nested, so we can assume that they are increasingly converging to $\ell_\infty$. But this is a contradiction, since then $\ell_{\infty}$ would have two distinct landing points in $S^1(L)$ contradicting that $L \in \cL$ (where all rays land in $\xi_L$). 

Now we need to show that it cannot be that both rays are contained in $G_1$ (or $G_2$). If that were the case, then $\ell$ as well as all iterates $\gamma^n \ell$, would be eventually contained in $G_1 \cup \cB$ (except for some compact intervals of bounded length) contradicting that $\bigcup_{n\leq 0} \gamma^n P_\ell = L$. 

Let $\ell_i$ the ray of $\ell$ contained in $G_i$. 
The initial point  of $\pi(\ell_i)$ is contained in 
an annulus $S_i$,  which is a component of $K \cap A$. 
This annulus has one boundary component in $F$ and another 
boundary component in a leaf $F_i$ of $\cF_1$.
Lift the annulus to $\mt$, so that starting point lifts to a point
in $\ell_i$ and  let $L_i$ be the corresponding lift of $F_i$.
Since $\ell = L \cap E$, we get that close to $\ell_i$ the leaf $E$ intersects $L_i$. Note that both $L_1,L_2$ are contained in $\partial \cV$.

We claim that $L_1 = L_2$. Denote by $\hat \ell_i = E \cap L_i$. Since $\cG_E$ is Hausdorff we have that either $\ell$ separates $\hat \ell_1$ from $\hat \ell_2$ or there is a transversal in $E$ from $\hat \ell_1$ to $\hat \ell_2$ disjoint from $\ell$. In the first case, we get that $L$ separates $L_1$ from $L_2$ and in the second we get that either $L_1$ separates $L$ from $L_2$ or $L_2$ separates $L$ from $L_1$. All these posibilities contradict the fact that the three leaves are in the boundary of $\cV$, giving a contradiction. We let $L' = L_1= L_2$. 

As in the previous lemma, apply Proposition \ref{prop-pushing} to show that we can push two disjoint rays in every leaf $\ell \in \cG_L$ to $L'$.
Then the compact interval in between the rays must belong then to the same leaf, showing that $\wcF_2$ is trivially foliated in $\cV$.  
This finishes the proof.
\end{proof}

We can now complete the proof of the first part of Proposition \ref{p.novisualallbubble}: 

\begin{lema}\label{lem-continuous}
For every $L \in \wcF_1$ there is a point $\xi_L \in S^1(L)$ so that for every leaf $\ell \in \cG_L$ both rays land in $\xi_L$ (that is, every leaf of $\cG$ is a bubble leaf in its corresponding leaf of $\wcF_1$). Moreover, the point $\xi_L$ varies continuously with $L$ (in the leaf space of $\wcF_1$).
\end{lema}

\begin{proof}
This follows from the fact that if $\cV$ is a complementary region of $\cL$ and $L, L'$ are its boundary leaves, we have shown that leaves in $\cG_L$ push to leaves of $\cG_{L'}$ entirely. So, the same proof as in Lemma \ref{lem-cont1} applies. 
\end{proof}

Moreover, we are in conditions to prove: 

\begin{lema}\label{lem-anosovrcovered}
The foliation $\cF_1$ is $\RR$-covered.
\end{lema}

\begin{proof}
We assume by contradiction that $\cF_1$ is not $\RR$-covered, 
 and thus there are distinct leaves $L, L' \in \wcF_1$ which are non-separated in the leaf space of $\wcF_1$, that is, there is a sequence of leaves $L_n \in \wcF_1$ with points $x_n,y_n \in L_n$ so that $x_n \to x_\infty \in L$ and $y_n \to y_\infty \in L'$. 

We want to show that there is a leaf $E \in \wcF_2$ which intersects some $L_n$ as well as $L$ and $L'$ which will contradict the fact that the foliation $\cG_E$ is Hausdorff. 

To do this, we use that there is a dense set of marker directions 
in  $L$ and $L'$ on the side  the leaves $L_n$ are limiting on. 

Let $m_1,m_2: [0,1] \times \RR_{\geq 0} \to \mt$ and $m_1', m_2': [0,1] \times \RR_{\geq 0} \to \mt$ distinct markers of $L$ and $L'$ respectively with the property that $m_i(0,0)=x_\infty$ and $m_i'(0,0)=y_\infty$. We can assume that if $n_0$ is large, for every $n>n_0$ we have that $x_n \in m_i(t_n,0)$ and $y_n \in m_i'(t_n',0)$. 

We can chooose the markers to be distinct, and to land at different points in $S^1(L), S^1(L')$. We can assume without loss of generality that we have $m_1(t_n \times \RR_{\geq 0})$ and $m_1'(t_n' \times \RR_{\geq 0})$ do not land in $\xi_{L_n}$. 


Lemma \ref{lem-continuous} implies that every leaf of $\cG_{L_n}$ verifies that both rays land in $\xi_{L_n}$.
We can find a leaf $c \in \cG_{L_n}$  which intersects both the
marker $m_1$ with $L$ and the marker $m'_1$ with $L'$.  
If $c = E \cap L_n$ we deduce that $E$ must then intersect $L$ and $L'$,
because  the `vertical' length of  the  markers is less than
a  foliation box size of $\cF_2$.
Denote by $c_1 = E \cap L$ and $c_2 = E \cap L'$. Since $\cG_E$ is Hausdorff, there is a transversal to $\cG_E$ intersecting $c_1$ and $c_2$. This produces a transversal to $\wcF_1$ 
which intersects $L$ and $L'$ contradicting the fact that these leaves are non-separated. This completes the proof. 
\end{proof}

We can now apply Proposition \ref{prop-precollapse}: There is a set closed $\wcF_1$-saturated and $\pi_1(M)$-invariant set $\cL$  so that $\pi(\cL)$ 
is the unique minimal set of $\cF_1$. 
Moreover, the complementary regions of $\pi(\cL)$ are $I$-bundles over non  compact surfaces,
and the foliation $\cF_1$ can be collapsed to a minimal
foliation by collapsing each complementary region to a single leaf. In this collapsed foliation, every leaf is a bubble leaf. So, from now on, we can assume that our foliation is minimal and $\RR$-covered. See Lemma \ref{lema-minimal} below for a precise statement.

\subsection{Constructing the Anosov flow}\label{ss.RcoveredAF}

To recap, we state the following result that we obtained in the
previous subsection:

\begin{lema}\label{lema-minimal}
Under the assumptions of Proposition \ref{p.novisualallbubble}, then, one can collapse $\cF_1$ to a minimal $\RR$-covered foliation $\cM_1$ which is still transverse to $\cF_2$ and such that the new intersected  foliation $\cG'$ verifies that $\wt{\cG'}$ does not have the small visual measure property
in $\wt{\cM_1}$. 
\end{lema}

Now, we show that $\cM_1$ is the weak stable foliation of a (topological) Anosov flow on $M$ (cf. \S \ref{ss.Anosovflows}). 

\begin{lema}\label{lem-anosovflow}
There is a (topological) Anosov flow $\Phi$ on $M$ for which $\cM_1$ is the weak stable foliation. 
\end{lema}

\begin{proof}
This follows exactly the proof of \cite[Theorem 5.5.8]{Calegari} (see also \cite{Calegari-book} and \cite{BFP} for some discussion). One considers a Candel metric on leaves of $\cM_1$ and fixes, for each $L \in \wt{\cM_1}$ a foliation by lines which is a \emph{geodesic fan} towards the point $\xi_L \in S^1(L)$. Since the point $\xi_L$ varies continuously with the leaf (Lemma \ref{lem-continuous}) we deduce that by considering the vector field tangent to each geodesic with unit size and toward $\xi_L$, we get a vector field which projects to a vector field $X$ in $M$ (because the points $\xi_L$ are equivariant as well as the Candel metric). \ Calegari  proved that $X$ generates an expansive flow $\phi$, we just  review a couple of
steps in the proof of
\cite[Theorem 5.5.8]{Calegari}: \ 1) Along leaves of $\cM_1$ flow lines
diverge backwards. 
\  2) Show that transversely to  $\cM_1$ the orbits diverge
in the forward direction. Hence either forward or backward two orbits
eventually diverge from each other, so the flow  is expansive.
To show 2), Calegari uses a marker $m: [0,1] \times \RR \to \mt$,
so that $m(\{ 0 \} \times \RR)$ projects to a closed curve
in a leaf. He shows that one can choose this marker strictly
decreasing to $0$ in thickness in the $\RR$  direction,
as opposed to not increasing. Each leaf of  $\cM_1$ is dense,
so a lift  intersects this marker. He shows that the contracting 
direction  in the marker is the direction opposite to the flow $\widetilde X$.
Hence the positive direction of $X$  expands transversely to 
$\widetilde{\cM_1}$. This is the main idea to prove 2)
and obtain expansivity of $\phi$.

Hence $\phi$ is expansive, and it preserves a foliation  ($\cM_1$).
It follows that $\phi$ is a topological Anosov flow (see \cite[Theorem 5.9]{BFP}). 
\end{proof}

Finally we show the following result which completes the proof of Proposition \ref{p.novisualallbubble}. Note that by definition, a (topological)  Anosov flow is $\RR$-covered if its weak stable foliation (and thus also its weak unstable foliation) is $\RR$-covered. We refer the reader to \cite{Barbot,FenleyAnosov} for background.

\begin{remark}
At this point we have shown that if $\cF_1$ and $\cF_2$ are two transverse foliations so that $\wcG= \wcF_1 \cap \wcF_2$ has Hausdorff leaf space when lifted to the universal cover and we know that $\cF_1$ is not $\RR$-covered, then, for every $L \in \wcF_1$ we know that $\cG_L$ is by quasigeodesics. In the next section we will extend this further to show that the only obstruction is given by (modifications) of the example from \S \ref{s.example}. 
\end{remark}

\section{Failure of small visual measure}\label{s.SVM}

This section is divided in two parts. On the one hand, we will show that if one of the foliations has the small visual measure property, then, both must have it, and thus we deduce that the foliation $\cG$ is leafwise quasigeodesic getting the conclusion of Theorem \ref{teo.main}. As an ingredient for this, we show that if $\cG$ is leafwise quasigeodesic in one of the foliations, then it must have a closed leaf. The second part of the section concludes the proof of Theorem \ref{teo.main}. 

\subsection{Closed leaves and small visual measure in both foliations} 

We first show the following statement which implies Corollary \ref{quasigeodesiccoro-periodic}: 

\begin{teo}\label{prop-closedleaf}
Let $\cF_1$ and $\cF_2$ be two transverse foliations by Gromov hyperbolic leaves intersecting in $\cG$ which verifies that $\wcG$ has Hausdorff leaf space (cf. Proposition \ref{prop-Hsdff2Dand3D}). If $\cG$ has the small visual measure property in $\cF_1$, then, there is a closed leaf of $\cG$. 
\end{teo}

\begin{proof}
Proposition \ref{prop.svmimpliesqg} gives that for every $L \in \wcF_1$ we have that $\cG_L$ is by uniform quasigeodesics. Using \cite[Proposition 6.9]{BFP} we know that there is a closed saturated $\pi_1(M)$-invariant sublamination $\Lambda$ of $\wcF_1$ so that for every $L \in \Lambda$, we have that $\cG_L$ is a \emph{weak-quasigeodesic fan}, meaning that there is a unique  point $\xi_L \in S^1(L)$ so that every leaf of $\cG_L$ has one ray landing on $\xi_L$ and such that for every $\xi \in S^1(L) \setminus \{\xi_L\}$ there is at least one leaf of $\cG_L$ with a ray landing in $\xi$. 

Using Theorem \ref{teo.nonholonomy} we can find a leaf $L \in \Lambda$ for which there is $\gamma \in \pi_1(M) \setminus \{\mathrm{id}\}$ so that $\gamma L = L$. It follows that $\gamma$ acting on $S^1(L)$ must fix both $\xi_L$ and another point $\xi^-$ (cf. Lemma \ref{lem-closedgeodesics}). The set of quasigeodesics of $\cG_L$ from $\xi^{-}$ to $\xi_L$ is a closed  interval in the leaf space $\cG_L$, thus, the boundaries must be fixed by $\gamma$ and project to closed leaves of $\cG$. 
\end{proof}

\begin{remark}\label{rem-manyperiodic}
Note that the proof implies that for every $L \in \wcF_1$ so that there is some $\gamma \in \pi_1(M)\setminus \{\mathrm{id}\}$ with $\gamma(L) = L$, then
 there is some $\ell \in \cG_L$ so that $\gamma \ell = \ell$. Note that under some conditions it is known that foliations by Gromov hyperbolic leaves admit several deck transformations with fixed leaves (see e.g. \cite[Proposition 3.3]{ABMP}). Also, admitting a leafwise quasigeodesic subfoliation can provide some extra information that could allow to improve the conclusion of Corollary \ref{quasigeodesiccoro-periodic} to obtain infinitely many closed leaves even if they may not be blow ups of Anosov foliations (see \cite{ChandaFenley} for more discussion).  
\end{remark}

As a consequence of this and Proposition \ref{p.novisualallbubble} we obtain: 

\begin{coro}\label{coro.nomixed}
Let $\cF_1$ and $\cF_2$ be two transverse foliations by Gromov hyperbolic leaves intersecting in $\cG$ which verifies that $\wcG$ has Hausdorff leaf space. If $\cG$ has the small visual measure property in $\cF_1$ then it also has the small visual measure property in $\cF_2$. In particular, we deduce that the leaves of $\cG_L$ are uniform quasigeodesics in $L$ for all $L \in \wcF_1, \wcF_2$. 
\end{coro}

\begin{proof}
Note that if $\cG$ does not have the small visual measure property in $\cF_2$ it follows that for every $E \in \wcF_2$ we have that all leaves of $\cG_E$ land in the same point $\xi_E \in S^1(E)$ (cf. Lemma \ref{lem-continuous}). Consider $c \in \wcG$ such that $c = \gamma c$ for some $\gamma \in \pi_1(M) \setminus \{\mathrm{id}\}$ provided by Theorem \ref{prop-closedleaf}. Them $c = L \cap E$ with $L \in \wcF_1$ and $E \in \wcF_2$. 

It follows that $c$ has two distinct limit points in $S^1(E)$ (cf. Lemma \ref{lem-closedgeodesics}) and this contradicts Lemma \ref{lem-continuous}. 

The uniform quasigeodesic property now follows from Proposition \ref{prop.svmimpliesqg}. 
\end{proof}

\subsection{Non solvable fundamental group}\label{ss.nonsolvable}

Let $\cF_1$ and $\cF_2$ be two transverse foliations on a closed 3-manifold $M$ with Gromov hyperbolic leaves and so that the leaf space of $\wt{\cG}$ is Hausdorff ($\cG = \cF_1 \cap \cF_2$). We will assume that the intersected  foliation $\cG$ fails the small visual measure property in both. We want to show that this implies that $M$ has solvable fundamental group, and that up to collapsing, $\cF_1$ and $\cF_2$ are topologically equivalent to the weak stable foliations of a suspension Anosov flow, as the example in \S \ref{s.example}. 

Using Lemma \ref{lema-minimal} (see \S\ref{ss.collapse}) we can collapse $\cF_1$ and $\cF_2$ to two transverse minimal foliations $\cM_1$ and $\cM_2$ which are (by Lemma \ref{lem-anosovflow}) the weak stable foliations of (topological) Anosov flows $\Phi_1$ and $\Phi_2$ respectively. These flows are constructed so that in an arbitrary leaf $L$ of $\wt{\cM_1}$ (resp. $\wt{\cM_2}$) orbits are quasigeodesics in $L$
pointing towards (that is the forward flow direction)
the point $\xi_L$ given by Lemma \ref{lem-continuous}. Lemma \ref{lem-anosovrcovered} implies that both $\Phi_1$ and $\Phi_2$ are $\RR$-covered. 

As explained in \S \ref{ss.Anosovflows} we know that if the flows $\Phi_1$ and $\Phi_2$ are not orbitally equivalent to suspension Anosov flows, we know that they must be skewed-$\RR$-covered. In particular, the minimal foliations $\cM_1$ and $\cM_2$ are uniform and correspond to the weak stable foliations of $\Phi_1$ and $\Phi_2$. (Recall \S \ref{ss.uniformfol} for definition of uniform and uniformly equivalent foliations, and \S \ref{ss.Anosovflows} for some properties of skewed $\RR$-covered Anosov flows that we will use.) 

The idea is to use the fact that inside each leaf of (say) the foliation $\wt{\cM_1}$, every leaf of $\cG_L$ has both rays which
converge to the same point at infinity in the universal circle of $\cM_1$.
 Then switch foliations and obtain a contradiction to  the fact (see \S \ref{ss.Anosovflows}) that for skewed-$\RR$-covered Anosov flows, the endpoint $\xi_L$ must vary in a monotonous way with $L$. To make sense of this and be able to switch foliations, we first need to know that $\cM_1$ and $\cM_2$ are uniformly equivalent. 
Hence we
first show:

\begin{lema}\label{lem-oneintersection}
Let $E \in \wt{\cM_2}$ and let $\gamma \in \pi_1(M) \setminus \{\mathrm{id}\}$, then $E$ can intersect at most one leaf of $\wt{\cM_1}$ fixed by $\gamma$. 
\end{lema}

\begin{proof}
Since $E$ is the lift of a leaf of a Reebless foliation in $M$ it separates $\mt$ in two connected components, which according to its orientation we denote by 

$$\mt \setminus E = \cE^+ \cup \cE^-.$$

We choose $\cE^+$ so that for an arbitrary $L \in \wt{\cF}_1$ intersecting
$E$ we have that if $\ell = L \cap E$ then $\cE^+ \cap L$ is the connected component of $L \setminus \ell$ whose closure in $L \cup S^1(L)$ only contains $\xi_L$ in $S^1(L)$. We first show that this does not depend on the choice of $L$ by continuity of $\xi_L$, and the push through property, Proposition \ref{prop-pushing}. 
In other words the components of $L \setminus \ell$ limiting only
in $\xi_L$ vary continuously with $L$. 
To prove this: note that for any closed 
segment $I$ in $S^1(L) \setminus \xi_L$,
then $I$ is the limit of a interval family of arc segments $\tau_t$ of 
$\cG_L$. Using the denseness of markers and the push through property,
this can be pushed to nearby leaves $L'$. Hence for nearby $L'$ to $L$, 
the connected component of $L' \setminus (E \cap L')$ which only
limits in $\xi_{L'}$ has to be disjoint from the push through 
family of arcs of $\cG_{L'}$ obtained from $\tau_t$ pushed through 
to $L'$.
This shows the independence of choice of $L$ as claimed above.
See figure \ref{fig-horo}.

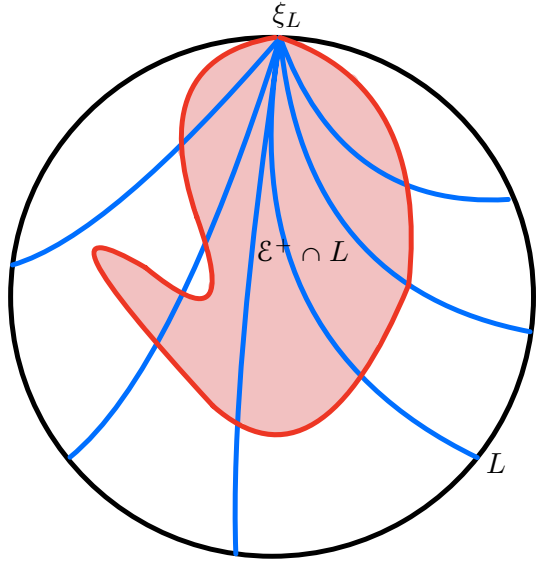
\begin{figure}[ht]
\begin{center}
\includegraphics[scale=0.92]{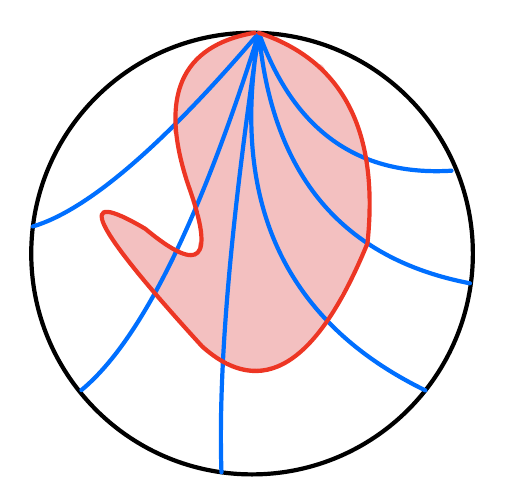}
\begin{picture}(0,0)
\put(-40,40){$L$}
\put(-120,210){$\xi_L$}
\put(-126,120){$\cE^+ \cap L$}
\end{picture}
\end{center}
\vspace{-0.5cm}
\caption{{\small Depiction of $\cE^+ \cap L$ in the shaded region. Note that the flow lines $\wt{\Phi_1}$ (depicted as a quasigeodesic fan) need not be transverse to the boundary.}}\label{fig-horo}
\end{figure}

If $\wt{\Phi_1}$ were transverse to  $\wt{\cM_2}$ then for every $x \in E$ we would have that $\wt{\Phi_1}^t(x) \in \cE^+$ for $t>0$ and $\wt{\Phi_1}^t(x) \in \cE^-$ for $t<0$. This may not be true in general, still we do get that for every $t<0$ of sufficiently large modulus, we have that $\wt{\Phi_1^t}(x) \in \cE^{-}$; this is because the negative ray of a flow line lands in a point different from $\xi_L$, so the ray eventually escapes $\cE^{+}$.

Let $L \in \wt{\cM_1}$ be a leaf fixed by $\gamma$ so that $E \cap L \neq \emptyset$ (if $\gamma$ does fix any leaf of $\wt{\cM_1}$ the lemma is vacuously true). Then, $L$ contains an orbit $o_L$ of $\wt{\Phi_1}$ which is invariant under $\gamma$, i.e. on which $\gamma|_{o_L} : o_L \to o_L$ is a translation. Up to taking $\gamma^{-1}$ we can assume that $\gamma$ moves $o_L$ forward with respect to the orientation of the orbits of the flow $\wt{\Phi_1}$.  

The leaf $E$ intersects $L$. We claim that $E$
intersects $o_L$. 
In fact  it must intersect every orbit of $\wt{\Phi_1}$ in $L$, but 
we will not prove that. Suppose that $E$ does not intersect $o_L$.
The set of leaves of $\cG_L$ intersecting $o_L$ is $\gamma$ invariant
(and connected),
and so is the set of leaves not intersecting $o_L$. By
assumption both are non empty, so there is a leaf in the boundary
of one separating it from the other one. This leaf is $\gamma$
invariant. Therefore cannot limit on $\xi_L$ in both directions,
contradiction.

Thus, we get that $\gamma^{-k} E$ must intersect $o_L$ in its negative orientation, so, by our choice of orientation we deduce that $\gamma^{-k} E$ intersects $\cE^-$ for all $k > k_0$. Since $\cM_2$ is a foliation, this implies that $\gamma^{-k} E \subset \cE^-$ for all $k > k_0$. Since the action of $\gamma$ on the leaf space of $\cG_L$ in $L$ is a translation, this gives that $\gamma E \subset \cE^+$.  

By connectedness and the properties of skewed-$\RR$-covered Anosov flows, if $E$ intersects another fixed leaf $L'$ of $\wt{\cM_1}$ which is fixed by $\gamma$ it must also intersect an adjacent one, in particular, a leaf $L''$ which is fixed by $\gamma$ and whose orbit $o_{L''}$ fixed by $\gamma$ is translated backwards by $\gamma$ (cf. Proposition \ref{p.skewedAnosov}).  This would imply that $\gamma E$ must be contained in $\cE^-$ a contradiction. This contradiction implies that $E$ can intersect at most one fixed leaf by $\gamma$ in $\wt{\cM_1}$ as stated. 
\end{proof}


\begin{lema}\label{l.uniformequiv}
The foliations $\cM_1$ and $\cM_2$ are uniformly equivalent. 
\end{lema}

\begin{proof}
Since both $\cM_1$ and $\cM_2$ are uniform, to prove that $\cM_1$ and $\cM_2$ are uniformly equivalent it is enough to show that for every leaf $E \in \wt{\cM_2}$ there is $L \in \wt{\cM_1}$ so that $E$ is contained in a bounded neighborhood of $L$ and for every $ L \in \wt{\cM_1}$ there is some leaf $E \in \wt{\cM_2}$ such that $L$ is contained in a bounded neighborhood of $E$. 

We first fix a deck transformation $\gamma_1 \in \pi_1(M)\setminus \{ \mathrm{id}\}$ which has a fixed leaf in $\wt{\cM_1}$. Note that the Hausdorff distance between two consecutive leaves of $\wt{\cM_1}$ fixed by $\gamma_1$ is uniformly bounded above by some constant $K_1$. Given a leaf $E \in \wt{\cM_2}$ we know that it can intersect at most one fixed leaf by $\gamma_1$ thanks to Lemma \ref{lem-oneintersection}.  This implies that $E$ must be contained in the $2K_1$ neighborhood of some fixed leaf of $\wt{\cM_1}$. The argument is symmetric, so we can find $\gamma_2$ and $K_2$ to obtain that every leaf $L \in \wt{\cM_1}$ is contained in a $2K_2$-neighborhood of any leaf of $\wt{\cM_2}$ that it intersects.  
This completes the proof. 
\end{proof}

\begin{remark}
In fact, one can push this argument to show that the set of deck transformations fixing leaves of $\wt{\cM_1}$ coincides with those fixing leaves of $\wt{\cM_2}$ and thus, by applying the main result of \cite{BM} it follows that the flows $\Phi_1$ and $\Phi_2$ are orbitally  equivalent by an orbit equivalence homotopic to identity. We will not use this fact. 
\end{remark}

\noindent
{\bf {Completion of the proof of Theorem \ref{teo.main}}} $-$
The following lemma will give a contradiction with the properties of skewed-$\RR$-covered Anosov flows and completes the proof that if $\cG$ fails the small visual measure property in both foliations, then the foliations (after collapsing) are topologically equivalent to the weak stable foliations of a suspension Anosov flow as in the example from \S \ref{s.example}. 
In other words if $\pi_1(M)$ is not solvable, then
(under the Hausdorff hypothesis for $\wcG$), it follows that
$\cG$ has the small visual measure property in one of 
$\cF_1$ or $\cF_2$. Corollary \ref{coro.nomixed}
then implies that $\cG$ has the small visual measure property
in both $\cF_1$ and $\cF_2$. Proposition \ref{prop.svmimpliesqg}
then implies that $\cG$ is uniformly leafwise quasigeodesic
in both $\cF_1$ and $\cF_2$. 
This will then complete the proof of Theorem \ref{teo.main}. 

\begin{lema}
The point $\xi_L$ which sends $L \in \wt{\cM_1}$ to the point $\xi_L \in S^1(L)$ on which all rays of the intersected  foliation land is constant in the universal circle $S^1_u(\cM_1)$. 
\end{lema}

Since $\cM_1$ is $\RR$-covered, the identification of the circles $S^1(L)$ with $S^1_u(\cM_1)$ is as explained in \S \ref{ss.uniformfol}. The local constancy of the point $\xi_L$ given by Lemma \ref{lem-continuous} contradicts a property of skewed-$\RR$-covered Anosov flows (cf. Proposition \ref{prop.nonmarkermoves}). 

\begin{proof}
Since $\cM_1$ and $\cM_2$ are uniformly equivalent, one can identify $S^1_u(\cM_1)$ and $S^1_u(\cM_2)$ as in \S \ref{ss.uniformfol} (for each leaf $L \in \wt{\cM_1}$ and $E \in \wt{\cM_2}$ there is a map $f_{L,E}: L \to E$ which maps each point in $L$ to a closest point in $E$ and is a coarsely well defined quasi-isometry and thus extends to a well defined map from $S^1(L)$ to $S^1(E)$ and allows to identify the universal circles). 

Now, fix a leaf $L \in \wt{\cM_1}$ and consider $E \in \wt{\cM_2}$ so that $L \cap E \neq \emptyset$. It follows that if $\ell = L \cap E$ is the unique leaf of $\wt{\cG}$ in the intersection, it follows that there is $\xi \in S^1_u(\cM_1) \cong S^1_u(\cM_2)$ (under the identification) so that both rays of $\ell$ converge to $\xi$ in $S^1(L)$ and $S^1(E)$ respectively (again, after the identification with the corresponding universal circles). 

Consider a transversal $\tau: (-\eps, \eps) \to E$ to $\wt{\cM_1}$ (with $\tau(0) \in \ell \subset L$) and denote by $L_t$ the leaf of $\wt{\cM_1}$ through the point $\tau(t)$. Let $\ell_t = L_t \cap E$. Lemma \ref{lem-continuous} implies that both rays of $\ell_t$ converge in $E \cup S^1(E)$ to $\xi$. This implies that in $L_t$ all rays converge to $\xi$ too. This implies that the map from $L$ to the endpoint of all rays in $L$ is locally constant, 
when thought of as a map from the leaf space of $\wt{\cM_1}$ to the
universal circle of $\cM_1$. Therefore this map is constant, completing the proof.
\end{proof}

\section{Parabolic leaves}\label{s.nonhyperbolicleaves}

In this section we discuss the assumption of having Gromov hyperbolic leaves. 

\subsection{Minimal case}

Here we show: 

\begin{teo}\label{teo.mainparabolic}
Let $\cF_1$ and $\cF_2$ be two transverse foliations so that $\cF_1$ is minimal, and  has a leaf which is not Gromov hyperbolic. Assume moreover that the leaf space of the intersected  foliation $\cG$ in the universal cover is Hausdorff. Then, the foliation $\wcG$ is \emph{leafwise quasigeodesic}. 
\end{teo}

We first use the following result from \cite[\S 5]{FP} (recall our standing assumption on orientability): 

\begin{prop}\label{prop-nonGHimplieslinear}
Let $\cF$ be a minimal foliation containing a leaf which is not Gromov hyperbolic. Then, either $\cF$ is uniformly equivalent to a (linear) irrational foliation by planes in $\TT^3$ or it is uniformly equivalent to a (linear) irrational foliation by cylinders in a nilmanifold $N$ (which could be $\TT^3$). In particular, no leaf of $\cF$ is Gromov hyperbolic. 
\end{prop}

Note that in \cite[Theorem 5.1]{FP} the existence of a holonomy invariant measure is assumed, but in our situation the existence of such measure follows from the existence of a leaf which is not Gromov hyperbolic (see \cite[Chapter 7]{Calegari-book}). 

In particular, as a consequence of the proposition, we know
that $M$ is a nilmanifold.

First, we will assume that $\cF_1, \cF_2$ are
uniformly equivalent to linear foliations, and in 
addition that the linear foliations which are uniformly equivalent to $\cF_1$ and $\cF_2$ are not the same. In this case, we can prove the conclusion of Theorem \ref{teo.mainparabolic} without requiring minimality of the foliations, nor that the equivalent foliations are irrational (we refer the reader to \cite[Appendix B]{HP} for a description of foliations in nilmanifolds, including their lineal models). 

\begin{lema}\label{lem.diferentequiv}
Let $M$ be a nilmanifold (possibly $\TT^3$) and let $\cF_1$ and $\cF_2$ be two foliations which are uniformly equivalent to different linear foliations and such that the intersected  foliation $\cG$ has Hausdorff leaf space in the universal cover. Then, the foliation $\cG_L$ is uniformly equivalent to a linear foliation in every $L \in \wcF_i$.  
\end{lema}

 Note also that since leaves $L$ are quasi-isometric to euclidean planes, being uniformly equivalent to linear foliations is a stronger property than being quasi-isometric to a linear foliation. 

\begin{proof}
In the universal cover, we know that the intersection between a leaf $L \in \wcF_1$ and a leaf $E \in \wcF_2$ takes place in a neighborhood of the intersection $c= P_1 \cap P_2$ between two linear planes: $P_1$ which is Hausdorff close to $L$ and $P_2$ which is Hausdorff close to $E$. 

Fix a leaf $L$ of $\wcF_1$.
For any leaf $E$ of $\wcF_2$, let $\ell_E = L \cap E$. By Hausdorffness of $\cG_L$ 
this has at most one component. By uniform equivalence with 
different linear foliations, this is always non empty.
Hence the leaf space of $\wcF_2$ is naturally  homeomorphic to the leaf 
space  of $\cG_L$ by $\ell_E = L \cap E$.
Let $\cI$ be the leaf space of $\cG_L$, which is homeomorphic to
the reals.

We want to show that for any $\ell_E$, some uniform neighborhood of
$\ell_E$ in $\mt$ contains a translate of $c$.  If this were not the case, we would have one $\ell_{E_0} \subset L$ which is close to only one ray of a translate of $c$. 
We will show that this contradicts that $\cG_L$ has Hausdorff leaf
space. 

Since $\ell_{E_0}$ is close to only one ray of a translate of $c$
it follows that there is a unique complementary component
of $\ell_{E_0}$ in $L$ which is contained in a bounded neighborhood
of $\ell_{E_0}$
in $L$. Let $B$ be this complementary component. Let $x$ be a point in
$\ell_{E_0}$. The leaf $\ell_{E_0}$ separates the leaf space
of $\cG_L$ and so there is one component $A$ of 
$\cI \ \setminus \ \{ \ell_{E_0} \}$
so that every leaf of $\cG_L$ in $A$ is contained in $B$ and
so it is a finite Hausdorff distance from $\ell_{E_0}$.

Let $P_1^x, P_2^x$ be the two linear planes parallel to $P_1, P_2$
respectively and both containing $x$. Now take a bi-infinite curve
$\eta$ in $L$ which is a finite Hausdorff distance from $P^x_1
\cap P^x_2$ and we assume that $\eta$ contains $x$.
We consider the set $C$ of all leaves of $\cG_L$ intersecting
$\eta$ plus all leaves in $A$. Both sets are connected,
and both contain the leaf of $\cG_L$ through $x$. 
Hence the set $C$ is an interval
in $\cI$. It contains a ray in $\cI$ as $A$ is a ray in $\cI$.
Note that $\ell_{E_0}$ is contained in a bounded neighborhood 
of $\eta$ in
$L$.

For any point $y$ in $\eta$, then $y$ is uniformly close to $P_1^x \cap P^x_2$,
hence its $\cG_L$ leaf is boundedly close to $P^x_1 \cap P^x_2$,
so boundedly close in $L$ to $\eta$. It follows that any
leaf in $C$ is contained in a bounded neighborhood of $\eta$ in $L$.
However $L \setminus C$ contains two disjoint half planes. 
$H_1, H_2$. Any leaf intersecting $H_1$ cannot connect to
a leaf intersecting $H_2$ without intersecting $\eta$, hence
without intersecting $C$. But the leaves intersecting $C$ 
form an interval which is a ray, so the complement is also an interval,
hence connected. This contradicts the above.
This finishes the proof.
\end{proof}

Next we show that if of one of the foliations is minimal the other cannot be uniformly equivalent to it. This follows arguments similar to \cite{Pot,HP}. 

\begin{lema}\label{lem-nonunifequiv}
Let $\cF_1$ be a minimal foliation that has a leaf which is not Gromov hyperbolic and let $\cF_2$ be a transverse foliation to $\cF_1$. Then, $M$ is a nilmanifold (possibly $\TT^3$) and $\cF_2$ is not uniformly equivalent to $\cF_1$. 
\end{lema}

\begin{proof}
A minimal foliation like $\cF_1$ is, due to Proposition \ref{prop-nonGHimplieslinear}, uniformly equivalent to a linear irrational foliations by planes on $\TT^3$ or a linear irrational foliation by cylinders in a nilmanifold (possibly $\TT^3$). Note that such foliations do not have holonomy, in particular, one can apply  \cite[Theorem VIII.2.2.1]{HectorHirsch} to deduce that every one dimensional foliation $\cT$ transverse to $\cF_1$ must have, in $\mt$ global product structure with $\wt{\cF_1}$, meaning that for every $\ell \in \wt{\cT}$ and $L \in \wcF_1$ the intersection $\ell \cap L$ is non empty and contains at least one point. 

In particular, since for every $R>0$ there are leaves $L,L'$ so that $L'$ does not intersect the $R$-neighbood of $L$ we deduce that $\cT$ cannot be contained in a foliation which is uniformly equivalent to $\cF_1$ since otherwise leaves of $\cT$ could not intersect every leaf of $\wcF_1$ (compare with \cite[Proposition 6.8]{Pot}).


Now fix a vector field tangent to $T\cF_2$ which is everywhere transverse to $T \cF_1$ (note that this can be done by the orientability assumptions, for instance, by considering a metric and choosing the orthogonal vector of norm $1$ to $T\cF_2 \cap T\cF_1$ in $T\cF_2$). We can assume that the vector field integrates to a one dimensional foliation (if needed, by smoothing along the leaves of $\cF_2$) and thus we can apply the discussion of the previous paragraph to complete the proof. 

\end{proof}

\begin{remark}\label{remark-lemauniform}
Foliations in nilmanifolds without Reeb components are well understood (see e.g. \cite[Appendix B]{HP}) and are either uniformly equivalent to linear foliations or contain some torus leaf. Note that if a foliation in a nilmanifold is minimal, then it has to be uniformly equivalent to a linear foliation which has some irrational direction. In conclusion, in the setting of the previous lemma, if $\cF_2$ has a torus leaf, it cannot be at bounded distance from the plane that directs the minimal foliation $\cF_1$.  
\end{remark}

Now we can complete the proof of Theorem \ref{teo.mainparabolic}. 

\begin{proof}[Proof of Theorem \ref{teo.mainparabolic}]
First, Lemma \ref{lem-nonunifequiv} applied to $\cF_1$ implies
that $M$ is a nilmanifold and $\cF_1$ is uniformly equivalent
to a linear foliation in $M$.
To prove the Theorem we just need to prove that the foliation $\cF_2$ cannot be uniformly equivalent to a foliation containing either: 
1) A Reeb component, or 2) A foliated $\TT^2 \times [0,1]$ which has
non Hausdorff leaf space in its universal cover.
The reason is that if we prove that, 
then the classification of foliations in nilmanifolds up to uniform equivalence (see e.g. \cite[Appendix B]{HP}) implies that
$\cF_2$ is uniformly equivalent to a linear foliation.
Lemma \ref{lem-nonunifequiv} implies that $\cF_1, \cF_2$ are 
not uniformly equivalent, hence they are individually uniformly equivalent
to different linear foliations. With the hypothesis that $\wcG$ has
Hausdorff leaf space, this satisfies the hypothesis
of Lemma \ref{lem.diferentequiv}, which then
implies that $\cG_L$ is uniformly equivalent to a linear
foliation in every $L \in \wcF_i$, and so completes the proof
of Theorem \ref{teo.mainparabolic}. 

The fact that $\cF_2$ does not have Reeb components follows from the transversality of the foliations $\cF_1, \cF_2$ and the fact that $\cF_1$ is uniformly equivalent to a linear foliation. Indeed, suppose that $\cF_2$ has a
Reeb component  $R$ and let $\hat R$ be a connected component of its lift to $\mt$. The inclusion $i: R \to M$ induces a map $i_\ast: \pi_1(R) \to \pi_1(M)$. If this map is zero, then $\hat R$ is a (compact) solid torus, else, it is an infinitely long tube in a bounded radius neighborhood of the lift of the closed geodesic representing a generator of the image of $i_\ast$.  In the first case, it follows that  the boundary of $\hat R$, which is a leaf of $\wcF_2$ must have a tangency with some leaf of $\wcF_1$ because $\wcF_1$ is uniformly equivalent to a linear foliation. 
Suppose now that $\hat R$ is an infinite tube. Then its boundary
is an infinite cylinder, which we denote by $L$. It is a leaf of 
$\wcF_2$. If the intersection of $\wcF_1$ with $L$ has a compact leaf,
then this compact leaf bounds a disk in its $\wcF_1$ leaf. The
disk has to be contained in $\hat R$ and this forces a tangency
with $\wcF_2$, impossible. Otherwise all leaves of $\wcF_1 \cap L$
are lines. Let $g$ be such a line and $F$ the $\wcF_1$ leaf
containing it. Looking at the intersection of $F$ with $\hat R$ near
$g$ we have a leaf $\ell$ of $\wcG$ contained in the interior of $\hat R$.
This leaf $\ell$ is contained in a planar leaf $Z$ of $\wcF_2$, since
$R$ is a Reeb component. In addition, again because $R$
is a Reeb component, then both ends of $\ell$ escape in $Z$ and 
get closer and closer to $L$, and finally the ends of $\ell$
escape in $\hat R$ towards the
same end of $\hat R$. Therefore rays of $\ell$ are asymptotic
to rays of leaves of $\wcF_1 \cap L$. The Hausdorff hypothesis
of leaf space of $\wcG$ 
implies that these two leaves of $\wcF_1 \cap L$ are the same leaf.
This is the main property, it implies that this leaf of $\cG_L$
is in fact $g$. But this implies that there is an end of $L$ so
that every leaf of $\wcF_1 \cap L$ has both rays escaping
to this end of $L$. This is impossible.
This shows that $\cF_2$ does not have Reeb components.

To complete the proof we must rule out that $\cF_2$ has a foliated $\TT^2 \times [0,1]$ so that when lifted to $\mt$ we have that the boundary  tori lift to planes $L_1$ and $L_2$ which are not separated in the leaf space of $\wcF_2$.
Let $Z$ be a leaf of $\wcF_1$ intersecting $L_1$.
Suppose that it intersects $L_2$, let $\ell_1, \ell_2$ be components
of the intersection. Since the leaf space of $\cG_Z$ is Hausdorff
there is a transversal to $\cG_Z$ connecting $\ell_1$ to $\ell_2$.
This transversal is also a transversal to $\wcF_2$. But it
connects $L_1$ and $L_2$, contradiction. Hence any such $Z$
does not intersect $L_2$. The union $U$ of the leaves of $\wcF_1$ 
intersecting $L_1$ is open and it is invariant under 
$G = \pi_1(\pi(L_1))$ (here $\pi(L_1)$ is a torus).
Hence there is a unique leaf $B$ of $\wcF_1$ in the boundary
of $U$ which separates it from $L_2$. This leaf $B$ is also invariant
under $G$ because $L_2$ is. It follows that $\pi(B)$ is a compact
leaf of $\cF_1$, contradiction to $\cF_1$ being minimal.
This completes the proof of Theorem \ref{teo.mainparabolic}.
\end{proof}
%

\subsection{Further questions} 

There are foliations which are not by Gromov hyperbolic leaves which admit transverse foliations.

A Reeb surface is (under our orientability assumptions) a foliated annulus in a leaf of $\cF_1$ so that when lifted to the universal cover, the leaves in the interior accumulate in both boundary components of the lifted band. See \cite{FP4} for more information.

\begin{example}
Consider a suspension Anosov flow and do a $DA$-modification at some periodic orbits obtaining some transverse tori to the flow $T_1, \ldots, T_k$ some attracting and some repelling. See for instance \cite{BBY}. One can start with the foliations by fibers transverse to the suspension and drill in the direction of the tori in order to keep the flow transverse to the foliation, but now the tori $T_i$ become leaves of the foliation $\cF$ we have constructed and which is transverse to the flow. Now, one can glue such piece to other pieces to obtain an Anosov flow in a closed 3-manifold, and create gluing foliations like $\cF$ to obtain a foliation with torus leaves which is transverse to the Anosov flow (and therefore to its weak stable and weak unstable foliations). This way one can produce Reebless or taut foliations transverse to foliations by Gromov hyperbolic leaves (note that in \cite{BBY} they produce transitive Anosov flows with this setting and that would allow us to produce transversals intersecting every leaf of $\cF$, thus it is possible to make the foliations taut). \end{example}

This proposes the following question:

\begin{question}
Is it possible to construct two transverse foliations $\cF_1$ and $\cF_2$ one of which does not have Gromov hyperbolic leaves and the other does, in such a way that the intersected  foliation is leafwise Hausdorff?
\end{question}

We mention two other open questions that we found relevant. 

\begin{question}
Let $\cF_1$ and $\cF_2$ be two transverse foliations by Gromov hyperbolic leaves. Is it true that if the leaf space of the intersected  foliation $\cG$ in the universal cover is not Hausdorff then there are Reeb surfaces in $\cF_1$ and $\cF_2$? 
\end{question}

\begin{question} Let $\cF_1$ and $\cF_2$ be orientable transverse foliations
which are minimal in an atoroidal closed 3-manifold $M$.
Is it true that $\cG$ is topologically equivalent to the 
flow foliation of an Anosov flow?
\end{question}

\section{Application to partially hyperbolic diffeomorphisms}
\label{s.ph}

Let $f$ be a partially hyperbolic diffeomorphism in a closed
$3$-manifold $M$. We refer to \cite{BFP} for the definition of
partial hyperbolicity, as well as branching foliations, and 
different forms of collapsed Anosov flow behavior. 

The center leaf space is defined as follows: let $E$ be a center stable
leaf in $\mt$ and $F$ a center unstable leaf in $\mt$.
A connected component of $E \cap F$ is called a {\em center leaf} of $f$.
The center leaf space has a very natural topology \cite{BFP} 
which makes it into a simply connected two dimensional,
possibly non Hausdorff manifold.

\begin{teo} Let $f$ be a partially hyperbolic diffeomorphism
in a closed 3-manifold $M$, with $\pi_1(M)$ not
virtually solvable, and such that $f$ admits branching foliations,
center stable ($\cs$) and center unstable ($\cu$), both of
which have Gromov hyperbolic leaves and are orientable.
Suppose that the center leaf space of $f$ is Hausdorff.
Then $f$ is a collapsed Anosov flow. \end{teo}

\begin{proof}
Under the orientability conditions
Burago and Ivanov \cite{BI} showed that $\cs, \cu$ are approximated
by actual foliations $\cF^{cs}_\eps, \cF^{cu}_\eps$, whose
tangent planes are $\eps$ near those of $\cs$ and $\cu$ and
are transverse to each other.

Let $\cG = \cF^{cs}_\eps \cap \cF^{cu}_\eps$. 
Then the leaf space of $\wcG$ is naturally homeomorphic to the center leaf
space of $f$. By assumption the center leaf space
is Hausdorff, so the leaf space of $\wcG$ is Hausdorff as well. In addition
leaves of $\widetilde{\cF^{cs}_\eps},  
\widetilde{\cF^{cu}_\eps}$ are Gromov hyperbolic.  
By Theorem \ref{teo.main} it follows that leaves of $\wcG$
are uniformly quasigeodesic in leaves of $\widetilde{\cF^{cs}_\eps}$
and $\widetilde{\cF^{cu}_\eps}$.
This implies that center leaves are uniform quasigeodesics
in leaves of $\wcs, \wcu$.

When $\cs, \cu$ are transversely orientable, then
\cite[Theorem D]{BFP} shows that $f$ is also 
a strong collapsed Anosov flow as desired.
\end{proof}

Applications to ergodicity can be found in \cite{FPAcc}.

\end{document}